\documentclass[reqno]{amsart}
\usepackage{standalone}
\usepackage[utf8]{inputenc}
\usepackage{import}
\usepackage{subcaption} 
\usepackage{caption}
\usepackage[table]{xcolor}
\usepackage{tabu}
\usepackage{blkarray}

\usepackage{calc}
\usepackage{pgfplots}
    \pgfplotsset{
        compat=1.14,
        width=8cm,
    }

\usepackage[thinlines]{easytable}
\usepackage{amssymb, amsmath, amsthm, enumerate,  mathtools, color}
\usepackage{mathabx} 
\usepackage{esint} 
\usepackage{bm}
\usepackage{tikz}
\usepackage{tikz-3dplot}
\usetikzlibrary{calc,patterns,angles,quotes}

\numberwithin{equation}{section}
\newtheorem{theorem}{Theorem}[section]
\newtheorem{corollary}[theorem]{Corollary}
\newtheorem{lemma}[theorem]{Lemma}
\newtheorem{proposition}[theorem]{Proposition}

\newtheorem{conjecture}[theorem]{Conjecture}

\newtheorem*{claim}{Claim}

\newtheorem*{claim 1}{Claim 1}
\newtheorem*{claim 2}{Claim 2}

\theoremstyle{definition}
\newtheorem{definition}[theorem]{Definition}
\newtheorem{remark}[theorem]{Remark}

\newtheorem{example}[theorem]{Example}

\newtheorem*{acknowledgment}{Acknowledgment}

\newcommand{\C}{\mathbb{C}}
\newcommand{\R}{\mathbb{R}}
\newcommand{\Z}{\mathbb{Z}}
\newcommand{\N}{\mathbb{N}}

\newcommand{\T}{\mathbb{T}}

\newcommand{\fD}{\mathfrak{D}}
\newcommand{\ud}{\mathrm{d}}

\def\inn#1#2{\langle#1,#2\rangle}

\newcommand{\supp}{\mathrm{supp}\,}

\renewcommand{\angle}{\measuredangle}

\makeatletter
\@namedef{subjclassname@2020}{%
  \textup{2020} Mathematics Subject Classification}
\makeatother

\begin{document}

\title[Oscillatory integral operators of arbitrary signature]{Sharp $L^p$ estimates for oscillatory integral operators of arbitrary signature}

\author[J. Hickman]{Jonathan Hickman}
\address{School of Mathematics, The University of Edinburgh, Edinburgh EH9 3JZ, UK}
\email{jonathan.hickman@ed.ac.uk}

\author[M. Iliopoulou]{Marina Iliopoulou}
\address{School of Mathematics, Statistics and Actuarial Science, University of Kent, Canterbury CT2 7PE, UK.}
\email{m.iliopoulou@kent.ac.uk}


\maketitle


\begin{abstract} The sharp range of $L^p$-estimates for the class of H\"ormander-type oscillatory integral operators is established in all dimensions under a general signature assumption on the phase. This simultaneously generalises earlier work of the authors and Guth, which treats the maximal signature case, and also work of Stein and Bourgain--Guth, which treats the minimal signature case. 
\end{abstract}

\section{Introduction} 




\subsection{Main results} This article concerns $L^p$ bounds for oscillatory integral operators that are natural variable coefficient generalisations of the Fourier extension operator associated to surfaces of non-vanishing Gaussian curvature. To describe the basic setup, for $d \geq 1$ let $B^d$ denote the unit ball in $\R^d$ and fix a dimension $n \geq 2$. Suppose $a \in C^{\infty}_c(\R^n \times \R^{n-1})$ is supported in $B^n \times B^{n-1}$ and consider a smooth function $\phi \colon B^n \times B^{n-1} \to \R$ which satisfies the following conditions:
\begin{itemize}
\item[H1)] $\mathrm{rank}\, \partial_{\omega x}^2 \phi(x;\omega) = n-1$ for all $(x;\omega) \in B^n \times B^{n-1}$.
\item[H2)] Defining the map $G \colon B^n \times B^{n-1} \to S^{n-1}$ by $G(x;\omega) := \frac{G_0(x;\omega)}{|G_0(x;\omega) |}$ where
\begin{equation*}
G_0(x;\omega) := \bigwedge_{j=1}^{n-1} \partial_{\omega_j} \partial_x\phi(x;\omega),
\end{equation*}
the curvature condition
\begin{equation}\label{non vanishing curvature}
\det \partial^2_{\omega \omega} \langle \partial_x\phi(x;\omega),G(x; \omega_0)\rangle|_{\omega = \omega_0} \neq 0
\end{equation}
holds for all $(x; \omega_0) \in \mathrm{supp}\,a$.
\end{itemize}
For any $\lambda > 1$ let $a^{\lambda}(x;\omega) := a(x/\lambda;\omega)$ and $\phi^{\lambda}(x;\omega) := \lambda\phi(x/\lambda;\omega)$ and define the operator $T^{\lambda}$ by
\begin{equation}\label{Hormander operator}
T^{\lambda}f(x) :=  \int_{B^{n-1}} e^{2 \pi i \phi^{\lambda}(x; \omega)} a^{\lambda}(x;\omega) f(\omega)\,\ud \omega
\end{equation}
for all integrable $f \colon B^{n-1} \to \C$. In this case $T^{\lambda}$ is said to be a \emph{H\"ormander-type operator}. 

A prototypical example is given by the choice of phase 
\begin{equation*}
\phi_{\mathrm{ell}}(x;\omega) := \langle x',\omega \rangle +  x_n \cdot  \frac{1}{2}|\omega|^2, \qquad x = (x',x_n) \in \R^{n-1} \times \R;
\end{equation*}
in this case \eqref{Hormander operator} is the well-known extension operator $E_{\mathrm{ell}}$ associated to the elliptic paraboloid (with the additional cutoff function $a^{\lambda}$ localising the operator to a spatial ball of radius $\lambda$): see Example~\ref{prototypical example} below.

Operators of the form \eqref{Hormander operator} were introduced by H\"ormander \cite{Hormander1973} as a simultaneous generalisation of Fourier extension operators and operators which arise in the Carleson--Sj\"olin approach to the study of Bochner--Riesz means \cite{Carleson1972}. The $L^p$ theory of H\"ormander-type operators has been investigated in a number of articles over the last few decades: see, for instance, \cite{Hormander1973, Stein1986, Bourgain1991, Bourgain1995, Minicozzi1997, Wisewell2005, Bennett2006, Lee2006, Bourgain2011, Bennett2014, GHI2019} and references therein. A recent survey of the history of the problem can be found in the introductory section of \cite{GHI2019}.

It has been observed that, in general, the $L^p$ mapping properties of $T^{\lambda}$ are determined by finer geometric conditions on the phase than H1) and H2) above \cite{Bourgain1991, Bourgain1995, Minicozzi1997, Wisewell2005}. In particular, in addition to the Hessian in \eqref{non vanishing curvature} having full rank, the behaviour of the operator can often depend on the \emph{signature} of the matrix.

\begin{definition}
Suppose $\phi$ is a phase which satisfies H1) and H2) above. The eigenvalues of the symmetric matrix 
\begin{equation*}
\partial^2_{\omega \omega} \langle \partial_x\phi(x;\omega),G(x;\omega_0)\rangle|_{\omega = \omega_0}
\end{equation*}
can be defined as continuous functions on $B^n \times B^{n-1}$ which are bounded away from 0. The signature of $\phi$ is defined to be the quantity $\mathrm{sgn} (\phi) := |\sigma_+ - \sigma_-|$ where $\sigma_+$ and $\sigma_-$ are, respectively, the number of positive and the number of negative eigenvalue functions.
\end{definition}

The aim of this article is to prove $L^p$ estimates for general H\"ormander-type operators, with a range of $p$ determined by the signature of the phase. 

\begin{theorem}\label{main theorem} Suppose $T^{\lambda}$ is a H\"ormander-type operator. For all $\varepsilon  > 0$ the \emph{a priori} estimate\footnote{Given a (possibly empty) list of objects $L$, for real numbers $A_p, B_p \geq 0$ depending on some Lebesgue exponent $p$ the notation $A_p \lesssim_L B_p$ or $B_p \gtrsim_L A_p$ signifies that $A_p \leq CB_p$ for some constant $C = C_{L,n,p} \geq 0$ depending on the objects in the list, $n$ and $p$. In addition, $A_p \sim_L B_p$ is used to signify that $A_p \lesssim_L B_p$ and $A_p \gtrsim_L B_p$.} 
\begin{equation}\label{Hormander estimate}
\|T^{\lambda}f\|_{L^p(\R^n)} \lesssim_{\varepsilon, \phi, a} \lambda^{\varepsilon}\|f\|_{L^{p}(B^{n-1})}
\end{equation}
holds whenever $p$ satisfies
\begin{equation}\label{main theorem 1}
p \geq \left\{\begin{array}{ll}
  \displaystyle 2 \cdot \frac{\mathrm{sgn} (\phi) + 2(n+1)}{\mathrm{sgn} (\phi) + 2(n-1)} & \textrm{if $n$ is odd} \\[8pt]
\displaystyle 2 \cdot \frac{\mathrm{sgn} (\phi) + 2n+3}{\mathrm{sgn} (\phi) + 2n-1} & \textrm{if $n$ is even} 
\end{array} \right. .
\end{equation}
\end{theorem}

The `extreme' cases of this result already appear in the literature:\medskip

\noindent\underline{Minimal $\sigma$.} Stein \cite{Stein1986} and Bourgain--Guth \cite{Bourgain2011} showed that \textit{all} H\"ormander-type operators satisfy \eqref{Hormander estimate} for\footnote{More precisely, Stein \cite{Stein1986} proved a stronger $L^2 \to L^p$ bound with no $\varepsilon$-loss in \textit{all} dimensions for $p \geq 2 \cdot \frac{n+1}{n-1}$. The larger range of exponents in the even dimensional case was later obtained by Bourgain--Guth \cite{Bourgain2011}.}
\begin{equation}\label{minimal signature range}
p \geq \left\{\begin{array}{ll}
  \displaystyle 2 \cdot \frac{n+1}{n-1} & \textrm{if $n$ is odd} \\[8pt]
\displaystyle 2 \cdot \frac{n+2}{n} & \textrm{if $n$ is even} 
\end{array} \right. .
\end{equation}
This yields Theorem~\ref{main theorem} in the special case where the signature is minimal (so that $\mathrm{sgn} (\phi) = 0$ if $n$ is odd and $\mathrm{sgn} (\phi) = 1$ if $n$ is even).\medskip

\noindent\underline{Maximal $\sigma$.}  On the other hand, if $\mathrm{sgn} (\phi) = n-1$, then it was shown by Lee \cite{Lee2006} for $n = 3$ (see also \cite{Bourgain2011}) and by Guth and the authors \cite{GHI2019} for $n \geq 4$ that \eqref{Hormander estimate} holds for \begin{equation*}
p \geq \left\{\begin{array}{ll}
  \displaystyle 2 \cdot \frac{3n+1}{3n-3} & \textrm{if $n$ is odd} \\[8pt]
\displaystyle 2 \cdot \frac{3n+2}{3n-2} & \textrm{if $n$ is even} 
\end{array} \right. ,
\end{equation*}
agreeing with the range of exponents in \eqref{main theorem 1}.\medskip

Theorem~\ref{main theorem} gives new bounds away from these extremes. In particular, in all other cases the previous best known range of exponents is \eqref{minimal signature range}, arising from the work of Stein \cite{Stein1986} and Bourgain--Guth \cite{Bourgain2011}. If $0 < \mathrm{sgn} (\phi) < n-1$ for $n$ odd or $1 < \mathrm{sgn} (\phi) < n-1$ for $n$ even, then \eqref{main theorem 1} provides a strictly larger range than \eqref{minimal signature range}. 

\subsection{Sharpness} An interesting feature of the result is that it is sharp for specific choices of operator, in the following sense.

\begin{proposition}\label{lin nec prop} For every dimension $n \geq 2$ and every $0 \leq \sigma \leq n-1$ such that $n-1 - \sigma$ is even, there exists a H\"ormander-type operator with $\mathrm{sgn}(\phi) = \sigma$ for which \eqref{Hormander estimate} fails whenever $p$ does not satisfy \eqref{main theorem 1}.
\end{proposition}

These examples are given by essentially taking tensor products of existent examples for the $\sigma = 0$ and $\sigma=n-1$ cases, which are due to Bourgain \cite{Bourgain1991, Bourgain1995} and Bourgain--Guth \cite{Bourgain2011} (see also \cite{Minicozzi1997, Wisewell2005}). The details are discussed in \S\ref{lin nec sec} below. 

\subsection{Non-sharpness} It is also important to note that there exist examples of operators for which \eqref{Hormander estimate} is known to hold for a wider range of exponents than \eqref{main theorem 1}. For instance, the extension operator $E_{\mathrm{ell}}$ associated to the elliptic paraboloid, which is a prototypical example in the maximal signature case, has been shown to satisfy a wider range of $L^p$ estimates than \eqref{main theorem 1} in all but a finite number of dimensions (see \cite{Bourgain2011, Guth2016, HR2019, Wang}). More generally, one may consider extension operators associated to arbitrary paraboloids.

\begin{example}\label{prototypical example} Given a non-degenerate quadratic form $Q \colon \R^{n-1} \to \R$, define the associated \textit{extension operator}
\begin{equation}\label{proto 1}
    E_Qf(x) := \int_{B^{n-1}} e^{2\pi i(\langle x',\omega\rangle +x_n Q(\omega))} f(\omega)\,\ud\omega, \qquad x=(x',x_n)\in\mathbb{R}^{n-1}\times \mathbb{R}.
\end{equation}
Let $0 \leq \sigma \leq n-1$ be such that $n-1-\sigma$ is even. Affine invariance typically reduces the study of these operators to that of the prototypical examples where
\begin{equation*}
    Q_\sigma(\omega):= \frac{1}{2}\inn{\mathrm{I}_{n-1,\sigma}\,\omega}{\omega} = \frac{1}{2} \sum_{j=1}^{\frac{n-1+\sigma}{2}} \omega_j^2 - \frac{1}{2} \sum_{j=\frac{n+1+\sigma}{2}}^{n-1} \omega_j^2.
\end{equation*}
Here, writing $\mathrm{I}_d$ for a $d \times d$ identity matrix, the $(n-1) \times (n-1)$ matrix $\mathrm{I}_{n-1,\sigma}$ is given in block form by
\begin{equation*}
   \mathrm{I}_{n-1,\sigma} :=
   \begin{bmatrix}
   \mathrm{I}_{\frac{n-1+\sigma}{2}} & 0 \\
   0 & - \mathrm{I}_{\frac{n-1-\sigma}{2}}
   \end{bmatrix}.
\end{equation*}
In this case, the corresponding phase in \eqref{proto 1} has signature $\sigma$ and $E_{\sigma} := E_{Q_{\sigma}}$ is the extension operator associated to (a compact piece of) the hyperbolic paraboloid \begin{equation*}
  \mathbb{H}^{n-1,\sigma} :=  \{(\omega,Q_{\sigma}(\omega)):\omega\in \R^{n-1}\}.
\end{equation*}
As discussed in \S\ref{reductions section} below, at a local level all H\"ormander-type operators are smooth  perturbations of the prototypical operators $E_{\sigma} $. \end{example}

It is conjectured \cite{Stein1979} that the operators $E_Q$ (and, in fact, extension operators associated to any surface of non-vanishing Gaussian curvature) are $L^p(B^{n-1}) \to L^p(\R^n)$ bounded for $p > 2\cdot \frac{n}{n-1}$, regardless of the signature. Restriction theory for hyperbolic parabol\ae\ involves a number of novel considerations compared with that of the elliptic case, and has been investigated in a variety of works \cite{Lee2005, Vargas2005, Bourgain2011, CL2017, Stovall2019, Barron}. There has also been a recent programme \cite{BMV, BMV2017, BMV2020, BMVa} to investigate $L^p$-boundedness of extension operators associated to negatively-curved surfaces given by smooth perturbations of the hyperbolic paraboloid $\mathbb{H}^{2,0}$ from Example~\ref{prototypical example}; this turns out to be a rather subtle problem for $p < 4$.  




\subsection{Relation to other problems} It is well-known that $L^p$ estimates for the Fourier extension operators are related to many central questions in harmonic analysis such as the Kakeya conjecture, the Bochner--Riesz conjecture and the local smoothing conjecture for the wave equation (see, for instance, \cite{Tao1998}). In the maximal signature case, $L^p$ estimates for H\"ormander-type operators imply Bochner--Riesz estimates and are further connected to curved variants of the above problems defined over manifolds (see, for instance, \cite{BHS2020, Sogge1987, Sogge2017}), although some of the implications are not as strong as in the Euclidean setting (see\footnote{Note the statements of Corollary 1.4 and Corollary 1.5 in \cite{GHI2019} contain an unwanted $\lambda^{(n-1)/2}$ factor. The authors thank Pierre Germain for pointing out this typographical error.}  \cite[\S1.2]{GHI2019} for results and further details). For operators with general signature, Theorem \ref{main theorem} relates to  further generalisations of the Kakeya and local smoothing problems, the latter now defined with respect to a class of Fourier integral operators. The connections with FIO theory are discussed in detail in \cite{BHS, BHS2020}; see \cite{Wisewell2005} and \cite{Bourgain2011} for further details of the underlying Kakeya-type problems.




\subsection{The r\^ole of the signature}\label{new sec}
The proof of Theorem~\ref{main theorem} follows the argument used to establish the $\mathrm{sgn}(\phi) = n-1$ case from \cite{GHI2019}, with a number of modifications to take account of the relaxed signature hypothesis. There are two significant points of departure from \cite{GHI2019}, where the signature plays a critical r\^ole in the argument (also reflected in the sharp examples in \S\ref{lin nec sec} and \S\ref{mult nec sec}). In both cases, to illustrate the underlying ideas it suffices only to consider the prototypical operators $E_Q$ introduced in Example~\ref{prototypical example}.\medskip

\noindent \textit{Partial transverse equidistribution.} Transverse equidistribution estimates were introduced in \cite{Guth2018} in relation to the elliptic extension operator $E_{\mathrm{ell}}$ and play a significant r\^ole here. In order to describe the setup, it is necessary to briefly review the notion of wave packet decomposition (see \S\ref{wave packet section} for further details). Decompose $B^{n-1}$ into a family of finitely-overlapping $R^{-1/2}$ discs $\theta = B(\omega_{\theta},R^{-1/2})$. By means of a partition of unity, for $f:B^{n-1}\rightarrow \mathbb{C}$ write $f = \sum_{\theta} f_{\theta}$ where each $f_{\theta}$ is supported in $\theta$. Forming a Fourier series decomposition, one may further decompose $f_{\theta} = \sum_v f_{\theta,v}$ where the frequencies $v$ lie in the lattice $R^{1/2}\Z^{n-1}$ and the $\hat{f}_{\theta,v}$ are essentially supported in disjoint balls of radius $R^{1/2}$. The functions $E_Q f_{\theta,v}$ satisfy the following key properties:
\begin{enumerate}[i)]
    \item On $B(0,R)$, each $E_Q f_{\theta,v}$ is essentially supported in a tube $T_{\theta,v}$ of length $R$ and diameter $R^{1/2}$ which is parallel to the normal direction $G(\omega_{\theta}):=(-\partial_\omega Q(\omega_\theta),1)^\top$ and has position dictated by $v$. 
    \item The Fourier transform $\big(E_Q f_{\theta,v}\big)\;\widehat{}\;$ has (distributional) support on the cap
    \begin{equation*}
        \Sigma(\theta) := \big\{ (\omega, Q(\omega)) : \omega \in \theta \big\}. 
    \end{equation*}
\end{enumerate}
For general H\"ormander-type operators $T^\lambda$ a similar setup holds, with the exception that the tubes $T_{\theta,v}$ carrying the functions $T^\lambda f_{\theta,v}$ may be curved (see \S\ref{wave packet section}).

The incidence geometry of the tubes $T_{\theta,v}$ is a major consideration in the $L^p$-theory of H\"ormander-type operators. A  critical case occurs when $f$ is chosen so that the $T_{\theta,v}$ for which $Ef_{\theta,v} \nequiv 0$\footnote{Or for which $Ef_{\theta,v}$ is ``non-negligable''.} are aligned along a lower dimensional manifold $Z$ (or, more precisely, a lower dimensional algebraic variety) inside $B(0,R)$; indeed, analogous situations appear when considering extremal configurations in classical incidence geometry (see, for instance, \cite{Guth2015}), and in fact the (variable coefficient) sharp examples in \S\ref{lin nec sec} exhibit similar structure. Under this hypothesis, by property i) above, $E_Qf$ is essentially supported in $N_{R^{1/2}}Z$, the $R^{1/2}$-neighbourhood of $Z$. It is important to note, however, that the $Ef_{\theta,v}$ each carry some oscillation. If there is sufficient constructive/destructive interference between the wave packets, then it could be the case that the mass of $E_Qf$ is concentrated in a much thinner subset of $N_{R^{1/2}}Z$. 

The signature influences the way in which the wave packets $E_Qf_{\theta,v}$ can interfere with each other. The reason behind this, as explained below, is that the signature largely determines the relationship between the direction $G(\omega_{\theta})$ of each tube $T_{\theta,v}$ on the spatial side and the position of the cap $\Sigma(\theta)$ on the frequency side. In the maximal signature case this relationship, together with the uncertainty principle, ensures that the mass of $E_Qf$ cannot concentrate in a thinner neighbourhood of the variety, but must be evenly spread across $N_{R^{1/2}}Z$. For general maximal signature H\"ormander-type operators, this property can be formally realised via \textit{transverse equidistribution estimates}, which roughly take the form\footnote{Here $\fint_E F := \frac{1}{|E|} \int_E F$ denotes the integral average.}
\begin{equation}\label{intro trans eq}
    \fint_{N_{\rho^{1/2}}Z \cap B(0,R)} |T^{\lambda}f|^2 \lesssim \fint_{N_{R^{1/2}}Z \cap B(0,R)} |T^{\lambda}f|^2, \qquad \rho \leq R.
\end{equation}
These estimates play an important r\^ole in the proof of the maximal signature case of Theorem~\ref{main theorem} by efficiently relating the wave packet geometry at different scales (see \cite{GHI2019, Guth2018}). 
If the maximal signature hypothesis is dropped,  however, then \eqref{intro trans eq} no longer holds in general. Nevertheless, there is a spectrum of weaker variants of \eqref{intro trans eq}, involving additional powers of $(R/\rho)$, which do hold in the general case. The relevant strength of these \textit{partial transverse equidistribution estimates} depends on the signature of the underlying operator. The precise form of these inequalities is discussed in \S\ref{trans eq sec} below.

It remains to explain how the signature affects the localisation properties of $E_Q f$. Here an elliptic case is contrasted with a hyperbolic case in $\mathbb{R}^3$, for wave packets aligned along the subspace $V:=\langle \vec{e}_1 \rangle^{\perp}$, the 2-dimensional plane orthogonal to $\vec{e}_1$.


%


\begin{figure}
    \centering
      \begin{subfigure}[b]{0.54\textwidth}
        \includegraphics[width=\textwidth]{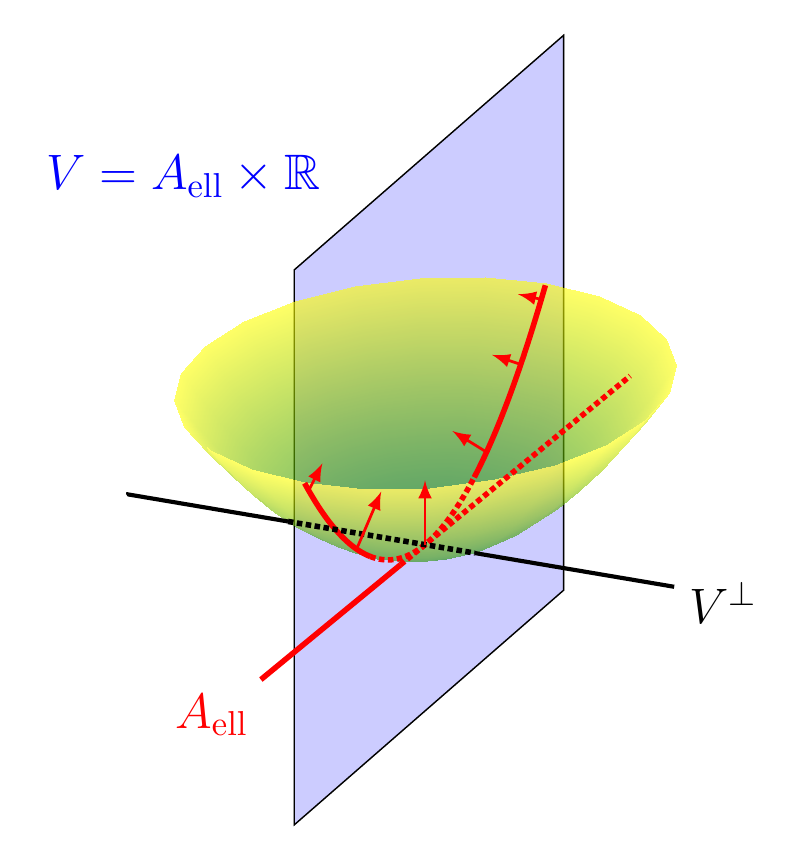}
    \end{subfigure}
\quad
    \begin{subfigure}[b]{0.42\textwidth}
        \includegraphics[width=\textwidth]{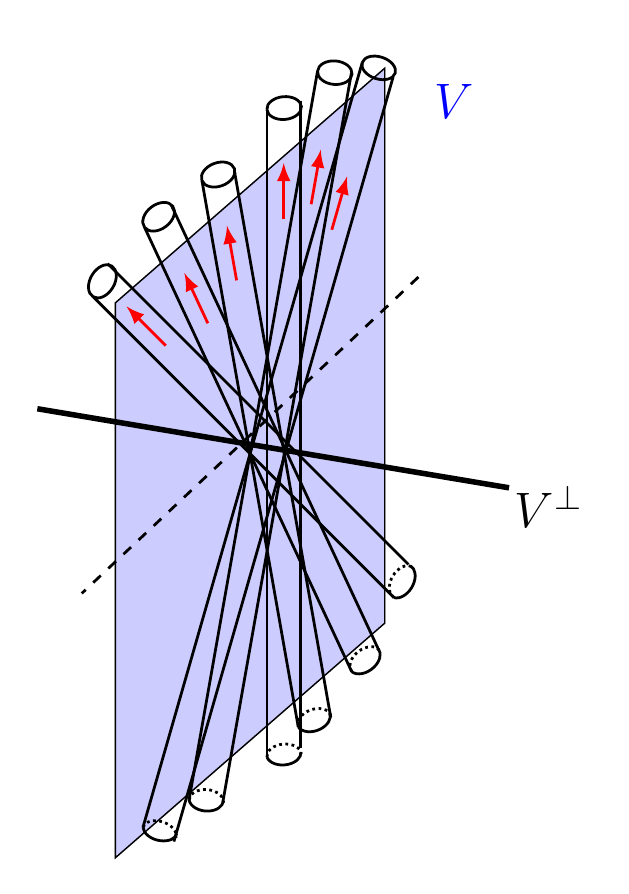}
         \end{subfigure}

    \caption{Transverse equidistribution in the elliptic case. On the spatial side (right-hand figure) the wave packets are aligned along a plane $V$. On the frequency side (left-hand figure), the frequency support is aligned along $V = A_{\mathrm{ell}} \times \R$.}
    \label{ell eq figure}
\end{figure}


%


In particular, consider the elliptic extension operator $E_{\mathrm{ell}}$ in $\R^3$ given by the signature 2 form $Q_{\mathrm{ell}}(\omega) := \frac{1}{2}\big(\omega_1^2 + \omega_2^2\big)$. The situation is depicted in Figure~\ref{ell eq figure}. The directions $G(\omega_{\theta})$ all lie inside $V$, thus the $\omega_{\theta}$ lie along the line $A_{\mathrm{ell}} = \{\omega_1 = 0\}$ in $\R^2$. The Fourier support of $E_{\mathrm{ell}}f$ thus lies in a union of caps $\Sigma(\theta)$ over $\theta$ centred along $A_{\mathrm{ell}}$, so $\supp \big(E_{\mathrm{ell}}f\big)\;\widehat{}\; \subseteq N_{R^{-1/2}} (A_{\mathrm{ell}} \times \R)$. Owing to this localisation, the uncertainty principle implies that $E_{\mathrm{ell}}f$ is essentially constant at scale $R^{1/2}$ in the direction transverse (that is, normal) to $A_{\mathrm{ell}} \times \R$. Crucially, $A_{\mathrm{ell}} \times \R=
V$, thus the mass of $E_{\mathrm{ell}}f$ must be \textit{equidistributed} across the slab $N_{R^{1/2}}(V)$ in the \textit{transverse} direction to $V$. This observation can be used to prove (a suitably rigorous formulation of) the transverse equidistribution estimate \eqref{intro trans eq} in this case: see \cite{Guth2018}.

The above case is somewhat special since $V$ equals $A_{\mathrm{ell}} \times \R$, the plane along which the Fourier support of $E_{\mathrm{ell}}f$ is aligned. For general 2-planes $V$, the Fourier support is aligned along a (possibly) different 2-plane $V'$. However, a key observation is that, in the elliptic case, $V$ and $V'$ only ever differ by a small angle, so again equidistribution of $E_{\mathrm{ell}}f$ holds at scale $R^{1/2}$ in the direction transverse to $V$. Moreover, the argument generalises to higher dimensions: if the tubes $T_{\theta,v}$ lie along a $k$-plane $V$ in $\R^n$, then $E_{\mathrm{ell}}f$ is equidistributed in directions belonging to $V^{\perp}$. Variants also hold when $V$ is replaced by a more general algebraic variety $Z$ (see \cite{Guth2018}).


%


\begin{figure}
    \centering
      \begin{subfigure}[b]{0.57\textwidth}
        \includegraphics[width=\textwidth]{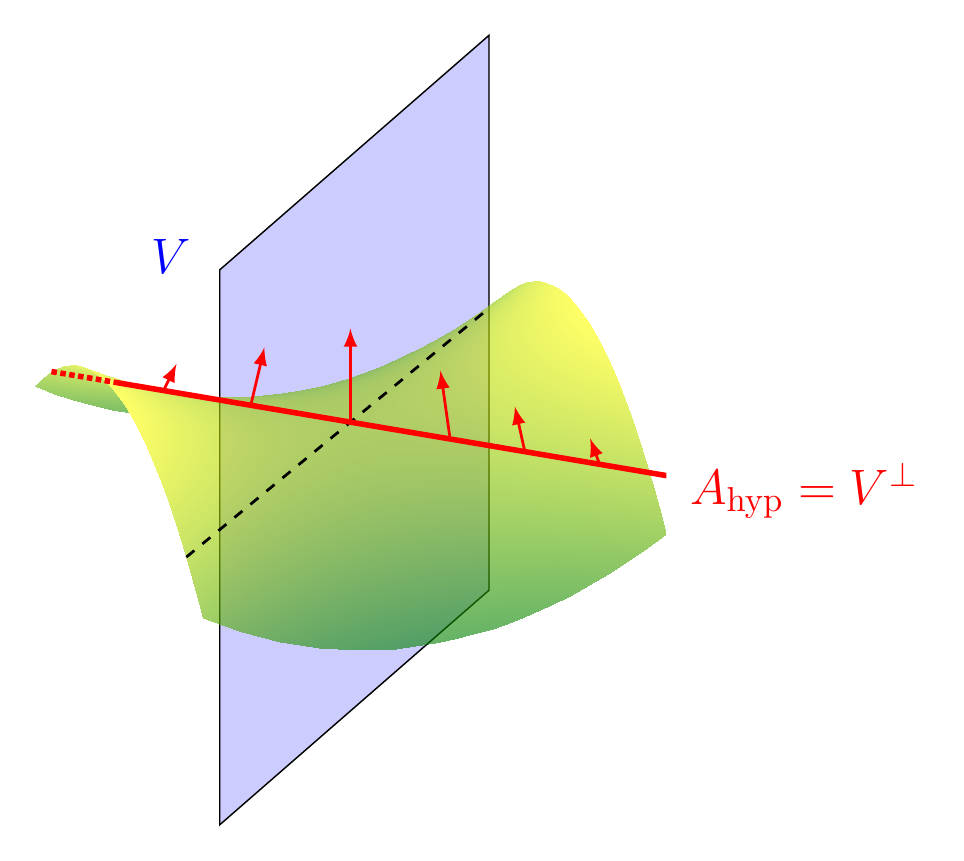}
    \end{subfigure}
\quad
    \begin{subfigure}[b]{0.39\textwidth}
        \includegraphics[width=\textwidth]{tube_eq_fig.pdf}
         \end{subfigure}

    \caption{Failure of transverse equidistribution in the hyperbolic case. On the spatial side (right-hand figure) the wave packets are aligned along the same plane $V$ as in the elliptic case. However, on the frequency side (left-hand figure), the frequency support is aligned along $V' = A_{\mathrm{hyp}} \times \R$ where $A_{\mathrm{hyp}} = V^{\perp}$.}
    \label{hyp eq figure}
\end{figure}


%


For contrast, now consider the case of the hyperbolic extension operator $E_{\mathrm{hyp}}$ in $\R^3$ given by the signature 0 form $Q(\omega) := \omega_1\omega_2$. This situation is depicted in Figure~\ref{hyp eq figure}. The $\omega_{\theta}$ must lie along $A_{\mathrm{hyp}} = \{\omega_2 = 0\}$, so $\supp \big(E_{\mathrm{hyp}}f\big)\;\widehat{}\;$ is contained in $N_{R^{-1/2}} (A_{\mathrm{hyp}} \times \mathbb{R})$. This localisation of the Fourier support guarantees that $E_{\mathrm{hyp}}f$ is equidistributed at scale $R^{1/2}$ in directions transverse to $A_{\mathrm{hyp}}\times \mathbb{R}$. However, this time, these directions are not transverse to $V$; instead, they lie \textit{along} $V$. Indeed, not only are $A_{\mathrm{hyp}}\times \mathbb{R}$ and $V$ different, but in fact $V^\perp \subseteq A_{\mathrm{hyp}}\times \mathbb{R}$. Consequently, the transverse equidistribution estimate \eqref{intro trans eq} no longer holds, and the constructive/destructive interference patterns between the $T_{\theta,v}$ can in fact lead to the concentration of the mass of $E_Qf$ in a tiny $O(1)$-neighbourhood of $V$. The variable coefficient counterexamples of Bourgain \cite{Bourgain1991, Bourgain1995} for H\"ormander-type operators of signature 0 exhibit destructive interference of this kind (see \cite{GHI2019} for further details).

In the mixed signature case in $\mathbb{R}^n$, in general only partial equidistribution occurs as a fusion of the above two situations. Specifically, consider an operator $E_Q$ associated to some $Q$ with signature $\sigma$ and let $V$ be a $k$-dimensional subspace of $\R^n$. In general, if the $T_{\theta,v}$ are aligned along $V$, then the Fourier support of $E_Qf$ will be aligned along a $k$-dimensional affine subspace $V' := A \times \R$, where $A = \{\omega \in B^{n-1} : G(\omega) \in V\}$. The problem is to understand the relationship between $V$ and $V'$. In particular, if $V$ and $V'$ are close to one another (that is, the angle between them is small), then this mirrors the situation in the elliptic case and transverse equidistribution holds. If $V$ and $V'$ are far from one another (that is, the angle between them is large), then this mirrors the above hyperbolic case and transverse equidistribution can fail. It transpires that, in general, a hybrid of these two situations occurs: a \textit{partial transverse equidistribution} holds for $E_Q f$ inside $N_{R^{1/2}}V$, where the equidistribution property holds only for directions lying in a certain \textit{subspace} $W$ of $V^{\perp}$. The dimension of $W$ can be bounded as a function of $n$, $k$ and, importantly, $\sigma$. If $\sigma$ is large then $W$ has large dimension and one is close to guaranteeing the full transverse equidistribution property \eqref{intro trans eq} enjoyed by the elliptic case. If $\sigma$ is small, then the dimension of $W$ is small and only a weak version of \eqref{intro trans eq} holds. For instance, if $\sigma \leq 2k - n - 1$, then the subspace $W$ can be zero dimensional, in which case no non-trivial transverse equidistribution estimates hold: see \S\ref{trans eq sec} for details.
 \medskip

\noindent \textit{Decoupling.} Although both elliptic and hyperbolic paraboloids have non-vanishing Gaussian curvature, hyperbolic paraboloids contain linear subspaces. The existence of such subspaces precludes certain bilinear estimates for extension operators associated to hyperbolic paraboloids \cite{Vargas2005,Lee2005} and means only weak $\ell^2$-decoupling inequalities hold for such operators \cite{BD2017}. In the present paper, the norm $\|T^{\lambda}f\|_{L^p(\R^n)}$ is studied via a \textit{broad/narrow analysis}, as introduced in \cite{Bourgain2011} (see also \cite{Guth2018, GHI2019}). This analysis involves certain $\ell^p$-decoupling estimates, the strength of which also depends on the signature. Similar observations have appeared previously in \cite{BD2017} and the recent paper~\cite{Barron}.

 In particular, the broad/narrow analysis requires analysing the so-called ``narrow'' contributions to $\|T^{\lambda}f\|_{L^p(\R^n)}$, which arise when the support of $f$ is localised close to a submanifold of $\mathbb{R}^{n-1}$. Consequently, one is led to consider certain slices of the (variable) hypersurfaces defined with respect to the phase $\phi$. These contributions are dealt with using a combination of a decoupling inequality and a rescaling argument. The efficiency of the decoupling inequality depends on how curved these slices are, which in turn depends on the signature.

More concretely, for the extension operator $E_{\sigma}f$ from Example~\ref{prototypical example}, the narrow contributions occur when the support of $f$ is localised close to an affine subspace $A$ of $\mathbb{R}^{n-1}$. In this case, as in the earlier discussion on transverse equidistribution, the Fourier transform $\big(E_{\sigma}f\big)\;\widehat{}\;$ is supported in a neighbourhood of the slice $\Sigma_A$ of $\mathbb{H}^{n-1,\sigma}$ formed by intersecting $\mathbb{H}^{n-1,\sigma}$ with the plane $A \times \R$. The favourable situation occurs when $\Sigma_A$ is well-curved, in the sense that the principal curvatures of this surface (viewed as a hypersurface lying in $A \times \R$) are all bounded away from zero. This is always the case for the elliptic paraboloid. For well-curved $\Sigma_A$ one may use the strong decoupling inequalities from \cite{BD2017} (or \cite{Bourgain2013, BD2015} in the elliptic case) to study the narrow contribution. For hyperbolic paraboloids, however, it can happen that a given slice coincides with a linear subspace of $\mathbb{H}^{n-1,\sigma}$: for instance, $\mathbb{H}^{n-1,\sigma}$ contains the $\frac{n-1-\sigma}{2}$-dimensional linear subspace of all $(\xi_1,\ldots,\xi_n)\in\hat{\mathbb{R}}^n$ satisfying \begin{equation*}
    \xi_j = \left\{\begin{array}{ll}
     \xi_{j+\frac{n-1+\sigma}{2}} &  \textrm{ for $1 \leq j \leq \tfrac{n-1-\sigma}{2}$,}  \\
     0 & \textrm{ for $\tfrac{n+1-\sigma}{2} \leq j \leq \tfrac{n-1+\sigma}{2}$ or $j=n$}\\
    \end{array}\right. .
\end{equation*}
In this case, owing to the lack of curvature, no non-trivial decoupling inequalities exist to control the narrow contribution and, consequently, much poorer estimates hold. In general, to obtain the best possible decoupling inequalities for a slice $\Sigma_A$, one needs to rely on the principal curvatures of $\Sigma_A$ which are bounded away from zero. The number of these curvatures can be estimated in terms of the signature $\sigma$. If $\sigma$ is large, then typically there will be many large principal curvatures and strong decoupling estimates will hold. If $\sigma$ is small, then for certain slices there will be few large principal curvatures and only weak decoupling estimates are available. This discussion is made precise in Proposition~\ref{dec prop} and Corollary~\ref{dec cor} below.




\subsection{Methodology: $k$-broad estimates} As in \cite{Guth2018, GHI2019}, the main ingredient in the proof of Theorem~\ref{main theorem} is a \textit{$k$-broad estimate}.

\begin{theorem}\label{k-broad theorem} Let $T^{\lambda}$ be a H\"ormander-type operator of reduced phase $\phi$. For all $2 \leq k \leq n$ and $\varepsilon > 0$ the $k$-broad estimate
\begin{equation*}
\|T^{\lambda}f\|_{\mathrm{BL}^p_{k,A}(\R^n)} \lesssim_{ \varepsilon} \lambda^{\varepsilon} \|f\|_{L^2(B^{n-1})}
\end{equation*}
holds for some integer $A\geq 1$ whenever $p$ satisfies $p \geq  \bar{p}(n,\mathrm{sgn}(\phi), k)$ for
\begin{equation*}
\bar{p}(n,\sigma, k) := \left\{\begin{array}{ll}
                                2 \cdot \frac{n+1}{n-1} & \textrm{for $1 \leq k \leq \frac{n+1-\sigma}{2}$} \\[6pt]
                                2 \cdot \frac{n+2k+1+\sigma}{n+2k-3+\sigma} & \textrm{for $\frac{n+1-\sigma}{2} \leq k \leq \frac{n+1+\sigma}{2}$} \\[6pt]
                                2 \cdot \frac{k}{k-1} & \textrm{for $\frac{n+1+\sigma}{2} \leq k \leq n$}
                               \end{array} \right. .
\end{equation*}
\end{theorem}

For the definition of the $k$-broad norm, see \cite{Guth2018, GHI2019}. For technical reasons, the theorem is stated for the slightly restrictive class of \textit{reduced phases}, which are defined in \S\ref{reductions section}. Once Theorem~\ref{k-broad theorem} is established, Theorem~\ref{main theorem} follows by a now-standard argument originating in \cite{Bourgain2011}: see \S\ref{broad/narrow sec} for further details.

As with Theorem~\ref{main theorem}, certain `extreme' cases of Theorem~\ref{k-broad theorem} can be deduced from existent results:
\begin{itemize}
\item For $1 \leq k \leq \frac{n+1-\mathrm{sgn}(\phi)}{2}$ the result follows from Stein's oscillatory integral estimate \cite{Stein1986}.
\item For $\frac{n+1+\mathrm{sgn}(\phi)}{2} \leq k \leq n$ the result follows from the multilinear oscillatory integral estimates of Bennett--Carbery--Tao \cite{Bennett2006}.\footnote{The oscillatory integral estimates in \cite{Bennett2006} are stated only at the $n$-linear level but the argument adapts to give results at all levels of linearity: see \cite[\S 5]{Bourgain2011} for an explicit statement of the $k$-linear estimates. The passage from multilinear to $k$-broad inequalities is described in detail in \cite[\S 6]{GHI2019}.}
\item If $\mathrm{sgn}(\phi) = n - 1$, then the $k = 2$ case follows from the bilinear estimates of Lee \cite{Lee2006} and all remaining values of $k$ (under the maximal signature assumption) are treated in \cite{GHI2019}.
\end{itemize}
In all other cases Theorem~\ref{k-broad theorem} is new. It is also sharp in the sense that the range of $p$ cannot be extended. This can be shown by considering extension operators of the type discussed in Example~\ref{prototypical example} above. The range of $L^p$ is then given by testing the estimate against functions formed by tensor products of the standard test functions appearing in, for instance, \cite{Vargas2005}. The sharpness of Theorem~\ref{k-broad theorem} is discussed in detail in \S\ref{mult nec sec} below.  

Theorem~\ref{k-broad theorem} has a multilinear flavour, and serves as a substitute for the stronger $k$-linear Conjecture~\ref{k-linear conjecture} below.

\begin{definition}
Let $1 \leq k \leq n$ and $\mathbf{T}= (T_1, \dots, T_k)$ be a $k$-tuple of H\"ormander-type operators of the same signature, where $T_j$ has associated phase $\phi_j$, amplitude $a_j$ and generalised Gauss map $G_j$ for $1 \leq j \leq  k$. Then $\mathbf{T}^{\lambda}$ is said to be $\nu$-transverse for some $0 < \nu \leq 1$ (and all $\lambda \geq 1$) if
\begin{equation*}
\big|\bigwedge_{j=1}^k G_j(x;\omega_j) \big| \geq \nu \quad \textrm{for all  $(x; \omega_j) \in \mathrm{supp}\,a_j$  for $1 \leq j \leq k$. }
\end{equation*}
\end{definition}

\begin{conjecture}\label{k-linear conjecture} Let $ (T_1, \dots, T_k)$ be a $\nu$-transverse $k$-tuple of H\"ormander-type operators of the same signature $\sigma$. For any $\lambda \geq 1$ and $1\leq k\leq n$ the $k$-linear estimate
\begin{equation*}
\big\| \prod_{j=1}^k |T^{\lambda}_jf_j|^{1/k}\big\|_{L^p(\R^n)} \lesssim_{\nu, \phi} \prod_{j=1}^k \|f_j\|_{L^2(B^{n-1})}^{1/k}
\end{equation*}
holds whenever $p$ satisfies $p\geq \bar{p}(n,\sigma,k)$. 
\end{conjecture}

This conjecture is a natural generalisation of a conjecture of Bennett \cite{Bennett2014} concerning the elliptic case. It formally implies Theorem \ref{k-broad theorem} (see \cite[\S 6.2]{GHI2019}).




\subsection{Structure of the article} The layout of the article is as follows:
\begin{itemize}
    \item In \S\ref{lin nec sec} the sharpness of Theorem~\ref{main theorem} is demonstrated and, in particular, the proof of Proposition~\ref{lin nec prop} is presented.
    \item In \S\ref{mult nec sec} the sharpness of Theorem~\ref{k-broad theorem} and Conjecture~\ref{k-linear conjecture} is discussed.
\end{itemize}

The remainder of the article deals with the proofs of Theorems~\ref{main theorem} and~\ref{k-broad theorem}. The presentation is \textit{not} self-contained. In particular, the sister paper \cite{GHI2019}, which treats the maximal signature case, is heavily referenced. The argument in \cite{GHI2019} is fairly modular in nature and, as discussed in \S\ref{new sec}, the signature hypothesis plays a crucial r\^ole only in two places in the argument:
\begin{enumerate}[i)]
    \item The \textit{transverse equidistribution estimates}, which are used to prove the bounds for the $k$-broad norms.
    \item The \textit{decoupling estimates}, used in the passage from $k$-broad to linear estimates as part of the Bourgain--Guth method \cite{Bourgain2011}. 
\end{enumerate} 
These two isolated steps are treated in detail in the present paper. Many other parts of the proof are merely sketched or even omitted entirely, since they are either minor modifications of or identical to corresponding arguments in \cite{GHI2019}. Indeed, once the transverse equidistribution and decoupling theory is established in the general signature setting, the rest of the argument from \cite{GHI2019} carries through with only changes to the numerology. In particular, the remainder of the article proceeds as follows:

\begin{itemize}
    \item In \S\ref{preliminaries section} various preliminaries for the proofs of Theorems~\ref{main theorem} and~\ref{k-broad theorem} are recalled from the literature. This includes the definition of the $k$-broad norms and operators of reduced phase. 
    \item In \S\ref{trans eq sec} the crucial transverse equidistribution estimates are stated and proved. 
    \item In \S\ref{broad theorem proof sec} there is a brief description of how to adapt the argument from \cite{Guth2018, GHI2019}, using the transverse equidistribution results from the previous section, to prove Theorem~\ref{k-broad theorem}.
    \item In \S\ref{narrow sec} the relevant  decoupling theory is discussed.
    \item In \S\ref{broad/narrow sec} Theorem~\ref{k-broad theorem} is combined with the decoupling estimates from \S\ref{narrow sec} to complete the proof Theorem~\ref{main theorem}.
\end{itemize}

\begin{acknowledgment} The authors would like to thank  Alex Barron and Larry Guth for discussions on topics related to this article. This material is partly based upon work supported by the National Science Foundation under Grant No. DMS-1440140 while the authors were in residence at the Mathematical Sciences Research Institute in Berkeley, California, during the Spring 2017 semester.
\end{acknowledgment}




\section{Necessary conditions: linear bounds}\label{lin nec sec}

\subsection{Overview} In this section sharp examples for Theorem~\ref{main theorem} are obtained, thereby proving Proposition~\ref{lin nec prop}. They arise simply by tensoring existing examples for the extremal cases of minimal and maximal signatures. 

All of the phases considered below are of the following basic form: given a smooth 1-parameter family of symmetric matrices $\mathbf{A} \colon \R \to \mathrm{Mat}(n-1, \R)$, define $\phi \colon \R^n \times \R^{n-1} \to \R$ by
\begin{equation}\label{lin ex 1}
    \phi(x;\omega) := \inn{x'}{\omega} + \frac{1}{2} \inn{\mathbf{A}(x_n)\,\omega}{\omega}. 
\end{equation}
In order for this phase function to satisfy the conditions H1) and H2) from the introduction, the component-wise derivative $\mathbf{A}'$ of $\mathbf{A}$ must be invertible on a neighbourhood of the origin. In this case, the signature of the phase function $\phi$ corresponds to the common signature of the matrices $\mathbf{A}'(x_n)$ for $x_n$ near 0.

In the forthcoming examples $T^{\lambda}$ is taken to be a H\"ormander-type operator defined with respect to the phase $\phi^{\lambda}$ for some $\phi$ as in \eqref{lin ex 1}, and an amplitude with sufficiently small support so that the conditions H1) and H2) are satisfied. The analysis pivots on finding suitable choices of $\mathbf{A}$ and test functions $f$ so that $T^{\lambda}f$ is highly concentrated near a low degree algebraic variety. In particular, the varieties in question will be hyperbolic paraboloids of the form 
\begin{equation}\label{lin ex 2}
    Z_d := \big\{x \in \R^d : x_{2j-1}x_d=\lambda x_{2j}\;\textrm{for all } 1 \leq j \leq \lfloor\tfrac{d-1}{2}\rfloor\big\}.
\end{equation}
Note that each $Z_d$ is of dimension $m_d := \lfloor \frac{d+2}{2} \rfloor$. This corresponds to the minimal dimension for `Kakeya sets of curves' in $\R^d$: see \cite{Bourgain1991, Wisewell2005, Bourgain2011}. For further details on the r\^ole of algebraic varieties in the study of oscillatory integral operators see, for instance, the introductory discussions in \cite{Guth2018} or \cite{GHI2019}.

\subsection{Hyperbolic example}\label{lin hyp sec} The first example is due to Bourgain \cite{Bourgain1991} (see also \cite{Bourgain1995}) and corresponds to the minimal signature case.  

For $d \geq 3$ odd let $\mathbf{H}_d \colon \R \to \mathrm{Mat}(d-1,\R)$ be given by
\begin{equation*}
\mathbf{H}_d(t) := 
\underbrace{\begin{pmatrix}
0 & t \\
t & t^2
\end{pmatrix}
\oplus \dots \oplus 
\begin{pmatrix}
0 & t \\
t & t^2
\end{pmatrix}}_{\text{$ \frac{d-1}{2}$-fold}} 
\end{equation*}
Near the origin the derivative matrix $\mathbf{H}_d'(t)$ is a perturbation of 
\begin{equation}\label{lin hyp ex 1}
    \begin{pmatrix}
0 & 1 \\
1 & 0
\end{pmatrix}
\oplus \dots \oplus 
\begin{pmatrix}
0 & 1 \\
1 & 0
\end{pmatrix}
\end{equation}
and is therefore invertible with signature 0. Note that \eqref{lin hyp ex 1} corresponds to the matrix $\mathrm{I}_{d-1,\sigma}$ from Example~\ref{prototypical example} after a coordinate rotation. 

Taking $\mathbf{A} = \mathbf{H}_d$, let $T_{\mathrm{hyp}}^{\lambda}$ be a  H\"ormander-type operator with phase $\phi^{\lambda}$ for $\phi$ as defined in \eqref{lin ex 1}. A key observation of Bourgain \cite{Bourgain1991} is that there exists\footnote{In fact, one may take $h \equiv 1$.} a smooth function $h \colon \R^{d-1} \to \mathbb{C}$ satisfying:
\begin{itemize}
    \item $|h(\omega)| \sim 1$ for all $\omega \in B^{d-1}$.
    \item There exists a dimensional constant $c > 0$ such that
    \begin{equation}\label{lin hyp ex 2}
    |T_{\mathrm{hyp}}^{\lambda} h(x)|\gtrsim \lambda^{-\frac{d-1}{4}} \qquad \text{for all $x\in N_cZ_d\cap B(0,\lambda)$,}
\end{equation}
where the variety $Z_d$ is as in \eqref{lin ex 2}. 
\end{itemize}
This bound follows from a simple stationary phase computation. In addition to \cite{Bourgain1991, Bourgain1995}, see the expositions in \cite{Wisewell2005, Sogge2017, GHI2019} for further details.

\subsection{Elliptic example}\label{lin ell sec} The second example is due to Bourgain--Guth \cite{Bourgain2011} and corresponds to the maximal signature case.  

For $d \geq 2$ let $\mathbf{E}_d \colon \R \to \mathrm{Mat}(d-1,\R)$ be given by 
\begin{equation*}
\mathbf{E}_d(t) := 
\underbrace{\begin{pmatrix}
t & t^2 \\
t^2 & t+t^3
\end{pmatrix}
\oplus \dots \oplus 
\begin{pmatrix}
t & t^2 \\
t^2 & t+t^3
\end{pmatrix}}_{\text{$\lfloor\frac{d-1}{2}\rfloor$-fold}}
\oplus 
\begin{pmatrix}
t
\end{pmatrix}^*
\end{equation*}
where the $^*$ indicates that the final $(t)$ block appears if and only if $d$ is even. Near the origin the derivative matrix $\mathbf{E}_n'$ is a perturbation of the identity and is therefore invertible with maximal signature $d-1$. 

Taking $\mathbf{A} = \mathbf{E}_d$, let $T_{\mathrm{ell}}^{\lambda}$ be a  H\"ormander-type operator with phase $\phi^{\lambda}$ for $\phi$ as defined in \eqref{lin ex 1}. Roughly speaking, in \cite{Bourgain2011} it is shown that there exists a smooth function $g \colon \R^{d-1} \to \mathbb{C}$ satisfying:
\begin{itemize}
    \item $|g(\omega)| \sim 1$ for all $\omega \in B^{d-1}$.
    \item There exists a dimensional constant $c > 0$ such that
    \begin{equation}\label{lin ell ex 1}
    |T_{\mathrm{ell}}^{\lambda}\, g(x)|\gtrsim \lambda^{-(d+m_d-2)/4} \qquad \text{for all $x\in N_{c\lambda^{1/2}}Z_d\cap B(0,\lambda)$,}
\end{equation}
where the variety $Z_d$ is as in \eqref{lin ex 2} and $m_d = \dim Z_d = \lfloor\frac{d+2}{2}\rfloor$. 
\end{itemize}
The estimate \eqref{lin ell ex 1} is not quite precise since the example in \cite{Bourgain2011} is randomised and the pointwise bound \eqref{lin ell ex 1} holds only in expectation. However, there exists a function $g$ for which the weaker substitute
   \begin{equation}\label{lin ell ex 2}\tag{\theequation$'$}
    \|T_{\mathrm{ell}}^{\lambda}\, g\|_{L^p(\R^d)} \gtrsim \lambda^{-(d+m_d-2)/4} \lambda^{(d+m_d)/2p} 
\end{equation}
does hold, and this suffices for the present purpose. In addition to \cite{Bourgain2011}, see the exposition in \cite{GHI2019} for further details. 

\subsection{Tensored examples} To prove Proposition~\ref{lin nec prop}, the linear estimates are tested against examples formed by tensoring the hyperbolic and elliptic examples described above. To this end, fix $1 \leq \sigma \leq n - 1$ with $n-1-\sigma$ even and let 
\begin{equation*}
    \mathbf{A}_{n,\sigma} := \mathbf{H}_{n-\sigma} \oplus \mathbf{E}_{\sigma+1} \colon \R \to \mathrm{Mat}(n-1,\R). 
\end{equation*}
Taking $\mathbf{A} = \mathbf{A}_{n,\sigma}$, let $T^{\lambda}$ be a H\"ormander-type operator with phase $\phi^{\lambda}$ for $\phi$ as defined in \eqref{lin ex 1}. Let $f$ denote the tensor product $f := h \otimes g \colon \R^{n-1} \to \C$ where
\begin{itemize}
    \item $h \colon \mathbb{R}^{n-\sigma-1}\rightarrow\mathbb{C}$ is a hyperbolic example as in \S\ref{lin hyp sec} in dimension $n - \sigma$,
    \item $g \colon :\mathbb{R}^\sigma\rightarrow\mathbb{C}$ is an elliptic example as in \S\ref{lin ell sec} in dimension $\sigma+1$.
\end{itemize}

If the amplitudes are suitably defined, then it follows that
\begin{equation}\label{lin ten ex 1}
    T^{\lambda}f(x) = T^{\lambda}_{\mathrm{hyp}}h(x',x_n)T^{\lambda}_{\mathrm{ell}}\,g(x'',x_n) \quad \textrm{for $x = (x',x'',x_n) \in \R^{n-\sigma - 1} \times \R^{\sigma} \times \R$,}
\end{equation}
where $T^{\lambda}_{\mathrm{hyp}}$ is defined with respect to $\mathbf{H}_{n-\sigma}$ and $T^{\lambda}_{\mathrm{ell}}$ is defined with respect to $\mathbf{E}_{\sigma+1}$. 

Suppose that for all $\varepsilon>0$ the estimate 
\begin{equation*}
    \|T^{\lambda} f\|_{L^p(\R^n)}\lesssim_{\varepsilon} \lambda^{\varepsilon} \|f\|_{L^p(B^{n-1})}
\end{equation*} 
holds for $T^{\lambda}$ and $f$ as above, uniformly in $\lambda$. The construction ensures that $\|f\|_{L^p(B^{n-1})} \sim 1$ and so 
\begin{equation}\label{lin ten ex 2}
    \|T^{\lambda} f\|_{L^p(\R^n)}\lesssim_{\varepsilon} \lambda^{\varepsilon}.
\end{equation} 
Thus, to obtain the desired $p$ constraints, the problem is to bound the left-hand side of \eqref{lin ten ex 2} from below.

Before proceeding, it is helpful to make a few simple geometric observations regarding the varieties $Z_d$. Given $x_d \in \R$ let
\begin{equation*}
 Z_d[x_d] := Z_d \cap (\R^{d-1} \times \{x_d\})  
\end{equation*}
denote the $x_d$-slice of $Z_d$. It is clear from the definition that the slices $Z_d[x_d]$ are affine subspaces of dimension $m_d-1$. Thus, for $c \sim 1$, one has the volume bound
\begin{equation} \label{lin ten ex 3}
    |N_{c} Z_d[x_d] \cap B(0,\lambda)| \gtrsim \lambda^{m_d-1} \qquad \textrm{for all $|x_d| \leq \lambda/2$},
\end{equation}
where, for each $x_d$, the neighbourhood $N_cZ_d[x_d]$ is considered inside the affine space $\mathbb{R}^{n-\sigma-1}\times \{0\}^{\sigma}\times\{x_d\}$. By \eqref{lin ten ex 1} and Fubini's theorem,
\begin{equation*}
     \|T^{\lambda} f\|_{L^p(\R^n)}^p = \int_{\R} \|T_{\mathrm{hyp}}^{\lambda}h\|_{L^p(\R^{n-\sigma-1}\times\{x_n\})}^p\|T_{\mathrm{ell}}^{\lambda}\,g\|_{L^p(\R^{\sigma}\times\{x_n\})}^p  \,\ud x_n.
\end{equation*}
At the expense of an inequality, one may restrict the $L^p(\R^{n-\sigma-1}\times \{x_n\})$ norm integration to the slice $N_cZ_{n-\sigma}[x_n] \cap B(0,\lambda)$
for the constant $c$ as in \S\ref{lin hyp sec}. In view of \eqref{lin hyp ex 2} and \eqref{lin ten ex 3}, it follows that 
\begin{equation*}\nonumber
     \|T^{\lambda} f\|_{L^p(\R^n)} \gtrsim \lambda^{-(n-1-\sigma)/4}\lambda^{(m_{n-\sigma}-1)/p}\|T_{\mathrm{ell}}^{\lambda}\,g\|_{L^p(\R^{\sigma} \times [-\lambda/2, \lambda/2])}
\end{equation*}     
If the amplitude of $T^{\lambda}_{\mathrm{ell}}$ has suitably small $x_d$-support, then the right-hand norm coincides with the global $L^p$-norm and one may apply \eqref{lin ell ex 2} to conclude that
\begin{equation}\label{lin ten ex 4}
 \|T^{\lambda} f\|_{L^p(\R^n)} \gtrsim \lambda^{-(n-1-\sigma)/4} \lambda^{(m_{n-\sigma}-1)/p} \lambda^{-(\sigma+m_{\sigma + 1}-1)/4} \lambda^{(\sigma+m_{\sigma+1} + 1)/2p}. 
\end{equation}

In order for \eqref{lin ten ex 2} to hold uniformly in $\lambda$, the exponent on the right-hand side of \eqref{lin ten ex 4} must be non-positive. Note that the parities of $n$ and $\sigma+1$ agree and so
\begin{equation*}
 m_{n-\sigma} = \frac{n-\sigma+1}{2} \quad \textrm{and} \quad   m_{\sigma + 1} = \left\{
    \begin{array}{ll}
      \frac{\sigma+2}{2}  & \textrm{if $n$ is odd} \\[3pt]
      \frac{\sigma+3}{2}  & \textrm{if $n$ is even}
    \end{array}
     \right. .
\end{equation*}
Thus, a little algebra shows that the non-positivity of the right-hand exponent in \eqref{lin ten ex 4} is equivalent to
\begin{align*}
    \frac{2(n-1)+\sigma}{2}- \frac{2(n+1)+\sigma}{p}&\geq 0 \qquad \textrm{if $n$ is odd,} \\
    \frac{2n-1+\sigma}{2}- \frac{2n+3+\sigma}{p}&\geq 0 \qquad \textrm{if $n$ is even,}
\end{align*}
which yields the desired condition \eqref{main theorem 1} after rearranging.




\section{Necessary conditions: multilinear bounds}\label{mult nec sec}

Here examples of H\"ormander-type operators are constructed which demonstrate that the range of exponents in Conjecture~\ref{k-linear conjecture} cannot be extended. 

\begin{proposition}\label{sharp multilinear proposition} Conjecture~\ref{k-linear conjecture} is sharp, in the sense that the conditions on $p$ are necessary.
\end{proposition}

The proof of Proposition~\ref{sharp multilinear proposition} can be slightly modified to demonstrate the sharpness of Theorem~\ref{k-broad theorem}, up to $\varepsilon$-loss. The details of this simple modification are omitted; see \cite{Guth2018} for a discussion of the elliptic case.

Similarly to the examples for Theorem~\ref{main theorem} discussed in the previous section, the sharpness of the multilinear estimates may be deduced by tensoring appropriate examples from extremal signature regimes. In the multilinear case, however, one may simply work with the prototypical extension operators associated to hyperbolic parabol\ae\ from Example~\ref{prototypical example}.




\subsection{Hyperbolic example}\label{hyp ex section} The first example exploits the fact that hyperbolic parabol\ae\ contain affine subspaces and is a direct generalisation of the bilinear example from \cite{Vargas2005}. The example is applied in the extreme case where the signature of the underlying quadratic form is zero. In particular, let $d \in \N$ be odd and consider the zero signature quadratic form
\begin{equation*}
    Q(\omega) := \sum_{j=1}^{d-1} \omega_{2j-1}\omega_{2j} \qquad \textrm{for} \quad \omega \in \R^{d-1}.
\end{equation*}
Note that this agrees with the form $Q_0$ from Example~\ref{prototypical example} after an orthogonal coordinate transformation.

Let $\psi \in C^{\infty}(\R^{(d-1)/2})$ be non-negative, supported in the unit ball and equal to 1 in a neighbourhood of the origin. Fix $a_i \in \R^{(d-1)/2}$ for $0 \leq i \leq (d-1)/2$ such that $a_0 = 0$, $|a_i| \leq 1/2$ and
\begin{equation}\label{hyp ex 1}
    |a_1 \wedge \dots \wedge a_{(d-1)/2}| \gtrsim 1.
\end{equation} 

 For $\lambda \geq 1$ and $1 \leq \ell \leq (d+1)/2$ define the \textit{$\ell$-linear hyperbolic example in $\R^d$} as the $\ell$-tuple of functions $\mathfrak{H}(d,\ell) := (h_1, \dots, h_{\ell})$ where each $h_j \in C^{\infty}_c(\R^{d-1})$ is given by
\begin{equation*}
h_j(\omega) := \psi(10(\omega_{\mathrm{odd}} - a_{j-1})) \psi(\lambda \omega_{\mathrm{even}}),
\end{equation*}
where $\omega_{\mathrm{odd}} \in \R^{(d-1)/2}$ (respectively,  $\omega_{\mathrm{even}} \in \R^{(d-1)/2}$) is the vector formed from the odd (respectively, even) components of $\omega$; see Figure~\ref{comparing examples figure}.

Clearly, for any such $\mathfrak{H}(d,\ell)$ one may bound
\begin{equation}\label{hyp ex 2}
    \|h_j\|_{L^2(\R^{d-1})} \lesssim \lambda^{-(d-1)/4} \qquad \textrm{for $1 \leq j \leq \ell$.}
\end{equation}
On the other hand, if $\omega \in \supp h_j$, then $|\omega_{\mathrm{odd}}| \leq 2$, $|\omega_{\mathrm{even}}| \leq \lambda^{-1}$ and $|Q(\omega)| \leq \lambda^{-1}$. Thus, 
\begin{equation}\label{hyp ex 3}
    |E_Qh_j(x,t)| \gtrsim \lambda^{-(d-1)/2} \qquad \textrm{for $(x,t) \in \Pi_d(\lambda)$}
\end{equation}
where $\Pi_d(\lambda)$ is the rectangular region
\begin{equation*}
  \Pi_d(\lambda)  := [-c,c]^{(d-1)/2} \times [-c\lambda,c\lambda]^{(d+1)/2}
\end{equation*}
for $c > 0$ a sufficiently small dimensional constant.


%


\begin{figure}
    \centering
      \begin{subfigure}[b]{0.48\textwidth}
        \includegraphics[width=\textwidth]{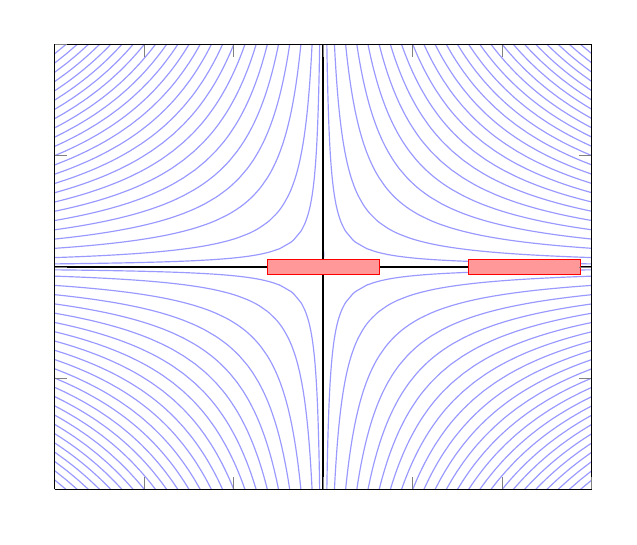}
    \end{subfigure}
\quad
    \begin{subfigure}[b]{0.48\textwidth}
        \includegraphics[width=\textwidth]{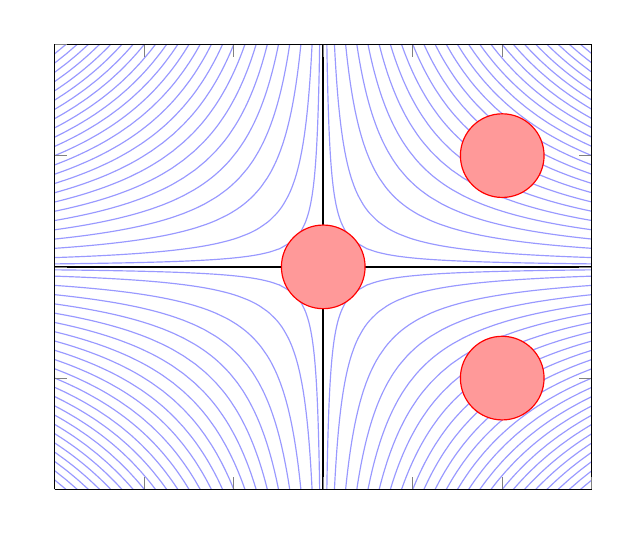}
         \end{subfigure}

    \caption{The \textit{hyperbolic} (left) and \textit{elliptic} (right) examples. In the hyperbolic case, slabs of thickness $\lambda^{-1}$ and width 1 are placed in relation to a linear subspace contained in the hyperbolic paraboloid (in this case the horizontal axis). In the elliptic case (which is also applied in the case of hyperbolic $Q$) a family of $\lambda^{-1/2}$ balls are placed around centres spanning a non-degenerate simplex. Since there is more freedom to place the balls in the elliptic case, it applies at higher levels of multilinearity.}
    \label{comparing examples figure}
\end{figure}


%





\subsection{Elliptic example}\label{ell ex section} The second example corresponds to the sharp example for $L^2$-based multilinear restriction for the elliptic paraboloid. It is a direct generalisation of the bilinear example described, for instance, in \cite{Tao2004}. This example will be applied in both elliptic and hyperbolic cases, but nevertheless is referred to as the \textit{elliptic} example to distinguish it from the hyperbolic example described above. 

For $1 \leq \sigma \leq d-1$ let
\begin{equation*}
    Q(\omega) := \sum_{j=1}^{\frac{d-1-\sigma}{2}} \omega_{2j-1} \omega_{2j} + \frac{1}{2}\sum_{j=d-\sigma}^{d-1} \omega_j^2. 
\end{equation*}
be a quadratic form in $d-1$ variables of signature $\sigma$. In contrast with the hyperbolic case, here the choice of $\sigma$ is not relevant to the numerology arising from the elliptic example. Note that this form agrees with the form $Q_{\sigma}$ from Example~\ref{prototypical example} after an orthogonal coordinate transformation.

 Let $G_{Q,0} \colon \R^{d-1} \to \R^d$ denote the (non-normalised) Gauss map $G_{Q,0}(\omega) :=  (-\partial_{\omega} Q(\omega), 1)^{\top}$ associated to $Q$. Fix $b_j \in \R^{d-1}$ for $0 \leq j \leq d-1$ satisfying $b_0=0$, $|b_j| \leq 1/2$ and
\begin{equation}\label{ell ex 1}
    |b_1 \wedge \dots \wedge b_{d-1}| \gtrsim 1.
\end{equation} 
For $1 \leq \ell \leq d$ let $V_{\ell}$ denote the $\ell$-dimensional subspace of $\R^d$ given by
\begin{equation*}
    V_{\ell} := \big\langle
  G_{Q,0}(b_j)
    \colon
   0 \leq j \leq \ell-1 \big \rangle.
\end{equation*}
For $C \geq 1$ a suitably large dimensional constant and given $\lambda \geq 1$, define $\mathcal{V}_\ell$ to be a maximal $C\lambda^{1/2}$-separated set in $V_\ell\cap \mathbb{R}^{n-1}\times\{0\}\cap B(0,\lambda)$.

The \textit{$\ell$-linear elliptic example in $\R^d$} is the $\ell$-tuple of functions $\mathfrak{G}(d,\ell) := (g_1, \dots, g_{\ell})$ where each $g_j \in C^{\infty}_c(\R^{d-1})$ is given by
\begin{equation*}
g_j := \sum_{v \in \mathcal{V}_{\ell}} g_{j,v} \qquad \textrm{where} \qquad g_{j,v}(\omega) := e^{-2\pi i \inn{v}{\omega - b_{j-1}}}\psi(\lambda^{1/2}(\omega - b_{j-1}))
\end{equation*}
 for $\psi \in C^{\infty}(\R^{d-1})$ a fixed function which is non-negative, supported in the unit ball and equal to 1 in a neighbourhood of the origin; see Figure~\ref{comparing examples figure}.

For any such $\mathfrak{G}(d,\ell)$, using Plancherel, one may bound
\begin{equation}\label{ell ex 3}
    \|g_j\|_{L^2(\R^{d-1})} \lesssim \Big(\sum_{v \in \mathcal{V}_{\ell}} \|g_{j,v}\|_{L^2(\R^{d-1})}^2\Big)^{1/2} \lesssim \lambda^{-(d-\ell)/4} \qquad \textrm{for $1 \leq j \leq \ell$.}
\end{equation}
On the other hand, (non)-stationary phase shows that, on $B(0,\lambda)$, the function $E_Qg_{j,v}$ is rapidly decaying away from the `tube'
\begin{equation*}
    T_{j, v} := \big\{ (x,t) \in B(0,\lambda) : |x - v + t \partial_{\omega}Q(b_{j-1})| \leq c\lambda^{1/2},\,\, |t| \leq \lambda\},
\end{equation*}
where $c > 0$ is a suitable choice of small dimensional constant, and satisfies
\begin{equation}\label{ell ex 4}
   | E_Qg_{j,v}(x,t)| \gtrsim \lambda^{-(d-1)/2} \chi_{T_{j,v}}(x,t).
\end{equation}
In particular, provided $C$ is chosen appropriately in the definition of $\mathcal{V}_{\ell}$, it follows that
\begin{equation}\label{ell ex 5}
    | E_Qg_{j}(x,t)| \gtrsim \lambda^{-(d-1)/2} \sum_{v \in \mathcal{V}_{\ell}} \chi_{T_{j,v}}(x,t).
\end{equation}
The tubes in each family $(T_{j,v})_{v \in \mathcal{V}_{\ell}}$ are pairwise disjoint and their union can be thought of as the intersection of a fixed (that is, independent of $j$) $\ell$-plane slab formed around $V_{\ell}$ of thickness $\lambda^{1/2}$ with $B(0,\lambda)$. More precisely, using the transversality condition \eqref{ell ex 1}, it is not difficult to show that
\begin{equation*}
    \Big|\bigcup_{v_1 \in \mathcal{V}_{\ell}} \cdots \bigcup_{v_{\ell} \in \mathcal{V}_{\ell}} \bigcap_{j=1}^{\ell} T_{j,v_j} \Big| \gtrsim \lambda^{(d+\ell)/2};
\end{equation*}
in particular, the left-hand set contains a union of roughly $\lambda^{\ell/2}$ disjoint balls in $\R^d$ of radius roughly $\lambda^{1/2}$.




\subsection{Tensored examples} To prove Proposition~\ref{sharp multilinear proposition}, the multilinear estimates are tested against examples formed by tensoring the hyperbolic and elliptic examples described above. To this end, fix $1 \leq \sigma \leq n - 1$ with $n-1-\sigma$ even and let
\begin{equation}\label{multi ex 1}
    Q(\omega) := \sum_{j=1}^{\frac{n-1-\sigma}{2}} \omega_{2j-1} \omega_{2j} + \frac{1}{2}\sum_{j=n-\sigma}^{n-1} \omega_j^2.
\end{equation}
be a quadratic form in $n-1$ variables of signature $\sigma$. The multilinear examples subsequently constructed will prove the sharpness of Conjecture \ref{k-linear conjecture} when tested against the extension operator $E_Q$, irrespective of the level $k$ of multilinearity.

Fix $d$ satisfying 
\begin{equation}\label{multi ex 2}
  d \quad \textrm{odd,} \qquad  1 \leq d \leq n- \sigma
\end{equation}
and split the variables $\omega$ and $x$ by writing 
\begin{equation*}
    \omega = (\omega', \omega''), \quad x = (x',x'') \in \R^{d-1} \times \R^{n-d}.
\end{equation*}
The quadratic form is decomposed accordingly by writing 
\begin{equation*}
   Q(\omega) = Q'(\omega') + Q''(\omega''), \qquad Q'(\omega') := Q(\omega',0) \textrm{ and } Q''(\omega'') := Q(0, \omega'').
\end{equation*}
The condition \eqref{multi ex 2} implies that $Q'$ has zero signature, and therefore it makes sense to consider the hyperbolic examples $\mathfrak{h}(d,\ell)$ defined in \S\ref{hyp ex section} applied to this form. Note that, for $h \in C(\R^{d-1})$ and $g \in C(\R^{n-d})$, the tensor product $f := h \otimes g \in C(\R^{n-1})$ satisfies
\begin{equation*}
    E_Qf(x,t) = E_{Q'}h(x',t)E_{Q''}g(x'',t),
\end{equation*}
where $E_Q$, $E_{Q'}$ and $E_{Q''}$ are the extension operators associated to the respective quadratic forms, as defined in Example~\ref{prototypical example}.

Fix $1 \leq k \leq n$ and for $1 \leq \ell \leq k$ satisfying
\begin{equation}\label{multi ex 3}
   1 \leq \ell \leq \frac{d+1}{2} \quad \textrm{and} \quad k - \ell + 1 \leq n - d + 1
\end{equation}
and $\lambda \geq 1$ a large parameter let
\begin{equation*}
    \mathfrak{H}(\ell, d) = (h_1, \dots, h_{\ell}), \qquad \mathfrak{G}(k - \ell + 1, n -d +1) = (g_1, \dots, g_{k - \ell + 1} )
\end{equation*}
be hyperbolic and elliptic examples as defined above. For every level of multilinearity $k$, appropriate $d$ and $\ell$ will be chosen so that tensor products of functions from $\mathfrak{H}(\ell,d)$ and $\mathfrak{G}(k-\ell+1,n-d+1)$ demonstrate the sharpness of Conjecture \ref{k-linear conjecture} for this $k$. The constraints on the parameters in \eqref{multi ex 3} are important:
\begin{itemize}
    \item The first constraint is required in order to carry out the construction of the hyperbolic example $\mathfrak{H}(\ell, d)$ from \S\ref{hyp ex section}. Combined with \eqref{multi ex 2}, it implies that $ \ell \leq (n - \sigma - 1)/2 + 1$, which corresponds to the fact that maximal linear subspaces contained in the graph $\Sigma_Q$ of the form \eqref{multi ex 1} have dimension $(n-\sigma-1)/2$. Furthermore, this constraint will account for the transition in the numerology of Proposition~\ref{sharp multilinear proposition} at $k = (n-\sigma + 1)/2$. 
    \item The second constraint is required in order to carry out the construction of the elliptic example $\mathfrak{G}(k - \ell + 1, n -d +1)$ from \S\ref{ell ex section}. This constraint will account for the transition in the numerology of Proposition~\ref{sharp multilinear proposition} at $k = (n+\sigma + 1)/2$. 
\end{itemize} 

Define $k$ functions
\begin{align*}
    \mathfrak{h}_i &:= h_i \otimes g_1 \colon \R^{n-1} \to \C \qquad \textrm{for $1 \leq i \leq \ell$,} \\
    \mathfrak{g}_j &:= h_1 \otimes g_j \colon \R^{n-1} \to \C \qquad \textrm{for $2 \leq j \leq k - \ell + 1$.}
\end{align*}
In order to apply these examples in the proof of Proposition~\ref{sharp multilinear proposition}, the supports of the $\mathfrak{h}_i$ and $\mathfrak{g}_j$ functions must satisfy the transversality hypothesis. Since the supports of these functions are well-separated, it suffices to check the transversality condition at the centres of the supports only. Given $\omega = (\omega_1, \dots, \omega_{(d-1)/2}) \in \R^{(d-1)/2}$, let $\uparrow \!\!\omega \in \R^{d-1}$ denote the vector 
\begin{equation*}
    \uparrow \!\!\omega := (\omega_1, 0, \omega_2, 0, \dots, \omega_{(d-1)/2}, 0)
\end{equation*}
and note that
\begin{itemize}
    \item $\supp \mathfrak{h}_i$ is centred around $(\uparrow \!\!a_{i-1}, 0)^{\top} \in \R^{d-1} \times \R^{n-d}$,
    \item $\supp \mathfrak{g}_j$ is centred around $(0, b_{j-1})^{\top} \in \R^{d-1} \times \R^{n-d}$.
\end{itemize}
Computing the values of the Gauss map applied to these vectors, forming the relevant matrix and rearranging the rows, it suffices to show that the $n \times k$ matrix\footnote{The numbers outside the matrix represent the numbers of columns or rows in each block.}
\begin{equation*}
\begin{blockarray}{cccc}
    \begin{block}{[ccc]c}
    \mathbf{0} & \mathbf{0} & \mathbf{0} & {\scriptstyle \color{blue} \frac{d-1}{2}} \\
    \mathbf{0} & \mathbf{A} & \mathbf{0} & {\scriptstyle \color{blue} \frac{d-1}{2}} \\
    \mathbf{0} & \mathbf{0} & \mathbf{B} & {\scriptstyle \color{blue} n - d} \\
    1 & \mathbf{1} & \mathbf{1} & {\scriptstyle \color{blue}  1 } \\
\end{block}
{\scriptstyle \color{blue}  1} & {\scriptstyle \color{blue}  \ell - 1} & {\scriptstyle \color{blue}  k - \ell} & 
\end{blockarray}    
\end{equation*}
has full rank, where
\begin{equation*}
    \mathbf{A} \in \mathrm{Mat}\big(\tfrac{d-1}{2},\ell-1\big) \qquad \textrm{and} \qquad \mathbf{B} \in \mathrm{Mat}(n-d,k - \ell)
\end{equation*}
are the matrices whose columns are formed by the vectors $(-a_1, \dots, -a_{l-1})$ and $(-b_1, \dots, -b_{k-\ell})$, respectively. The desired rank condition is immediate from the choices of $a_i$ and $b_j$ and, in particular, \eqref{hyp ex 1} and \eqref{ell ex 1}. 

For now, suppose that the $k$-linear inequality 
\begin{equation}\label{multi ex 4}
\Big\|\prod_{i=1}^{\ell}|E_Q\mathfrak{h}_i|^{1/k}\prod_{j=2}^{k-\ell+1}|E_Q\mathfrak{g}_i|^{1/k}\Big\|_{L^p(B(0,\lambda))} \lesssim \prod_{i=1}^{\ell}\|\mathfrak{h}_i\|_{L^2(\R^{n-1})}^{1/k}\prod_{j=2}^{k-\ell+1}\|\mathfrak{g}_i\|_{L^2(\R^{n-1})}^{1/k}
\end{equation}
holds uniformly in $\lambda$. Presently, it is shown that, for appropriately chosen $d$, this forces
\begin{equation}\label{multi ex 5}
  p \geq q(n, k,\ell) \quad \textrm{where} \quad q(n,k,\ell) := 2 \cdot \frac{n + k - \ell + 1}{n + k - \ell -1}. 
\end{equation}
Plugging the optimal values of $\ell$ into the formula for $q(n, k,\ell)$ yields the desired range of $p$ described in Proposition~\ref{sharp multilinear proposition}. In particular, to maximise $q(n,k,\ell)$ one should choose $\ell$ as large as possible, under the condition that \eqref{multi ex 2} and \eqref{multi ex 3} should hold for some $d$. The correct choices of $\ell$ and $d$, which depend on the $k$ regime, are tabulated in Figure~\ref{l d table}. 

\begin{figure}
    \centering
       {\tabulinesep=1.2mm
     \begin{tabu}{|[1.5pt]c|[1.5pt]c|[1.5pt]c|[1.5pt]c|[1.5pt]}
    \tabucline[1.5pt]{-}
    \rowcolor{gray!25}
    $k$ range  & $\ell$ & $d$ & $q(n,k,\ell)$ \\
    \tabucline[1.5pt]{-}   
    $  \displaystyle 1 \leq k \leq \frac{n-\sigma+1}{2}$     & $k$ & $n- \sigma$ & $\displaystyle 2 \cdot \frac{n+1}{n-1}$ \\
     \hline     
    $\displaystyle \frac{n-\sigma+1}{2} \leq k \leq \frac{n+1 + \sigma}{2}$     &  $\displaystyle \frac{n-\sigma+1}{2}$ & $n- \sigma$ & $\displaystyle 2 \cdot \frac{n+2k + \sigma + 1}{n + 2k + \sigma -3}$ \\  
     \hline     
    $\displaystyle \frac{n+1 + \sigma}{2} \leq k \leq n$     & $n-k+1$ & $2n - 2k + 1$ & $\displaystyle 2 \cdot \frac{k}{k-1}$ \\
     \tabucline[1.5pt]{-}
    \end{tabu}}
    \caption{The value of $q(n,k,\ell)$ is obtained by substituting the corresponding $\ell$ value into the formula in \eqref{multi ex 5}. In all cases, $\ell$ and $d$ are chosen so as to satisfy \eqref{multi ex 2} and \eqref{multi ex 3}.}  
    \label{l d table}
\end{figure}

The first step is to obtain a lower bound for the expression on the left-hand side of \eqref{multi ex 4}. One may write the funtion appearing in the $p$-norm as a product of two functions 
\begin{align*}
 H(x',t) &:= |E_{Q'}h_1(x',t)|^{(k-\ell+1)/k} \prod_{i = 2}^{\ell} |E_{Q'}h_i(x',t)|^{1/k} \qquad \textrm{for $(x',t) \in \R^d$},\\
 G(x'',t) &:= |E_{Q''}g_1(x'',t)|^{\ell/k} \prod_{j = 2}^{k-\ell+1} |E_{Q''}g_j(x'',t)|^{1/k}\qquad \textrm{for $(x'',t) \in \R^{n-d+1}$}.
\end{align*}
Apply \eqref{hyp ex 3} at multilinearity $\ell$ and dimension $d$ to each factor in $H$ to deduce that
\begin{equation}\label{multi ex 6}
    H(x',t) \gtrsim \lambda^{-(d-1)/2} \chi_{\Pi_d(\lambda)}(x',t)
\end{equation}
On the other hand, apply \eqref{ell ex 4} at multilinearity $k - \ell + 1$ and dimension $n - d +1$ to each factor in $G$ to deduce that 
\begin{equation}\label{multi ex 7}
    G(x'',t) \gtrsim \lambda^{-(n-d)/2} \prod_{j=1}^{k-\ell+1} \sum_{v \in \mathcal{V}_{\ell}} \chi_{T_{j,v}}(x'',t),
\end{equation}
using the fact that the tubes $T_{j, v}$ are pairwise disjoint as $v$ varies over $\mathcal{V}_\ell$. Combining these observations,
\begin{equation}\label{multi ex 8}
    \Big\|\prod_{i=1}^{\ell}|E_{Q}\mathfrak{h}_i|^{1/k}\prod_{j=2}^{k-\ell+1}|E_{Q}\mathfrak{g}_i|^{1/k}\Big\|_{L^p(B(0,\lambda))} \gtrsim \lambda^{-(n-1)/2}\lambda^{(d-1)/2p+(n+k - d - \ell +2)/2p}.
\end{equation}
where:
\begin{itemize}
    \item the $\lambda^{-(n-1)/2}$ factor is the product of the coefficients from \eqref{multi ex 6} and \eqref{multi ex 7},
    \item the $\lambda^{(d-1)/2p}$ factor corresponds to the $L^p_{x'}$-norm of the characteristic function in \eqref{multi ex 6},
    \item the $\lambda^{(n+k - d - \ell +2)/2p}$ factor arises from \eqref{multi ex 7} owing to \eqref{ell ex 5}.
    \end{itemize}  

The right-hand side of \eqref{multi ex 4} is now bounded from above. In particular, by exploiting the tensor structure and applying the bounds \eqref{hyp ex 2} and \eqref{ell ex 3}, 
\begin{align}
\nonumber
    \prod_{i=1}^{\ell}\|\mathfrak{h}_i\|_2^{1/k}\prod_{j=2}^{k-\ell+1}\|\mathfrak{g}_i\|_2^{1/k} &= 
    \|h_1\|_2^{(k-\ell+1)/k}\Big(\prod_{i=2}^{\ell}\|h_i\|_2^{1/k}\Big)\|g_1\|_2^{\ell/k}\Big(\prod_{j=2}^{k - \ell + 1}\|g_j\|_2^{1/k}\Big) \\
    \label{multi ex 9}
    &\lesssim \lambda^{-(d-1)/4} \lambda^{-(n-k-d+\ell)/4}.
\end{align}
Note that, as before, \eqref{hyp ex 2} is applied at multilinearity $\ell$ and dimension $d$ whilst \eqref{ell ex 3} is applied with multilinear $k - \ell + 1$ and dimension $n - d +1$.

Plugging \eqref{multi ex 9} and \eqref{multi ex 8} into \eqref{multi ex 4} one concludes that
\begin{equation*}
\lambda^{-(n-1)/2} \lambda^{(n+k - \ell +1)/2p}  \lesssim \lambda^{-(n-1)/4}\lambda^{(k-\ell)/4}.
\end{equation*}
Since the inequality is assumed to hold for all large $\lambda$, this forces the condition described in \eqref{multi ex 5}.




\section{Proof of Theorems~\ref{main theorem} and~\ref{k-broad theorem}: Preliminaries}\label{preliminaries section}




\subsection{Overview} The remainder of the article deals with the proof of the $k$-broad estimates from Thereom ~\ref{k-broad theorem} and the passage from $k$-broad to linear estimates used to establish Theorem~\ref{main theorem}. In this section a variety of definitions and basic results are recalled from the literature (primarily \cite{Guth2018} and \cite{GHI2019}), which will be used throughout the remainder of the paper. In particular:
\begin{itemize}
    \item In \S\ref{basic geometry section} the underlying geometry of H\"ormander-type operators is discussed.
      \item In \S\ref{reductions section} the notation of a \textit{reduced phase} is introduced, and various technical reductions are described.
    \item In \S\ref{wave packet section} the wave packet decomposition for H\"ormander-type operators is recounted. 
\end{itemize}
The treatment here is rather brief and readers new to these concepts are encouraged to consult \cite{Guth2018} or \cite{GHI2019} for further details.




\subsection{Variable coefficient operators: basic geometry}\label{basic geometry section} Consider a smooth phase function $\phi \colon B^n \times B^{n-1} \to \R$ satisfying H1) and H2) from the introduction. Fixing $\bar{x} \in B^n$, the condition H1) implies that the mapping
\begin{equation*}
    \Sigma_{\bar{x}} := \big\{ \partial_x \phi(\bar{x};\omega) : (\bar{x};\omega) \in \supp a\}
\end{equation*}
is a (compact piece of) a smooth hypersurface in $\R^n$. Furthermore, the condition H2) implies that for each $\bar{x}$ the corresponding hypersurface has non-vanishing Gaussian curvature. After further localisation and a suitable coordinate transformation, the condition H1) ensures the existence of a local diffeomorphism $\Psi_{\bar{x}}$ on $\R^{n-1}$ such that
\begin{equation*}
    \partial_{x'}\phi(\bar{x};\Psi_{\bar{x}}(u)) = u \qquad \textrm{for all $u \in \mathrm{Domain}(\Psi_{\bar{x}})$}
\end{equation*}
In particular, the map $\Psi_{\bar{x}}$ corresponds to a graph reparametrisation of the hypersurface $\Sigma_{\bar{x}}$, with graphing function
\begin{equation*}
    h_{\bar{x}}(u) := \partial_{x_n}\phi(\bar{x};\Psi_{\bar{x}}(u)).
\end{equation*}
Throughout the remainder of the paper, it is always assumed that any H\"ormander-type operator with phase $\phi$ is suitably localised and that coordinates are chosen so that the above functions are defined globally on the support of the amplitude.

In view of the rescaled phase and amplitude functions appearing in the definition of $T^{\lambda}$, given $\lambda \geq 1$ and $\bar{x} \in B(0,\lambda)$ define $\Sigma_{\bar{x}}^{\lambda} := \Sigma_{\bar{x}/\lambda}$,  $\Psi_{\bar{x}}^{\lambda} := \Psi_{\bar{x}/\lambda}$ and $h_{\bar{x}}^{\lambda} := h_{\bar{x}/\lambda}$. Similarly, define the rescaled generalised Gauss map 
\begin{equation*}
G^{\lambda}(x; \omega) := G(x/\lambda;\omega) \qquad \textrm{for $(x;\omega) \in \mathrm{supp}\, a^{\lambda}$},
\end{equation*}
taking $G$ to be as defined in condition H2) from the introduction. Since the mapping $\Psi_{\bar{x}}^{\lambda}$ corresponds only to a change of coordinates, it follows that $G^{\lambda}(\bar{x};\omega)$ is parallel to the vector 
\begin{equation*}
    \begin{pmatrix}
    -\partial_u h_{\bar{x}}^{\lambda}(u) \\
    1
    \end{pmatrix}
\end{equation*}
for $u$ satisfying $\Psi_{\bar{x}}^{\lambda}(u) = \omega$. 




\subsection{Reductions}\label{reductions section}
To prove Theorem~\ref{main theorem} for all H\"ormander-type operators with phases of a given signature $\sigma$, one needs only to consider operators which are perturbations of the prototypical extension operators $E_{\sigma}$ from Example~\ref{prototypical example}. In particular, recall that the H\"ormander-type operators under consideration are those of the form
\begin{equation*}
    T^\lambda f(x)=\int_{B^{n-1}}e^{2\pi i \phi^\lambda(x;\omega)}a^\lambda(x;\omega)f(\omega){\rm d}\omega,
\end{equation*}
where the phase $\phi$ satisfies the general conditions H1) and H2). For any $0\leq \sigma\leq n-1$ with $n-1-\sigma$ even, let $\mathrm{I}_{n-1, \sigma}$ denote the $(n-1)\times(n-1)$ matrix of signature $\sigma$ from Example~\ref{prototypical example}.

\begin{lemma}\label{reduction lemma} Let $0\leq \sigma\leq n$ with $n-1-\sigma$ even and $\varepsilon>0$. To prove Theorem~\ref{main theorem} for this fixed $\varepsilon > 0$ for all H\"ormander-type operators with phase function of signature $\sigma$, it suffices to consider the case where the amplitude $a$ is supported on $X \times \Omega$, where $X := X' \times X_n$ and $X' \subset B^{n-1}$, $X_n \subset B^1$ and $\Omega \subset B^{n-1}$ are small balls centred at 0 upon which the phase $\phi$ has the form
\begin{equation*}
\phi(x; \omega) = \langle x', \omega \rangle + x_nh(\omega) + \mathcal{E}(x;\omega).
\end{equation*}
Here $h$ and $\mathcal{E}$ are smooth functions, $h$ is quadratic in $\omega$ and $\mathcal{E}$ is quadratic in $x$ and $\omega$.\footnote{Explicitly, if $(\alpha, \beta) \in \N_0 \times \N_0^{n-1}$ is a pair of multi-indices, then:
\begin{enumerate}[i)]
\item $\partial^{\beta}_{\omega}h(0) = \partial^{\beta}_{\omega}\partial^{\alpha}_x\mathcal{E}(x;0)=0$ whenever $x \in X$ and $|\beta| \leq 1$;
\item  $\partial^{\beta}_{\omega}\partial^{\alpha}_x\mathcal{E}(0;\omega)=0$ whenever $\omega \in \Omega$ and $|\alpha| \leq 1$.
\end{enumerate}} Furthermore, letting $c_{\mathrm{ex}} > 0$ be a small constant, which may depend on the admissible parameters $n$, $p$ and $\varepsilon$, one may assume that the phase function $\phi$ satisfies
\begin{gather*}
\|\partial_{\omega x'}^2 \phi(x; \omega) - \mathrm{I}_{n-1}\|_{\mathrm{op}} < c_{\mathrm{ex}}, \quad |\partial_{\omega}\partial_{x_n} \phi(x; \omega)| < c_{\mathrm{ex}}, \\
\|\partial_{\omega \omega}^2\partial_{x_k} \phi(x; \omega) - \delta_{kn}\mathrm{I}_{n-1,\sigma}\|_{\mathrm{op}} < c_{\mathrm{ex}}
\end{gather*}
 for all $(x;\omega) \in X \times \Omega$ and $1\leq k \leq n$. In addition, 
 \begin{equation}\label{derivative reduction}
\|\partial_{\omega}^{\beta}\partial_x^{\alpha} \phi \|_{L^{\infty}(X \times \Omega)} < c_{\mathrm{ex}} \qquad \textrm{for $1 \leq |\alpha| \leq N_{\mathrm{ex}}$, $3 \leq |\beta| \leq N_{\mathrm{ex}}$.}
\end{equation}
for some large integer $N_{\mathrm{ex}} \in \N$, which can be chosen to depend on $n$, $p$ and $\varepsilon$. If $|\alpha| \geq 2$, then the lower bound on $|\beta|$ can be relaxed to 0 in \eqref{derivative reduction}. Finally, it may assumed that the amplitude $a$ satisfies
\begin{equation*}
\|\partial_{\omega}^{\beta} \partial_x^{\alpha}a\|_{L^{\infty}(X \times \Omega)} \lesssim_{\beta} 1 \qquad \textrm{for all $0 \leq |\alpha|,|\beta| \leq N_{\mathrm{ex}}$.}
\end{equation*}
\end{lemma}

The proof of Lemma~\ref{reduction lemma} is a simple adaptation of the proofs of Lemma 4.1 and Lemma 4.3 in \cite{GHI2019} (which describe the case $\sigma=n-1$) and is thus omitted here. 

\begin{definition} Henceforth $c_{\mathrm{ex}} > 0$ and $N_{\mathrm{ex}} \in \N$ are assumed to be fixed constants (which are allowed to depend only on admissible parameters), chosen to satisfy the requirements of the forthcoming arguments. A phase of signature $\sigma$ satisfying the properties of Lemma~\ref{reduction lemma} for this choice of $\sigma$, $c_{\mathrm{ex}}$ and $N_{\mathrm{ex}}$ is said to be \emph{reduced}.
\end{definition}

\subsection{Wave packet decomposition}\label{wave packet section} The wave packet decomposition from \cite{GHI2019} is now reviewed and some notation is established. All statements in this subsection are proved in \cite{GHI2019}.

Throughout the following sections $\varepsilon > 0$ is a fixed small parameter and $\delta > 0$ is a tiny number satisfying\footnote{For $A, B \geq 0$ the notation $A \ll B$ or $B \gg A$ is used to denote that $A$ is `much smaller' than $B$; a more precise interpretation of this is that $A \leq C_{\varepsilon}^{-1}B$ for some constant $C_{\varepsilon} \geq 1$ which can be chosen to be large depending on $n$ and $\varepsilon$.} $\delta \ll \varepsilon$ and $\delta \sim_{\varepsilon} 1$. For any spatial parameter satisfying $1\ll R\ll\lambda$, a wave packet decomposition at scale $R$ is carried out as follows. Cover $B^{n-1}$ by finitely-overlapping balls $\theta$ of radius $R^{-1/2}$ and let $\psi_{\theta}$ be a smooth partition of unity adapted to this cover. These $\theta$ are referred to as $R^{-1/2}$-\emph{caps}. Cover $\R^{n-1}$ by finitely-overlapping balls of radius $CR^{(1+\delta)/2}$ centred on points belonging to the lattice $R^{(1+\delta)/2}\Z^{n-1}$. By Poisson summation one may find a bump function adapted to $B(0, R^{(1+\delta)/2})$ so that the functions $\eta_v(z) := \eta(z- v)$ for $v \in R^{(1+\delta)/2}\Z^{n-1}$ form a partition of unity for this cover. Let $\mathbb{T}$ denote the collection of all pairs $(\theta, v)$. Thus, for $f \colon \R^{n-1} \to \C$ with support in $B^{n-1}$ and belonging to some suitable \emph{a priori} class one has
\begin{equation*}
f = \sum_{(\theta, v) \in \mathbb{T}} (\eta_v(\psi_{\theta}f)\;\widecheck{}\;)\;\widehat{}\; = \sum_{(\theta, v) \in \mathbb{T}} \hat{\eta}_v \ast(\psi_{\theta}f).
\end{equation*} 
For each $R^{-1/2}$-cap $\theta$ let $\omega_{\theta} \in B^{n-1}$ denote its centre. Choose a real-valued smooth function $\tilde{\psi}$ so that the function $\tilde{\psi}_{\theta}(\omega) := \tilde{\psi}(R^{1/2}(\omega - \omega_{\theta}))$ is supported in $\theta$ and $\tilde{\psi}_{\theta}(\omega) = 1$ whenever $\omega$ belongs to a $cR^{-1/2}$ neighbourhood of the support of $\psi_{\theta}$ for some small constant $c > 0$. Finally, define
\begin{equation*}
f_{\theta, v} := \tilde{\psi}_{\theta} \cdot [\hat{\eta}_v \ast (\psi_{\theta}f)]. 
\end{equation*}
It is not difficult to show
\begin{equation*}
\|f - \sum_{(\theta, v) \in \T}f_{\theta, v}\|_{L^{\infty}(\R^{n-1})} \leq \mathrm{RapDec}(R)\|f\|_{L^2(B^{n-1})},
\end{equation*}
whilst the functions $f_{\theta, v}$ are also almost orthogonal: if $\mathbb{S} \subseteq \T$, then
\begin{equation*}
\big\| \sum_{(\theta, v) \in \mathbb{S}} f_{\theta, v} \big\|_{L^2(\R^{n-1})}^2 \sim  \sum_{(\theta, v) \in \mathbb{S}} \|f_{\theta, v} \|_{L^2(\R^{n-1})}^2.
\end{equation*}
A precise description of the rapidly decaying term ${\rm RapDec}(R)$, frequently used in forthcoming sections, is inserted here.
\begin{definition}  The notation $ \mathrm{RapDec}(R)$ is used to denote any quantity $C_R$ which is rapidly decaying in $R$. More precisely, $C_R = \mathrm{RapDec}(R)$ if 
\begin{equation*}
    |C_R| \lesssim_{\varepsilon} R^{-N} \qquad \textrm{for all $N \leq \sqrt{N_{\mathrm{ex}}}$,}
\end{equation*}
where $N_{\mathrm{ex}}$ is the large integer appearing in the definition of reduced phase from \S\ref{reductions section}. Note that $N_{\mathrm{ex}}$ may be chosen as large as desired, under the condition that it depends only on $n$ and $\varepsilon$.
\end{definition}
Let $T^{\lambda}$ be an operator with reduced phase $\phi$ and amplitude $a$ supported in $X\times \Omega$ as in Lemma~\ref{reduction lemma}. For $(\theta, v)\in \T$, within $B(0,R)$ the function $T^\lambda f_{\theta,v}$ is essentially supported inside a curved $R^{1/2+\delta}$-tube $T_{\theta,v}$ determined by $\phi$, $\theta$ and $v$. More precisely, there exists a curve 
\begin{equation*}
   \Gamma_{\theta, v}^{\lambda} = (\gamma_{\theta,v}^{\lambda}(\,\cdot\,),\,\cdot\,) \colon I_{\theta,v}^{\lambda} \to \R^n,
\end{equation*}
for some $I_{\theta,v}^\lambda\subset [-\lambda,\lambda]$, that parametrises the set
\begin{equation*}
    \{x\in X:\partial_\omega\phi(x;\omega_\theta)=v\}.
\end{equation*}
This curve $\Gamma_{\theta,v}^\lambda$ forms the core of the tube $T_{\theta,v}$. In particular, for
\begin{equation*}
T_{\theta,v} := \big\{ (x',x_n) \in B(0,R) : x_n \in I_{\theta, v}^{\lambda} \textrm{ and } |x' - \gamma^\lambda_{\theta, v}(x_n)| \leq R^{1/2 + \delta} \big\}
\end{equation*}
the following concentration estimate holds.

\begin{lemma} If $1 \ll R \ll \lambda$ and $x \in B(0,R) \setminus T_{\theta, v}$, then 
\begin{equation*}
|T^\lambda f_{\theta, v}(x)| \leq (1 + R^{-1/2}|\partial_{\omega}\phi^{\lambda}(x;\omega_{\theta}) - v|)^{-(n+1)}\mathrm{RapDec}(R)\|f\|_{L^2(B^{n-1})}. 
\end{equation*}
\end{lemma}

The geometry of the core curve of $T_{\theta,v}$ is related to the generalised Gauss map $G^\lambda$ associated to the operator $T^\lambda$: the tangent line $T_{\Gamma^{\lambda}_{\theta,v}(t)}\Gamma^{\lambda}_{\theta,v}$ lies in the direction of the unit vector $G^{\lambda}(\Gamma^{\lambda}_{\theta,v}(t); \omega_{\theta})$ for all $t \in I^{\lambda}_{\theta,v}$. For instance, if $\phi^{\lambda}(x; \omega)$ is of the form $\langle x', \omega \rangle + x_nh(\omega)$, giving rise to an extension operator, then the $T_{\theta,v}$ are straight tubes.




\section{Partial transverse equidistribution estimates}\label{trans eq sec}




\subsection{Overview} In this section the key tool required for the proof of Theorem~\ref{k-broad theorem} is introduced and proved. This is a `partial' transverse equidistribution estimate, which bounds the $L^2$ norm of $T^{\lambda}g$ under certain geometric hypotheses on the wave packets of $g$: see Lemma \ref{trans eq lem 1} below. This lemma generalises the transverse equidistribution estimates for the elliptic case in \cite{Guth2018} and \cite{GHI2019}. It is a key step in the argument where the signature $\mathrm{sgn}(\phi)$ plays a r\^ole. Indeed, once Lemma \ref{trans eq lem 1} is in place, the remainder of the proof of Theorem~\ref{k-broad theorem} follows as in the elliptic case, with only minor numerological changes, as discussed in the following section.




\subsection{Tangential wave packets and transverse equidistribution}\label{tang wp sec} Throughout this section let $T^{\lambda}$ be a H\"ormander-type operator with reduced phase $\phi$ of signature $\sigma$ and for some $1\ll R \ll \lambda$ define the (curved) tubes $T_{\theta,v}$ as in \S\ref{wave packet section}. Here a special situation is considered where $T^{\lambda}g$ is made up of a sum of wave packets which are tangential to some algebraic variety, in a sense described below. To begin, the relevant algebraic preliminaries are recounted. 

\begin{definition} Given any collection of polynomials $P_1, \dots, P_{n-m} \colon \R^n \to \R$ the common zero set
\begin{equation*}
    Z(P_1, \dots, P_{n-m}) := \big\{\,x \in \R^n \,:\, P_1(x) = \cdots = P_{n-m}(x) = 0\,\big\}
\end{equation*}
will be referred to as a \emph{variety}.\footnote{The ideal generated by the $P_j$ is \emph{not} required to be irreducible.} Given a variety $Z = Z(P_1, \dots, P_{n-m})$, define its \emph{(maximum) degree} to be the number 
\begin{equation*}
\overline{\deg}\,Z := \max \{\deg P_1, \dots, \deg P_{n-m} \}.
\end{equation*}
\end{definition}
It will often be convenient to work with varieties which satisfy the additional property that
\begin{equation}\label{non singular variety}
\bigwedge_{j=1}^{n-m}\nabla P_j(z) \neq 0 \qquad \textrm{for all $z \in Z = Z(P_1, \dots, P_{n-m})$.}
\end{equation}
In this case the zero set forms a smooth $m$-dimensional submanifold of $\R^n$ with a (classical) tangent space $T_zZ$ at every point $z \in Z$. A variety $Z$ which satisfies \eqref{non singular variety} is said to be an \emph{$m$-dimensional transverse complete intersection}.

Let $\delta_m$ denote a small parameter satisfying $0 < \delta \ll \delta_m \ll 1$ (here $\delta$ is the same parameter as that which appears in the definition of the wave packets).

\begin{definition} Suppose $Z = Z(P_1, \dots, P_{n-m})$ is a transverse complete intersection. A tube $T_{\theta, v}$ is $R^{-1/2 + \delta_m}$-tangent to $Z$ in $B(0,R)$ if
\begin{equation*}
T_{\theta, v} \subseteq N_{R^{1/2 + \delta_m}}(Z)
\end{equation*}
and
\begin{equation*}
\angle(G^{\lambda}(x; \omega_{\theta}), T_zZ) \leq \bar{c}_{\mathrm{tang}} R^{-1/2 + \delta_m}
\end{equation*}
for any $x \in T_{\theta, v}$ and $z \in Z \cap B(0,2R)$ with $|x - z| \leq\bar{C}_{\mathrm{tang}} R^{1/2 + \delta_m}$.
\end{definition}

Here $\bar{c}_{\mathrm{tang}} > 0$ (respectively, $\bar{C}_{\mathrm{tang}} \geq 1$) is a dimensional constant, chosen to be sufficiently small (respectively, large) for the purposes of the following arguments. 

\begin{definition} If $\mathbb{S} \subseteq \T$, then $f$ is said to be concentrated on wave packets from $\mathbb{S}$ if
\begin{equation*}
f = \sum_{(\theta, v) \in \mathbb{S}} f_{\theta, v} + \mathrm{RapDec}(R)\|f\|_{L^2(B^{n-1})}.
\end{equation*}
\end{definition}

One wishes to study functions concentrated on wave packets from the collection
\begin{equation*}
\T_Z := \big\{(\theta, v) \in \T : T_{\theta, v} \textrm{ is $R^{-1/2 + \delta_m}$-tangent to $Z$ in $B(0,R)$}\big\}.
\end{equation*}

Let $B \subseteq \R^n$ be a fixed ball of radius $R^{1/2 + \delta_m}$ with centre $\bar{x} \in B(0,R)$. Throughout this section the analysis will be essentially confined to a spatially localised operator $\eta_B\cdot T^{\lambda}g $ where $\eta_B$ is a suitable choice of Schwartz function concentrated on $B$. It is remarked that, for any $(\theta, v) \in \T$, a stationary phase argument shows that the Fourier transform of $\eta_B\cdot T^{\lambda}g_{\theta, v}$ is concentrated near the surface
\begin{equation}\label{Sigma def}
\Sigma := \{ \Sigma(\omega) : \omega \in \Omega \} \quad \textrm{where} \quad \Sigma(\omega) := \partial_{x}\phi^{\lambda}(\bar{x};\omega).
\end{equation}
Now consider the refined set of wave packets 
\begin{equation*}
\T_{Z,B} := \big\{(\theta, v) \in \T_Z : T_{\theta, v} \cap B \neq \emptyset \big\}.
\end{equation*}

 Let $R^{1/2} < \rho \ll R$ and throughout this subsection let $\tau \subset \R^{n-1}$ be a fixed cap of radius $O(\rho^{-1/2 + \delta_m})$ centred at a point in $B^{n-1}$. Now define
\begin{equation*}
\T_{Z, B,\tau} := \big\{(\theta, v) \in \T_Z : \theta \cap \tau \neq \emptyset \textrm{ and }  T_{\theta, v} \cap B \neq \emptyset  \big\}.
\end{equation*}
For $1 \leq m \leq n$ denote
\begin{equation*}
   \mu(n,\sigma,m) :=  \max\Big\{n -2 m +1, \frac{n+1+\sigma}{2} - m, 0 \Big\}
\end{equation*}
so that
\begin{equation*}
    \mu(n,\sigma,m) = \left\{\begin{array}{lll}
        n - 2m + 1&  \textrm{if} & 1 \leq m \leq \frac{n-\sigma +1}{2}\\[3pt]
         \frac{n+\sigma+1}{2} -m & \textrm{if} & \frac{n-\sigma + 1}{2} \leq m \leq \frac{n + \sigma +1}{2} \\[3pt]
        0 & \textrm{if} & \frac{n + \sigma +1}{2} \leq m \leq n
    \end{array} \right. .
\end{equation*}

With these definitions, the key partial transverse equidistribution result is as follows. 

\begin{lemma}\label{trans eq lem 1} With the above setup, if $\dim Z = m$ and $\overline{\deg}\,Z \lesssim_{\varepsilon} 1$ and $g$ is concentrated on wave packets from $\T_{Z, B,\tau}$, then
\begin{equation*}
\int_{N_{\rho^{1/2 + \delta_m}}(Z) \cap B} |T^{\lambda}g|^2 \lesssim_{\varepsilon,\delta} R^{1/2+O(\delta_m)}(\rho/R)^{\mu(n,\sigma,m)/2}\|g\|^2_{L^2(B^{n-1})}.
\end{equation*}
\end{lemma}

The remainder of the section is dedicated to the proof of this lemma. For a discussion of the philosophy and heuristics behind estimates of this kind, see \cite[\S6]{Guth2018} or \cite[\S8]{GHI2019}, as well as \S\ref{new sec}. It is noted that in the maximum signature case $\mu(n,n-1,m) = n-m$ for all $1 \leq m \leq n$, so this lemma recovers the previous elliptic case result in \cite[Lemma 8.4]{GHI2019} (see also \cite[Lemma 6.2]{Guth2018}). On the other hand, in the range $\frac{n + \sigma +1}{2} \leq m \leq n$ where $\mu(n,\sigma,m) = 0$ the result follows from a classical $L^2$ bound of H\"ormander and does not depend on any geometric considerations regarding the wave packets.




\subsection{Wave packets tangential to linear subspaces} Here, as a step towards  Lemma~\ref{trans eq lem 1}, transverse equidistribution estimates are proven for functions concentrated on wave packets tangential to some fixed linear subspace $V \subseteq \R^n$. As before, let $B$ be a ball of radius $R^{1/2 + \delta_m}$ with centre $\bar{x} \in \R^n$ and define
\begin{equation*}
\T_{V,B} := \big\{(\theta, v) \in \T : \angle(G^{\lambda}(\bar{x}, \omega_{\theta}), V) \lesssim R^{-1/2 + \delta_m}  \textrm{ and }  T_{\theta, v} \cap B \neq \emptyset   \big\}.
\end{equation*}
Let $R^{1/2} < \rho < R$ and for $\tau \subset \R^{n-1}$ a ball of radius $O(\rho^{-1/2+ \delta_m})$ centred at a point in $B^{n-1}$ define
\begin{equation*}
\T_{V, B,\tau} := \big\{(\theta, v) \in \T_{V,B} : \theta \cap (\tfrac{1}{10}\cdot \tau) \neq \emptyset\big\}
\end{equation*}
where $(\tfrac{1}{10}\cdot \tau)$ is the cap concentric to $\tau$ but with $1/10$th of the radius.

The key estimate is the following.

\begin{lemma}\label{trans eq lem 2} If $V \subseteq \R^n$ is a linear subspace, then there exists a linear subspace $V'$ with the following properties:
\begin{enumerate}[1)]
\item $\mu(n,\sigma,\dim V) \leq \dim V' \leq n-\dim V$.
\item $V, V'$ are quantitatively transverse in the sense that there exists a uniform constant $c_{\mathrm{trans}} > 0$ such that
\begin{equation*}
 \angle(v, v') \geq 2c_{\mathrm{trans}} \qquad \textrm{for all non-zero vectors $v \in V$ and $v' \in V'$.}
\end{equation*}
\item If $g$ is concentrated on wave packets from $\T_{V, B,\tau} $, $\Pi$ is any plane parallel to $V'$ and $x_0 \in \Pi \cap B$, then the inequality
\begin{align*}
\int\displaylimits_{\Pi \cap B(x_0, \rho^{1/2 + \delta_m})}\!\!\!\!\!\!\!\!\!\!\!\! |T^{\lambda}g|^2 \lesssim_{\delta} R^{O(\delta_m)} (\rho/R)^{\dim V'/2} \|g\|_{L^2(B^{n-1})}^{2\delta/(1+\delta)} \big(\int\displaylimits_{\Pi \cap 2B} |T^{\lambda}g|^2 \big)^{1/(1+\delta)}
\end{align*}
holds up to the inclusion of a $\mathrm{RapDec}(R)\|g\|_{L^2(B^{n-1})}$ term on the right-hand side. 
\end{enumerate}
\end{lemma}

\begin{proof}[Proof (of Lemma~\ref{trans eq lem 2})] Many of the steps of the proof are similar to the proof of Lemma 8.7 from \cite{GHI2019}, although the construction of $V'$ itself is different from that used in the positive-definite case. 


 
\subsection*{Constructing the subspace $V_{\mathrm{aux}}'$} The first step in the argument is to construct an auxiliary space $V_{\mathrm{aux}}'$; the desired subspace $V'$ is then obtained by rotating $V_{\mathrm{aux}}'$.

One may assume without loss of generality that  
\begin{equation}\label{TE2 1}
\overline{\angle}(V, e_n^{\perp}) := \max_{v \in V\cap S^{n-1}} \angle(v, e_n^{\perp}) \gtrsim 1
\end{equation}
since otherwise the family of tubes $\mathbb{T}_{V,B}$ is empty and there is nothing to prove. Consider the horizontal slice $V_{\mathrm{sl}} := \mathrm{proj}_{\R^{n-1}} (V \cap \R^{n-1} \times \{0\}) \subseteq \R^{n-1}$. The angle condition \eqref{TE2 1} ensures that $\dim V_{\mathrm{sl}} = \dim V -1$. Let $\widetilde{V}_{\mathrm{sl}}$ denote the preimage of $V_{\mathrm{sl}}$ (which also corresponds to the image) under the linear mapping induced by the matrix $
\mathrm{I}_{n-1, \sigma}$; recall, $
\mathrm{I}_{n-1, \sigma}$ is the matrix appearing in Example \ref{prototypical example} and in the definition of reduced form from \S\ref{reductions section}. The auxiliary space is defined to be
\begin{equation*}
    V_{\mathrm{aux}}' := V_{\mathrm{sl}}^{\perp} \cap \widetilde{V}_{\mathrm{sl}}^{\perp},
\end{equation*}
where the orthogonal complements are taken inside $\mathbb{R}^{n-1}$.
The following example partially motivates the above definition. 

\begin{example}\label{Gauss preimage example} Consider the prototypical case of the extension operator $E_{\sigma}$ from Example~\ref{prototypical example}. Here the unnormalised Gauss map $G_0$ is an affine map, and so
\begin{equation*}
    A_{\omega} := \{ \omega \in \R^{n-1} :  G_0(\omega) \in V \}
\end{equation*}
is an affine subspace. A simple computation shows that $A_{\omega}$ is parallel to $\widetilde{V}_{\mathrm{sl}}$. 
\end{example}




\subsection*{Dimension bounds for $V_{\mathrm{aux}}'$} The next step of the proof is to show that the auxiliary space satisfies the dimension bounds described in part 1) of the lemma. It is clear that $\dim V_{\mathrm{aux}}' \leq n - \dim V$ since $V_{\mathrm{aux}}' \subseteq V_{\mathrm{sl}}^{\perp}$ and the latter subspace has dimension equal to 
\begin{equation*}
 n - 1 - \dim V_{\mathrm{sl}} = n-1 - (\dim V - 1) =  n - \dim V. 
\end{equation*} 
It remains to show that $\dim V_{\mathrm{aux}}' \geq \mu(n,\sigma,\dim V)$. Since $V_{\mathrm{sl}}^{\perp} \cap \widetilde{V}_{\mathrm{sl}}^{\perp} = (V_{\mathrm{sl}} + \widetilde{V}_{\mathrm{sl}})^{\perp}$ and $\dim V_{\mathrm{sl}} = \dim \widetilde{V}_{\mathrm{sl}} = \dim V - 1$, it follows that
\begin{align*}
    \dim V_{\mathrm{aux}}' &= n - 1 - \dim \big( V_{\mathrm{sl}} +\widetilde{V}_{\mathrm{sl}}\big) \\
    &= n - 1 - \dim V_{\mathrm{sl}} - \dim \widetilde{V}_{\mathrm{sl}} + \dim V_{\mathrm{sl}} \cap \widetilde{V}_{\mathrm{sl}} \\
    &=n-2\dim V+1+\dim V_{\mathrm{sl}}\cap\widetilde{V}_{\mathrm{sl}},
\end{align*}
from which the estimate $\dim V'_{\mathrm{aux}}\geq n-2\dim V+1$ directly follows. It thus suffices to prove that $\dim V'_{\mathrm{aux}}\geq (n+\sigma+1)/2-\dim V$, or equivalently
\begin{equation*}
\mathrm{codim}\, V_{\mathrm{sl}} \cap \widetilde{V}_{\mathrm{sl}}\leq n-\dim V+\frac{n-1-\sigma}{2}.
\end{equation*}

Fix an orthonormal basis $\{N_1, \dots, N_{n - \dim V}\}$ for $V^{\perp}$ so that
\begin{equation*}
    V = \big\{ \xi \in \hat{\R}^n : \langle \xi, N_k \rangle = 0 \textrm{ for $1 \leq k \leq n - \dim V$} \big\}.
\end{equation*}
The angle condition \eqref{TE2 1} implies that $\{N_1', \dots, N_{n-\dim V}'\}$ is a linearly independent set of vectors, where $N_k = (N_k', N_{k,n}) \in \R^{n-1} \times \R$ and, clearly,\footnote{To establish the desired dimensional bounds, the only required property of the vectors $N_k'$ is that they form a basis of $V_{\mathrm{sl}}^\perp$, not that they arise from a basis for $V^\perp$ in the above manner. However, the vectors $N_k$ are introduced as they will be used in subsequent parts of the proof.}
\begin{equation*}
     V_{\mathrm{sl}} = \big\{ u \in \R^{n-1} : \langle u , N_k' \rangle = 0 \textrm{ for $1 \leq k \leq n - \dim V$} \big\}.
\end{equation*}
On the other hand, 
\begin{equation*}
    \widetilde{V}_{\mathrm{sl}} = \big\{ u \in \R^{n-1} : \langle u, \tilde{N}_k \rangle = 0 \textrm{ for $1 \leq k \leq n - \dim V$} \big\}
\end{equation*}
where the vectors $\tilde{N}_k := I_{n-1, \sigma}(N_k') \in \R^{n-1}$ satisfy
\begin{equation*}
    \tilde{N}_k = 
    \begin{pmatrix}
    N_{k, +}\\
    -N_{k, -}
    \end{pmatrix} \quad \textrm{for} \quad N_k' = \begin{pmatrix}N_{k,+} \\
    N_{k,-}
    \end{pmatrix}
    \in \R^{\sigma} \times \R^{n-1-\sigma}.
\end{equation*}

Combining the observations of the previous paragraph,
\begin{equation*}
     V_{\mathrm{sl}} \cap \widetilde{V}_{\mathrm{sl}} = \big\{ u \in \R^{n-1} : \langle u , N_k' \rangle = \langle u , \tilde{N}_k \rangle = 0 \textrm{ for $1 \leq k \leq n - \dim V$} \big\}
\end{equation*}
and, consequently, 
\begin{equation*}
 \mathrm{codim}\, V_{\mathrm{sl}} \cap \widetilde{V}_{\mathrm{sl}} =  \mathrm{rank}\, 
 \begin{pmatrix}
    N_1' & \dots & N_{n - \dim V}' & \tilde{N}_1 & \cdots & \tilde{N}_{n-\dim V}
    \end{pmatrix}.
\end{equation*}
Note that
\begin{equation*}
    \frac{1}{2}\cdot \big(N_k' - \tilde{N}_k \big) =
    \begin{pmatrix}
    0 \\
    N_{k,-}
    \end{pmatrix}
\end{equation*}
and, since matrix rank is preserved under elementary column operations, 
\begin{equation*}
   \mathrm{codim}\, V_{\mathrm{sl}} \cap \widetilde{V}_{\mathrm{sl}} =  \mathrm{rank}\, 
 \begin{pmatrix}
N_{1,+} & \dots & N_{n - \dim V,+} & 0 & \cdots & 0 \\
N_{1,-} & \dots & N_{n - \dim V,-} & N_{1,-} & \cdots & N_{n-\dim V,-}
    \end{pmatrix}.
\end{equation*}
The left $(n-1)\times (n-\dim V)$ block is made up of $n-\dim V$ linearly independent columns $N_1', \dots, N_{n-\dim V}'$. For the right-hand block, the number of linearly independent columns can be at most the number of non-zero rows, which is equal to $(n-1-\sigma)/2$. Altogether, this bounds the matrix rank above by
\begin{equation*}
    n - \dim V + \frac{n-1-\sigma}{2},
\end{equation*}
as desired.


 
\subsection*{Constructing the subspace $V'$} 

One may assume without loss of generality that $S_{\omega} \cap \tau \neq \emptyset$ where
\begin{equation*}
S_{\omega} := \big\{ \omega \in \Omega : G^{\lambda}(\bar{x}; \omega) \in V \big\},
\end{equation*}
since otherwise the family of tubes $\mathbb{T}_{V,B, \tau}$ is empty and there is nothing to prove. Recalling \eqref{TE2 1}, it follows that $S_{\omega}$ is a smooth surface in $\R^{n-1}$ of dimension $\dim V - 1$; indeed, this can be verified as a simple calculus exercise, but it is also treated explicitly as Claim 1 in the proof of Lemma 8.7 from \cite{GHI2019} (the claim is stated in the positive-definite case, but the argument does not depend on the signature). For notational convenience, write
\begin{equation}\label{convenient notation}
\Psi(u)  := \Psi_{\bar{x}}^{\lambda}(u) \qquad \textrm{and} \qquad   \bar{h}(u) := h_{\bar{x}}^{\lambda}(u) = \partial_{x_n} \phi^{\lambda}(\bar{x}; \Psi(u))
\end{equation} 
for the functions as defined in \S\ref{basic geometry section}.  Consider the surface
\begin{equation*}
 S_u := \Psi^{-1}(S_{\omega}) = \{u \in U : G_0^{\lambda}(\bar{x}; \Psi (u)) \in V\},
\end{equation*}
given by the diffeomorphic image of $S_{\omega}$ under the map $\Psi$. Fix some $u_0 \in S_u \cap \Psi^{-1}(\tau)$ and let $A_u$ denote the tangent plane to $S_u$ at $u_0$. Here, the tangent plane is interpreted as a $(\dim V - 1)$-dimensional affine subspace of $\R^{n-1}$ through $u_0$. Now define $A_{\xi} := A_u \times \R \subseteq \R^n$, so that $\dim A_{\xi} = \dim V$, and let $V_u$ and $V_{\xi}$ be the linear subspaces parallel to $A_u$ and $A_{\xi}$, respectively. 

The spaces $\widetilde{V}_{\mathrm{sl}} \subset \R^{n-1}$ and $V_u \subset \R^{n-1}$ both have dimension $\dim V - 1$. Moreover, the localisation to the cap $\tau$ and ball $B$ implies that $\widetilde{V}_{\mathrm{sl}}$ and $V_u$ are close to one another in the following sense.

\begin{claim} Let $c_{\mathrm{ex}}$ be the constant defined in \S\ref{reductions section}. Then 
\begin{equation*}
    \max_{v^* \in \widetilde{V}_{\mathrm{sl}} \cap S^{n-2}} \angle(v^*, V_u) = O(c_{\mathrm{ex}}).
\end{equation*}
\end{claim}

The proof of the claim is temporarily postponed. Assuming its validity, it follows that there exists a choice of $O_V \in \mathrm{SO}(n-1,\R)$ mapping $\widetilde{V}_{\mathrm{sl}}$ to $V_u$ which satisfies
\begin{equation*}
    \|O_V - I_{n-1}\|_{\mathrm{op}} = O(c_{\mathrm{ex}}). 
\end{equation*}
Indeed, if $\{v_1^*, \dots, v_{\dim V -1}^*\}$ is a choice of orthonormal basis for $\widetilde{V}_{\mathrm{sl}}$, then the claim implies that there exists a basis $\{v_1, \dots, v_{\dim V - 1}\}$ for $V_u$ satisfying 
\begin{equation*}
    \angle(v_k^*, v_k) = O(c_{\mathrm{ex}}) \qquad \textrm{for $1 \leq k \leq \dim V - 1$.}
\end{equation*} 
Applying the Gram--Schmidt process, one may further assume $\{v_1, \dots, v_{\dim V - 1}\}$ is orthonormal, at the expense of a larger implied constant. A rotation $O_V$ with the desired properties is given by stipulating that it maps $v_k^*$ to $v_k$ for $1 \leq k \leq n$. 

Fixing a rotation $O_V$ which satisfies the above property, 
\begin{equation*}
    V' := \big(O_V(V_{\mathrm{sl}}^{\perp}) \cap V_u^{\perp}\big) \times \{0\} = O_V(V_{\mathrm{sl}}^{\perp}) \times \{0\} \cap V_{\xi}^{\perp}.
\end{equation*}
Since $V_u^{\perp} = O_V(\widetilde{V}_{\mathrm{sl}}^{\perp})$, clearly $V' = O_V(V'_{\mathrm{aux}}) \times \{0\}$. In particular, the space $V'$ inherits the dimension bounds from $V'_{\mathrm{aux}}$ and therefore the dimension condition 1) from the lemma is immediately verified. 

It remains to prove the claim. The argument is almost identical to that used to prove Claim 4 in the proof of Lemma 8.7 of \cite{GHI2019}. Nevertheless, here the signature of the phase plays a r\^ole and therefore the details are sketched. 

\begin{proof}[Proof (of Claim)] 

Fixing $v^* \in \widetilde{V}_{\mathrm{sl}} \cap S^{n-2}$, elementary linear geometry considerations reduce the problem to showing 
\begin{equation*}
    |\mathrm{proj}_{V_u^{\perp}} v^*| = O(c_{\mathrm{ex}}).
\end{equation*}
For $\bar{h}$ as in \eqref{convenient notation}, recall that $u \mapsto (u, \bar{h}(u))$ is a graph parametrisation of the surface $\Sigma_{\bar{x}}^{\lambda}$ from \S\ref{basic geometry section} and $u \mapsto G_0^{\lambda}(\bar{x};\Psi(u))$ is the unnormalised Gauss map associated to this parametrisation. It follows that
\begin{equation*}
    S_u = \big\{ u \in U : -\langle\partial_u \bar{h}(u), N_k' \rangle + N_{k,n} = 0 \textrm{ for $1 \leq k \leq n-\dim V$} \big\}.
\end{equation*}
Differentiating the defining equations in the above expression and recalling that $u_0$ is a fixed point featured in the definition of $A_u$, one deduces that a basis for $V_u^{\perp}$ is given by $\{M_1, \dots, M_{n - \dim V}\}$ where
\begin{equation*}
M_k := \partial_{uu}^2 \bar{h}(u_0)N_k' \qquad \textrm{for $1 \leq k \leq n - \dim V$.}
\end{equation*}
Lemma~\ref{reduction lemma} together with some calculus (see \cite[Lemma 4.5]{GHI2019} for a similar computation) imply that 
\begin{equation*}
\|\partial_{uu}^2 \bar{h}(u_0) - \mathrm{I}_{n-1, \sigma}\|_{\mathrm{op}} = O(c_{\mathrm{ex}}).
\end{equation*}
 Since  $\langle v^*, \tilde{N}_k \rangle = 0$ for $1 \leq k \leq n - \dim V$ and $\tilde{N}_k = I_{n-1, \sigma} (N_k')$, it follows that 
\begin{equation}\label{V space claim 4}
|\langle v^*, M_k \rangle| = |\langle v^*, M_k - \tilde{N}_k \rangle| \leq |M_k - \tilde{N}_k| = O(c_{\mathrm{ex}}).
\end{equation}

Let $\mathbf{M}$ be the $(n-1) \times (n - \dim V)$ matrix whose $k$th column is given by the vector $M_k$. The orthogonal projection of $v^*$ onto the subspace $V_u^{\perp}$ can be expressed in terms of $\mathbf{M}$ via the formula  
\begin{equation*}
\mathrm{proj}_{V_u^{\perp}} v^* := \mathbf{M}(\mathbf{M}^{\top}\mathbf{M})^{-1}\mathbf{M}^{\top} v^*.
\end{equation*}
By \eqref{V space claim 4}, the components of the vector $\mathbf{M}^{\top} v^*$ are all $O(c_{\mathrm{ex}})$. Furthermore, it is not difficult to show that $\|\mathbf{M}(\mathbf{M}^{\top}\mathbf{M})^{-1}\|_{\mathrm{op}} \lesssim 1$, and combining these observations establishes the claim.  
\end{proof}




\subsection*{Verifying the transversality condition in  2)} Provided $c_{\mathrm{ex}}$ is chosen to be sufficiently small, the transversality condition holds for the subspace $V'$. To see this, first consider the auxiliary space $V'_{\mathrm{aux}}$. By elementary geometric considerations,
\begin{equation*}
    \min_{\substack{v \in V \cap S^{n-1} \\ v' \in V_{\mathrm{sl}}^{\perp} \times \{0\} \cap S^{n-1}}}\angle(v, v') = \overline{\angle}(V, e_n^{\perp}) \gtrsim 1,
\end{equation*}
where the latter inequality is by \eqref{TE2 1}; this computation is discussed in detail in  \cite[Sublemma 6.6]{Guth2018} and is represented diagrammatically in Figure~\ref{transverse figure}. The above inequality implies that $V$ and $V_{\mathrm{aux}}'$ are quantitatively transverse, since $V_{\mathrm{aux}}'$ is a subspace of $V_{\mathrm{sl}}^{\perp}$. 

It remains to pass from the auxiliary space $V_{\mathrm{aux}}'$ to $V'$.

\begin{figure} 
\centering 
\tdplotsetmaincoords{70}{110}
\begin{tikzpicture}[tdplot_main_coords,font=\sffamily, scale=0.9]
\draw[-latex] (0,0,0) -- (4,0,0);
\draw (-4,0,0) -- (-3,0,0);
\draw[-latex] (0,-4,0) -- (0,4,0);
\draw[-latex] (0,0,0) -- (0,0,4) node[anchor=north east]{$\xi_n$};

\node[align=center] at (0,0.4,0.4) {\Large{$\bm{\theta}$}};

\node[anchor=south west,align=center] at (-0.8,1.2,0) {\Large$U$};

\begin{scope}[canvas is yz plane at x=0]
\draw[very thick,black] (1,0) arc (0:75:1cm) node[midway] (line) {};

   \end{scope}

\draw[very thick, blue](0,-3,0)--(0,3,0);
\node[anchor=south west,align=center] (line) at (3,3,-1) {\Large{\color{blue}$V_{\mathrm{sl}}^{\perp}$}};
\draw[-latex, blue] (line) to[out=180,in=-105] (0.4,1.6,0.05);

\draw[fill=yellow,opacity=0.2] (-3,-3,0) -- (-3,3,0) -- (3,3,0) -- (3,-3,0) -- cycle;
\draw[black, thin] (-3,-3,0) -- (-3,3,0) -- (3,3,0) -- (3,-3,0) -- cycle;
\draw[blue, thick](0,0,0)--(3,0,0);
\draw[blue, thick](-3,0,0)--(-1.9,0,0);
\draw[->,black, very thick] (0,0,0) -- (0,2,0);

\node[anchor=south west,align=center] (line) at (3,3.3,4.5) {\Large${\color{blue}V}$};

\tdplotsetrotatedcoords{-90}{75}{0}
\begin{scope}[tdplot_rotated_coords]
\draw[fill=blue,opacity=0.2] (-3,-3,0) -- (-3,3,0) -- (3,3,0) -- (3,-3,0) -- cycle;
\draw[black, thin] (-3,-3,0) -- (-3,3,0) -- (3,3,0) -- (3,-3,0) -- cycle;
\draw[->,black, very thick] (0,0,0) -- (-2,0,0);
\end{scope}
\end{tikzpicture}
\captionsetup{singlelinecheck=off}
\caption[.]{The transversality condition 
\begin{equation*}
    \theta := \min_{\substack{v \in V \cap S^{n-1} \\ v' \in V_{\mathrm{sl}}^{\perp} \cap S^{n-1}}}\angle(v, v') = \overline{\angle}(V, e_n^{\perp}) \gtrsim 1; \qquad\qquad\qquad\qquad
\end{equation*}
see \cite[Sublemma 6.6]{Guth2018} for a formal proof of this fact.}\label{transverse figure}
\end{figure}




\subsection*{Verifying the transverse equidistribution estimate in 3)} The remaining steps of the proof closely follow the argument used to prove Lemma 8.7 of \cite{GHI2019}. The localisation to $\tau$ implies that the tangent space $A_u$ is a good approximation for the surface $S_u$. In particular, the key observation is that if $(\theta, v) \in \mathbb{T}_{V, B, \tau}$, then 
\begin{equation}\label{Fourier trans eq}
    \mathrm{dist}(\xi_{\theta}, A_{\xi}) \lesssim R^{-1/2 + \delta_m} \qquad \textrm{for $\xi_{\theta} := \Sigma(\omega_{\theta})$.}
\end{equation}
As in \S\ref{wave packet section}, here  $\omega_{\theta} \in B^{n-1}$ denotes the centre of the cap $\theta$ whilst $\Sigma$ is the parametrisation of the smooth hypersurface from \eqref{Sigma def}. 

The inequality \eqref{Fourier trans eq} follows from the proof of Claim 3 in the proof of Lemma 8.7 of \cite{GHI2019}.  Since $V_{\xi}$ is the linear subspace parallel to the affine subspace $A_{\xi}$, the above inequality implies that $\mathrm{proj}_{V_{\xi}^{\perp}}\xi_{\theta}$ lies in some fixed ball of radius $O(R^{-1/2 + \delta_m})$ whenever $(\theta, v) \in \mathbb{T}_{V, B, \tau}$. 

As in \cite{GHI2019} and \cite{Guth2018}, the desired transverse equidistribution estimate 3) follows as a consequence of the  localisation of the $\mathrm{proj}_{V_{\xi}^{\perp}}\xi_{\theta}$ described above. Indeed, since each $\eta_B \cdot T^{\lambda}g_{\theta,v}$ is essentially Fourier supported in a small ball around $\xi_{\theta}$, this implies the projection of the Fourier support of $\eta_B \cdot T^{\lambda}g_{\theta,v}$ onto $V_{\xi}^{\perp}$ is also localised to a $O(R^{-1/2 + \delta_m})$-ball. The transverse equidistribution property now follows as a manifestation of the uncertainty principle (see, in particular, \cite[Lemma 8.5]{GHI2019}). The reader is referred to \cite{GHI2019} for the full details.


\end{proof}




\subsection{The proof of the transverse equidistribution estimate} Using ideas from \cite{Guth2018, GHI2019}, one may easily pass from Lemma~\ref{trans eq lem 2} to Lemma~\ref{trans eq lem 1}. Much of the proof is essentially identical to the proof of \cite[Lemma 8.4]{GHI2019} therefore only a sketch of the argument is provided.

It suffices to prove Lemma~\ref{trans eq lem 1} in the case $1 \leq m = \dim Z \leq (n+\sigma+1)/2$, as otherwise $\mu(n,\sigma,m)=0$ and the statement is a simple consequence of H\"ormander's classical $L^2$ bound (see the discussion around \eqref{trans eq 4} below). 

 Consider $Z, B, \tau$ and $g$ as in the statement of Lemma~\ref{trans eq lem 1}. It may be assumed that $g$ is concentrated on those wave packets $(\theta,v)$ from $\mathbb{T}_{Z,B,\tau}$ for which $T_{\theta,v}$ intersects $N_{R^{1/2+\delta_m}}(Z)\cap B$, as for all other $(\theta,v)$ the function $|T^\lambda_{\theta,v}g|$ is very small on $N_{\rho^{1/2+\delta_m}}(Z)\cap B$. By the $R^{1/2+\delta_m}$-tangent condition, it follows that there exists $z\in Z\cap 2B$ such that
 \begin{equation*}
    \angle(G^\lambda(\bar{x},\theta),T_zZ) \lesssim R^{-1/2 + \delta_m}
\end{equation*}
for all such $(\theta,v)$. Therefore, there exists a subspace $V\subset \mathbb{R}^n$ of minimal dimension $\dim V\leq \dim Z$ such that
\begin{equation*}
    \angle(G^\lambda(\bar{x},\theta),V)\lesssim R^{-1/2 + \delta_m}
\end{equation*}
for all wave packets $(\theta,v)$ upon which $g$ is concentrated. This implies that $g$ is concentrated on wave packets $\mathbb{T}_{V,B,\tau}$, as defined in \S\ref{tang wp sec}. By Lemma~\ref{trans eq lem 2} there exists a linear subspace $V'\subseteq \mathbb{R}^n$ satisfying
\begin{equation}\label{dimension bounds}
    \mu(n,\sigma,\dim V)\leq\dim V'\leq n-\dim V,
\end{equation}
\begin{equation*}
    \angle(v,v')\geq 2c_{\rm trans} \qquad \text{for all non-zero vectors }v\in V \text{ and } v'\in V'
\end{equation*}
and the transverse equidistribution estimate
\begin{equation} \label{trans eq 1}
    \int\displaylimits_{\Pi' \cap B(x_0, \rho^{1/2 + \delta_m})}\!\!\!\!\!\!\!\!\!\!\!\! |T^{\lambda}g|^2 \lesssim_{\delta} R^{O(\delta_m)} (\rho/R)^{\dim V'/2} \|g\|_{L^2(B^{n-1})}^{2\delta/(1+\delta)} \big(\int\displaylimits_{\Pi' \cap 2B} |T^{\lambda}g|^2 \big)^{1/(1+\delta)}
\end{equation}
for every affine subspace $\Pi'$ parallel to $V'$ and $x_0\in B$. 

In contrast to the positive-definite case in \cite{GHI2019}, where one may ensure that $\dim V+\dim V'=n$, only the generally weaker dimension bounds \eqref{dimension bounds} hold here. However, the subspace $\widetilde{V}:=V'\oplus (V+V')^\perp$  satisfies $\dim V+\dim \widetilde{V}=n$ and the quantitative transversality condition
\begin{equation*}
    \angle(v,\widetilde{v})\geq 2c_{\rm trans} \qquad \text{for all non-zero vectors }v\in V \text{ and } \widetilde{v}\in\widetilde{V},
\end{equation*}
as well the transverse equidistribution estimate
\begin{eqnarray} \label{trans eq 2}
\begin{aligned}
    \int\displaylimits_{\Pi \cap B(x_0, \rho^{1/2 + \delta_m})}\!\!\!\!\!\!\!\!\!\!\!\! |T^{\lambda}g|^2 \lesssim_{\delta} R^{O(\delta_m)} (\rho/R)^{\dim V'/2} \|g\|_{L^2(B^{n-1})}^{2\delta/(1+\delta)} \big(\int\displaylimits_{\Pi \cap 2B} |T^{\lambda}g|^2 \big)^{1/(1+\delta)}
\end{aligned}
\end{eqnarray}
for every affine subspace $\Pi$ parallel to $\widetilde{V}$ and $x_0\in \Pi\cap B$, which follows from \eqref{trans eq 1} by Fubini and H\"older's inequality (as well as the fact that $\delta\ll\delta_m)$. Following closely the proof of Lemma 8.4 in \cite{GHI2019}, one may further prove that for each $z\in Z\cap 2B$ the pair $T_zZ, \widetilde{V}$ satisfies the quantitative transversality condition
\begin{equation*}
    \angle(v,\widetilde{v})\geq c_{\rm trans}
\end{equation*}
for all non-zero vectors $v\in T_zZ\cap (T_zZ\cap \widetilde{V})^\perp$ and $\widetilde{v}\in\widetilde{V}\cap (T_zZ\cap \widetilde{V})^\perp$. Since in addition $\dim T_zZ+\dim \widetilde{V}\geq n$, Lemma 8.13 in \cite{GHI2019} implies that
\begin{equation*}
    \Pi\cap N_{\rho^{1/2+\delta_m}}(Z)\cap B\subseteq N_{C\rho^{1/2+\delta_m}}(\Pi\cap Z)\cap 2B
\end{equation*}
for every plane $\Pi$ parallel to $\widetilde{V}$. As $\Pi\cap Z$ is a complete transverse intersection of dimension $\dim Z+\dim\widetilde{V}-n=m -\dim V$, it follows by Wongkew's theorem \cite{Wongkew1993} that $\Pi\cap N_{\rho^{1/2+\delta_m}}(Z)\cap B$ can be covered by
\begin{equation*}
    O_{\varepsilon}\left(R^{O(\delta_m)}(R/\rho)^{(m-\dim V)/2} \right)
\end{equation*}
balls of radius $\rho^{1/2+\delta_m}$. Applying the estimate \eqref{trans eq 2} in each of these balls and summing, one obtains
\begin{equation*}
    \int\displaylimits_{\Pi \cap N_{\rho^{1/2 + \delta_m}}(Z)\cap B}\!\!\!\!\!\!\!\!\!\!\!\! |T^{\lambda}g|^2 \lesssim_{\varepsilon,\delta} R^{O(\delta_m)} (\rho/R)^{(\dim V+\dim V'-m)/2} \|g\|_{L^2(B^{n-1})}^{2\delta/(1+\delta)} \big(\int\displaylimits_{\Pi \cap 2B} |T^{\lambda}g|^2 \big)^{1/(1+\delta)}
\end{equation*}
for all planes $\Pi$ parallel to $\widetilde{V}$. Integrating over all such planes and applying H\"older's inequality, one deduces that
\begin{equation}\label{trans eq 3}
    \int\displaylimits_{N_{\rho^{1/2 + \delta_m}}(Z)\cap B}\!\!\!\!\!\!\!\!\!\!\!\! |T^{\lambda}g|^2 \lesssim_{\varepsilon,\delta} R^{O(\delta_m)} (\rho/R)^{(\dim V+\dim V'-m)/2} \|g\|_{L^2(B^{n-1})}^{2\delta/(1+\delta)} \big(\int\displaylimits_{2B} |T^{\lambda}g|^2 \big)^{1/(1+\delta)}.
\end{equation}
By H\"ormander's $L^2$ bound \cite{Hormander1973} (see also \cite[Chapter IX]{Stein1993} or \cite[Lemma 5.5]{GHI2019}),
\begin{equation}\label{trans eq 4}
    \big(\int\displaylimits_{2B} |T^{\lambda}g|^2 \big)^{1/(1+\delta)}\lesssim R^{1/2+O(\delta_m)}\big( \int\displaylimits_{B^{n-1}}|g|^2\big)^{1/(1+\delta)}.
\end{equation}
Substituting this into \eqref{trans eq 3}, the desired estimate in Lemma \ref{trans eq lem 1} follows provided 
\begin{equation}\label{trans eq 5}
    \dim V+\dim V'-m\geq \mu(n,\sigma,m).
\end{equation}

It remains to show \eqref{trans eq 5} holds. In view of \eqref{dimension bounds}, this would follow from
\begin{equation*}
    \dim V+\mu(n,\sigma,\dim V)-m\geq \mu(n,\sigma,m).
\end{equation*}
By the initial reduction at the beginning of the subsection, $\dim V\leq m \leq (n+\sigma+1)/2$. If $0 \leq \dim V\leq (n-\sigma+1)/2$, then $\mu(n,\sigma,\dim V)=n-2\dim V+1$ and 
\begin{align*}
\dim V+\mu(n,\sigma,\dim V)-m &= n-m-\dim V+1\\
&\geq \max\left\{n-2m+1,n-m-\frac{n-\sigma+1}{2}+1\right\}\\
&=\mu(n,\sigma,m).
\end{align*}
On the other hand, if $(n-\sigma+1)/2 \leq \dim V\leq (n+\sigma+1)/2$, then $\mu(n,\sigma,\dim V)=(n+\sigma+1)/2-\dim V$ and
\begin{equation*}
\dim V+\mu(n,\sigma,\dim V)-m = \frac{n+\sigma+1}{2}-m=\mu(n,\sigma,m).
\end{equation*}
This concludes the proof of Lemma \ref{trans eq lem 1}.



\section{Proof of Theorem~\ref{k-broad theorem}}\label{broad theorem proof sec}

Theorem~\ref{k-broad theorem} is a special case of the following inductive proposition (in place of Proposition 10.1 from \cite{GHI2019}). Define
\begin{equation*}
    e_{k,n,\sigma}(p):=\frac{1}{2}\left(\frac{1}{2}-\frac{1}{p}\right)\frac{n+1+\sigma+2k}{2}.
\end{equation*}

\begin{proposition} Given $\varepsilon>0$ sufficiently small and $1\leq m\leq n$ there exist 
\begin{equation*}
0<\delta\ll\delta_{n-1}\ll \delta_{n-2} \ll \ldots \ll \delta_1\ll \varepsilon
\end{equation*}
and constants $\bar{C}_\varepsilon$, ${\bar A}_\varepsilon$ dyadic, $D_{m,\varepsilon} \lesssim_{\varepsilon} 1$ and $\vartheta_m < \varepsilon$ such that the following holds.

Suppose $Z=Z(P_1,\ldots,P_{n-m})$ is a transverse complete intersection with $\overline{\deg}\, Z\leq D_{m,\varepsilon}$. For all $0\leq \sigma\leq n-1$, $2\leq k\leq n$, $1\leq A\leq {\bar A}_\varepsilon$ dyadic and $1\leq K\leq R \leq \lambda$ the inequality 
\begin{equation*}
\|T^{\lambda} f\|_{\mathrm{BL}^p_{k,A}(B(0,R))}\lesssim_{\varepsilon} K^{\bar{C}_{\varepsilon}} R^{\vartheta_m + \delta\left(\log {\bar{A}}_\varepsilon - \log A\right)-e_{k,n,\sigma}(p)+1/2}\|f\|_{L^2(B^{n-1})}
\end{equation*}
holds for all translates $T^\lambda$ of H\"ormander-type operators with reduced phase of signature $\sigma$, whenever $f$ is concentrated on wave packets from $\T_Z$ and
\begin{equation*}
2\leq p \leq \bar{p}_0(m,\sigma,k) := \left\{
\begin{array}{ll}
\bar{p}(m,\sigma,k)          & \textrm{if $k < m$} \\
\bar{p}(m,\sigma,m)+\delta   & \textrm{if $k = m$}
\end{array} \right. .
\end{equation*}
\end{proposition}
Here, $\T_Z$ is defined as in $\S$\ref{trans eq sec}; that is,
\begin{equation*}
\T_Z := \{(\theta,v) \in \mathbb{T} :\; T_{\theta,v}\text{ is }R^{-1/2+\delta_m}\text{-tangent to } Z\text{ in }B(0,R)\}
\end{equation*}
and the parameters $D_{m,\varepsilon}$, $\theta_m$, $\bar{A}_\varepsilon$, $\delta,\delta_1,\ldots,\delta_{n-1}$, as well as translates of H\"ormander-type operators, are defined as in \cite{GHI2019}.

\begin{proof}
The proof is the same as that of Proposition 10.1 in \cite{GHI2019}, with the exception that the exponent $n-m$ in inequality (10.30) of \cite{GHI2019}, which is due to equidistribution under a positive definite assumption on the phase, is here replaced by $\mu(n,\sigma,m)$, the exponent appearing in the equidistribution Lemma \ref{trans eq lem 1}. This exponent is carried through to the end of the inductive proof, and the induction closes due to the above definition of $e_{k,n,\sigma}(p)$.
\end{proof}




\section{Proof of Theorem~\ref{main theorem}: Narrow decoupling}\label{narrow sec}




\subsection{Overview} It remains to pass from the $k$-broad estimates of Theorem~\ref{k-broad theorem} to linear estimates for the oscillatory integral operators $T^{\lambda}$. As in \cite{Guth2018, GHI2019}, this is achieved via the Bourgain--Guth method from \cite{Bourgain2011}, which recursively partitions the norm $\|T^{\lambda}f\|_{L^p(B_R)}$ into two pieces:\medskip

\noindent \textbf{Broad part.} This is the part of the norm which can be estimated using the $k$-broad inequalities from Theorem~\ref{k-broad theorem}.\medskip

\noindent \textbf{Narrow part.} This consists of the remaining contributions to the norm, which cannot be controlled using the $k$-broad estimates.\medskip

In this section the tools for analysing the narrow part are reviewed. The main ingredient is a Wolff-type $\ell^p$-decoupling inequality: see Proposition~\ref{dec prop} below. In the next section, a sketch of the Bourgain--Guth argument is provided which combines Theorem~\ref{k-broad theorem} and Proposition~\ref{dec prop} (or, more precisely, Corollary~\ref{dec cor}) in order to deduce Theorem~\ref{main theorem}. 




\subsection{Decoupling regions} Let $h \colon B^{n-1} \to \R$ be a smooth function such that 
\begin{equation*}
    h(0) = 0 \qquad \textrm{and} \qquad \partial_u h(0) = 0
\end{equation*}
and such that the Hessian $\partial_{uu}^2 h(u)$ is non-degenerate for all $u \in B^{n-1}$ with fixed signature $0 \leq \sigma \leq n-1$. In such cases $h$ is said to be of signature $\sigma$. Consider the surface
\begin{equation*}
    \Sigma[h] := \big\{ \Gamma_h(u) : u \in B^{n-1} \big\}, \quad \textrm{where } \Gamma_h(u) := \begin{pmatrix}
    u \\
    h(u)
    \end{pmatrix},
\end{equation*}
which is of non-vanishing Gaussian curvature and has second fundamental form of constant signature $\sigma$. Note that the Gauss map $G_h \colon B^{n-1} \to S^{n-1}$ associated to this surface is given by
\begin{equation*}
    G_h(u) := \frac{1}{(1 + |\partial_{u}h(u)|^2)^{1/2}} G_{h,0}(u) \qquad \textrm{where} \qquad G_{h,0}(u) :=
    \begin{pmatrix}
    -\partial_{u}h(u) \\
    1
    \end{pmatrix}.
\end{equation*}
In particular, $G_h(0) = \vec{e}_n$ and the image set $G_h(B^{n-1})$ is contained in a spherical cap in the northern hemisphere, centred around the north pole. 

Given $\bar{u} \in B^{n-1}$ and $\delta > 0$ define the matrices
\begin{equation*}
    [h]_{\bar{u}} := 
    \begin{bmatrix}
        I_{n-1} & 0 \\
        \partial_{u}h(\bar{u})^{\top} & 1 
    \end{bmatrix} \qquad \textrm{and} \qquad [h]_{\bar{u},\delta} := [h]_{\bar{u}} \circ D_{\delta}
\end{equation*}
where $D_{\delta} = \mathrm{diag}[\delta^{1/2}, \dots, \delta^{1/2}, \delta]$ corresponds to an anisotropic (parabolic) scaling of the coordinates. This definition may be partially motivated by considering a quadratic form $Q(u) := \frac{1}{2} \inn{\mathcal{L}u}{u}$ for $\mathcal{L} \colon \R^{n-1} \to \R^{n-1}$ an invertible, self-adjoint linear mapping. By forming the Taylor expansion of $\Gamma_Q$, it follows that
\begin{equation}\label{dec reg 1}
    \Gamma_{Q}(\bar{u} + \delta^{1/2} u) = [Q]_{\bar{u},\delta} \cdot \Gamma_{Q}(u) + \Gamma_{Q}(\bar{u}).
\end{equation}
In particular, the above identity shows that the surface $\Sigma[Q]$ can be diffeomorphically mapped to a $\delta^{1/2}$-cap\footnote{In particular, the set $\Gamma_{Q}\big(B(\bar{u},\delta^{1/2})\big)$.} via an affine transformation of the ambient space. Moreover, the matrix $[Q]_{\bar{u},\delta}$  corresponds to the linear part of this affine transformation. 

\begin{definition}\label{slab def} A \textit{$\delta^{1/2}$-slab} on $\Sigma[h]$ is a set of the form
\begin{equation*}
    \theta(\bar{u}; \delta) := \big\{ \xi \in \hat{\R}^n : \xi = [h]_{\bar{u}, \delta} \cdot \eta + \Gamma_h(\bar{u}) \textrm{ for some $\eta \in [-1,1]^n$}\big\}. 
\end{equation*}
If $\theta = \theta(\bar{u}; \delta)$ is a $\delta^{1/2}$-slab, then $\bar{u}$ is referred to as the \textit{centre} of the slab and in such cases the notation $\bar{u} = u_{\theta}$ is used. It will also be convenient to write $[h]_{\theta}$ for $[h]_{u_{\theta}, \delta}$ whenever $\theta = \theta(u_{\theta}; \delta)$. 
\end{definition}

These regions are defined in view of the scaling considerations discussed above. In particular, in the quadratic case, where $h = Q$ as above, the slabs inherit a scaling structure from \eqref{dec reg 1}, as described in the proof of Lemma~\ref{dec rescaling lemma} below.

\begin{definition} Given $V$ a subspace of $\R^n$, a \textit{$\delta^{1/2}$-slab decomposition on $\Sigma[h]$ along $V$} is a family $\Theta(V, \delta)$ of $\delta^{1/2}$-slabs satisfying:
\begin{enumerate}[i)]
    \item The $\delta^{1/2}$-slabs belonging to $\Theta(V, \delta)$ are finitely-overlapping, and in particular the maximum number of overlapping slabs is bounded by a dimensional constant.
    \item $\angle(G_h(u_{\theta}), V) \leq \delta^{1/2}$ for all $\theta \in \Theta(V, \delta)$.
\end{enumerate}
\end{definition}

As in \S\ref{trans eq sec} (see \eqref{TE2 1}), to avoid degenerate situations it is assumed that
\begin{equation}\label{V angle}
    \overline{\angle}(V, e_n^{\perp}) := \max_{v \in V\cap S^{n-1}} \angle(v, e_n^{\perp}) \gtrsim 1.
\end{equation}
Thus, any maximal $\delta^{1/2}$-slab decomposition $\Theta(V,\delta)$ essentially forms a decomposition of the neighbourhood of the $(d-1)$-dimensional submanifold
\begin{equation*}
  \Sigma[h;V] := \big\{\Gamma_h(u) \in \Sigma[h] : G_h(u) \in V\}
\end{equation*}
of height $\delta$ in the normal direction to $\Sigma[h]$ and of width $\delta^{1/2}$ in the tangential directions to $\Sigma[h]$. 



\subsection{Constant coefficient decoupling: quadratic case} For $n \geq d \geq 2$ and $0 \leq \sigma \leq n-1$ such that $n-1-\sigma$ is even, define the exponents
\begin{align}\label{dec Leb exp}
    p_{\mathrm{dec}}(n,\sigma,d) &:= \left\{\begin{array}{lll}
        \infty &  \textrm{if} & 2 \leq d \leq \frac{n-\sigma +1}{2}\\[3pt]
         2\cdot\frac{2d - n + \sigma +3}{2d - n + \sigma - 1} & \textrm{if} & \frac{n-\sigma + 1}{2} \leq d \leq \frac{n + \sigma +1}{2} \\[3pt]
        2 \cdot \frac{2d - n + 1}{2d - n - 1} & \textrm{if} &  \frac{n + \sigma +1}{2} \leq d \leq n
    \end{array} \right. , \\
\label{decoupling exponent}
    e(n,\sigma,d) &:= \left\{\begin{array}{lll}
        d-1 &  \textrm{if} & 2 \leq d \leq \frac{n-\sigma +1}{2}\\[3pt]
         \frac{d-1}{2} + \frac{n - 1 - \sigma}{4} & \textrm{if} & \frac{n-\sigma + 1}{2} \leq d \leq \frac{n + \sigma +1}{2} \\[3pt]
        \frac{n - 1}{2} & \textrm{if} &  \frac{n + \sigma +1}{2} \leq d \leq n
    \end{array} \right. .
\end{align}
With this and the definitions from the previous subsection, the main decoupling inequality reads as follows.

\begin{proposition}\label{dec prop} Let $2 \leq d \leq n$, $0 \leq \sigma \leq n-1$ with $n-1-\sigma$ even and $\delta >0$. Suppose that $h \colon B^{n-1} \to \R$ is of signature $\sigma$, that $V \subseteq \R^n$ is a vector subspace of dimension $d$ satisfying \eqref{V angle} and $\Theta(V,\delta)$ is $\delta^{1/2}$-slab decomposition on $\Sigma[h]$ along $V$. For all $2 \leq p \leq p_{\mathrm{dec}}(n,\sigma,d)$ and $\varepsilon > 0$, the inequality
\begin{equation*}
    \big\| \sum_{\theta \in \Theta(V,\delta)} F_{\theta} \big\|_{L^p(\R^n)} \lesssim_{\varepsilon, h} \delta^{- e(n,\sigma,d)(1/2 - 1/p)  - \varepsilon}  \Big(\sum_{\theta \in \Theta(V,\delta)}\| F_{\theta} \|_{L^p(\R^n)}^p \Big)^{1/p}
\end{equation*}
holds whenever $(F_{\theta})_{\theta \in \Theta(V,\delta)}$ is a tuple of functions satisfying $\supp \hat{F}_{\theta} \subseteq \theta$ for all $\theta \in \Theta(V,\delta)$. 
\end{proposition}

For the $2 \leq d \leq \frac{n-\sigma +1}{2}$ range the decoupling is elementary, but for the remaining $d$ values Proposition \ref{dec prop} relies on the Bourgain--Demeter decoupling theorem for surfaces of non-vanishing Gaussian curvature \cite{BD2017}.  

In this subsection the proof of Proposition~\ref{dec prop} (or, more precisely, the reduction of this proposition to the main theorem in \cite{BD2017}) is described in the special case where the surface under consideration is quadratic. In particular, here $h := Q$ for some quadratic form
\begin{equation}\label{quad dec 1}
    Q(u) := \frac{1}{2} \inn{\mathcal{L}u}{u} 
\end{equation}
where $\mathcal{L} \colon \R^{n-1} \to \R^{n-1}$ is an invertible, self-adjoint linear mapping of signature $\sigma$. This prototypical case is essentially treated in \cite{BEH} (see also \cite{BD2017}) but, for completeness, the details are given.\medskip

\noindent\textit{Slice geometry.} Fix $Q$ as in \eqref{quad dec 1} and a $d$-dimensional subspace $V$ satisfying \eqref{V angle}. The first step is to understand the basic geometry of $\Sigma[Q;V]$. This is a quadratic surface, associated to some potentially degenerate quadratic form. The key is to determine the possible degree of degeneracy, which depends on the signature $\sigma$ of the original matrix $\mathcal{L}$.

For $Q$ as in \eqref{quad dec 1}, the unnormalised Gauss map $G_{Q,0}$ is an affine function. Thus,  the preimage
\begin{equation}\label{quad dec 1.5}
    A_u := \big\{u \in \R^{n-1} : G_{Q}(u) \in V\big\}
\end{equation}
is an affine subspace of dimension $d-1$ (see also Example \ref{Gauss preimage example} above) and
\begin{equation*}
    \Sigma[Q;V] = \Sigma[Q] \cap A_{\xi}, \qquad \textrm{where $A_{\xi} := A_u \times \R$}.
\end{equation*}
In particular, $\Sigma[Q;V]$ is the graph of the form $Q$ restricted to the subspace $A_u$. Furthermore, if $V_u$ denotes the $(d-1)$-dimensional linear subspace parallel to $A_u$, then $\Sigma[Q;V]$ is the image of the graph of $Q$ over $V_u$ under an invertible affine transformation.\medskip

\noindent\textit{Restrictions of quadratic forms.} Given a linear subspace $U \subseteq \R^{n-1}$ of dimensions $d-1$, consider the restriction $Q|_U$ of the quadratic form $Q$ to $U$, which is a (possibly degenerate) quadratic form on $U$. In particular, there exists a self-adjoint linear map $\mathcal{L}_{U} \colon U \to U$ such that $Q|_U(u) = \frac{1}{2} \inn{\mathcal{L}_{U}(u)}{u}$ for all $u \in U$. For $\rho > 0$ let $\mathbf{N}(\mathcal{L}_U;\rho)$ denote the number of eigenvalues of $\mathcal{L}_U$ inside the interval $(-\rho,\rho)$ and let $\rho(\mathcal{L}) > 0$ denote the minimum modulus of the eigenvalues of $\mathcal{L}$. 

The following lemma is a minor modification of \cite[Lemma 3.3]{BEH}, which in turn is adapted from the proof of Proposition 3.2 in \cite{BD2017}.

\begin{lemma}[\cite{BEH,BD2017}]\label{rest form lem} Let $2 \leq d \leq n$, $0 \leq \sigma \leq n-1$ with $n-1 - \sigma$ even and $\mathcal{L} \colon \R^{n-1} \to \R^{n-1}$ be an invertible, self-adjoint linear mapping of signature $\sigma$. If $U$ is a vector space of dimension $d-1$, then 
\begin{equation*}\label{flat dimensions}
    \mathbf{N}\big(\mathcal{L}_U; \rho(\mathcal{L}) \big) \leq 
    \nu(n,\sigma,d) := 
    \left\{
    \begin{array}{ll}
        d-1 & \textrm{if $1 \leq d \leq \frac{n-\sigma + 1}{2}$} \\[3pt]
        \frac{n-\sigma-1}{2} & \textrm{if $\frac{n-\sigma + 1}{2} \leq d \leq \frac{n+\sigma + 1}{2}$}  \\[3pt]
        n-d &  \textrm{if $\frac{n+\sigma+1}{2} \leq d \leq n$}
    \end{array}\right. ,
\end{equation*}
where $\mathcal{L}_U$ is the linear mapping obtained by restricting to $U$ the quadratic form associated to $\mathcal{L}$, as described above. 
\end{lemma}

Applying Lemma~\ref{rest form lem} to the subspace $U := V_u$, it follows that the slice $\Sigma[Q;V]$ has at least $d-1 - \nu(n,\sigma, d)$ principal curvatures bounded away from zero. 

\begin{proof}[Proof (of Lemma~\ref{rest form lem})] The desired inequality is equivalent to showing
\begin{equation}\label{quad dec 2}
  \mathbf{N}\big(\mathcal{L}_U; \rho(\mathcal{L}) \big) \leq \min\Big\{ d-1, \frac{n-\sigma-1}{2}, n-d\Big\}.
\end{equation}
The bound $d-1$ is obvious, since the total number of eigenvalues cannot exceed the dimension of $U$. 

In order to prove the remaining bounds, form the following orthogonal decompositions of $\R^{n-1}$ and $U$:
\begin{itemize}
    \item Let $X_-$ and $X_+$ denote the subspaces of $\R^{n-1}$ spanned by the eigenvectors of $\mathcal{L}$ with negative and positive eigenvalues, respectively. 
    \item Let $E_-$, $E_0$ and $E_+$ denote the subspaces of $U$ spanned by the eigenvectors of $\mathcal{L}_U$ with eigenvalues lying in the intervals $(-\infty, -\rho(\mathcal{L})]$, $(-\rho(\mathcal{L}), \rho(\mathcal{L}))$ and $[\rho(\mathcal{L}), \infty)$, respectively.
\end{itemize}
In this notation, 
\begin{equation}\label{quad dec 3}
  \frac{n-\sigma-1}{2} = \min\{\dim X_-, \dim X_+\} \quad \textrm{and} \quad  \mathbf{N}\big(\mathcal{L}_U; \rho(\mathcal{L}) \big) = \dim E_0.
\end{equation}

The key observation is that
\begin{equation*}
   \big( E_- \oplus E_0 \big) \cap X_+ = \big( E_+ \oplus E_0 \big) \cap X_- = \{0\},
\end{equation*}
which is a simple consequence of the definitions. Thus,
\begin{subequations}
\begin{align}
\label{quad dec 4a}
    \dim E_- + \dim E_0 &\leq \dim X_-, \\
\label{quad dec 4b}    
    \dim E_+ + \dim E_0 &\leq \dim X_+,
\end{align}
\end{subequations}
and these inequalities together with \eqref{quad dec 3} immediately imply the $\frac{n-\sigma-1}{2}$ bound in \eqref{quad dec 2}. The remaining bound in \eqref{quad dec 2} follows by summing together \eqref{quad dec 4a} and \eqref{quad dec 4b}, using the fact that $\dim E_- + \dim E_+ = d - 1 - \dim E_0$ and $\dim X_- + \dim X_+ = n-1$. 
\end{proof}

\noindent \textit{Trivial decoupling.} Recall, $V_u$ is the $(d-1)$-dimensional linear subspace parallel to the affine subspace $A_u$ defined in \eqref{quad dec 1.5}. Consider the eigenspace decomposition $V_u = E_- \oplus E_0 \oplus E_+$ defined with respect to $\mathcal{L}_{V_u}$ as in the proof of Lemma~\ref{rest form lem}. The eigenvectors generating $E_0$ have eigenvalues of small modulus and therefore correspond to the (relatively) flat directions of $\Sigma[Q;V]$. Note that $E_0$ has dimension $\nu := \mathbf{N}\big(\mathcal{L}_{V_u}; \rho(\mathcal{L}) \big)$, which is bounded by Lemma~\ref{rest form lem}. In these flat directions one applies a trivial decoupling inequality, based on Plancherel's theorem. 

To make the above discussion precise, note that \begin{equation*}
    \supp \hat{F}_{\theta} \subseteq N_{\delta^{1/2}} A_{\xi} \cap B(0,1) \qquad\textrm{for all $\theta \in \Theta(V;\delta)$},
\end{equation*}
where $A_{\xi} = A_u \times \R$ is the affine subspace introduced above. Since $A_u$ is parallel to $V_u$, one may write $A_u = V_u + b$ for some $b \in \R^{n-1}$. Thus, $W_u^{(0)} := E_- \oplus E_+ + b$ is a subspace of $A_u$ and $A_{\xi}$ may be foliated into translates of $W_{\xi}^{(0)} := W_u^{(0)} \times \R$ by writing
\begin{equation*}
    A_{\xi} = \bigcup_{a \in E_0} W_{\xi}^{(a)} \quad \textrm{for $W_{\xi}^{(a)} := W_u^{(a)} \times \R$ and $W_u^{(a)} := W_u^{(0)} + a$.}
\end{equation*}

Let $\mathcal{A}(V;\delta)$ denote a collection of sets $\alpha := N_{\delta^{1/2}} W_{\xi}^{(a)}$ for $a$ varying over a $\delta^{1/2}$-net in $E_0 \cap B(0,1)$, so that $\mathcal{A}(V;\delta)$ forms a cover of the support of the $\hat{F}_{\theta}$ by finitely-overlapping sets. Note that $\# \mathcal{A}(V;\delta) \lesssim \delta^{-\nu/2}$.

Fix a smooth partition of unity $(\zeta_{\alpha})_{\alpha \in \mathcal{A}(V;\delta)}$. Thus, given any $g \in L^1(\R^{n-1})$ with Fourier support in $N_{\delta^{1/2}} A_{\xi} \cap B(0,1)$, one may write $g = \sum_{\alpha \in \mathcal{A}(V;\delta)} g_{\alpha}$ where each $g_{\alpha}$ is defined via the Fourier transform by $\hat{g}_{\alpha} := \hat{g} \cdot \zeta_{\alpha}$. In particular, 
\begin{equation*}
    F := \sum_{\theta \in \Theta(V;\delta)} F_{\theta} \quad \textrm{may be written as} \quad   F = \sum_{\alpha \in \mathcal{A}(V;\delta)} F_{\alpha} = \sum_{\alpha \in \mathcal{A}(V;\delta)}\sum_{\theta \in \Theta(V;\delta)} (F_{\theta})_{\alpha}.
\end{equation*}
For all $2 \leq p \leq \infty$, an elementary argument shows that
\begin{align}
\nonumber
    \big\|\sum_{\theta \in \Theta(V;\delta)} F_{\theta} \big\|_{L^p(\R^n)} &= \big\|\sum_{\alpha \in \mathcal{A}(V;\delta)} F_{\alpha} \big\|_{L^p(\R^n)} \\
    \label{quad dec 5}   
    &\lesssim \delta^{-\nu(1/2-1/p)} \big(\sum_{\alpha \in \mathcal{A}(V;\delta)}\| \sum_{\theta \in \Theta(V;\delta)} (F_{\theta})_{\alpha} \|_{L^p(\R^n)}^p\big)^{1/p}.
\end{align}
Indeed, this follows by interpolation between the $p = 2$ and $p=\infty$ cases (first setting up the estimate in a suitably general formulation, amenable to interpolation), which follow from Plancherel's theorem and the triangle inequality, respectively.\medskip

\noindent\textit{Applying the Bourgain--Demeter theorem.} Now consider the $(d-1-\nu)$-dimensional eigenspace $E :=  E_- \oplus E_+$. The eigenvectors generating $E$ have eigenvalues of large modulus and correspond to `curved' directions. In particular, the restriction of $Q$ to $E$ is a non-degenerate form. Owing to this, one may take advantage of the Bourgain--Demeter theorem \cite{BD2017}. 

Fix $\alpha = N_{\delta^{1/2}}W_{\xi}^{(a)} \in \mathcal{A}(V;\delta)$ and consider the linear subspace 
\begin{equation*}
  V^{(a)} := \langle G_Q(u) : u \in W_u^{(a)} \rangle \subseteq V.   
\end{equation*}
 It is not difficult to show that $V^{(a)}$ is of dimension $d-\nu$ and
\begin{equation*}
    W_u^{(a)} = \big\{ u \in \R^{n-1} : G_Q(u) \in V^{(a)}\big\}. 
\end{equation*}
Choose coordinates $x = (x', x'')$ for $x' \in V^{(a)}$ and $x'' \in (V^{(a)})^{\perp}$. Fix $x'' \in (V^{(a)})^{\perp}$ and define
\begin{equation*}
    g_{x'',\theta}(x') := (F_{\theta})_{\alpha}(x',x'').
\end{equation*}
By elementary properties of the Fourier transform, it follows that
\begin{equation*}
    \supp \hat{g}_{\theta} \subseteq \mathrm{proj}_{V^{(a)}} \theta 
\end{equation*}

Since the eigenvalues associated to eigenvectors in $E$ are bounded away from zero, it follows that $Q$ restricts to a nondenegerate form on $W_u^{(a)}$. Consequently:
\begin{enumerate}[i)]
    \item $\mathrm{proj}_{V^{(a)}} \Sigma[Q;V^{(a)}]$ is a smooth hypersurface in $V^{(a)}$ of non-vanishing Gaussian curvature.
    \item $\mathrm{proj}_{V^{(a)}} \theta$ are finitely-overlapping and appropriate neighbourhoods of the $\mathrm{proj}_{V^{(a)}} \theta$ form a $\delta^{1/2}$-slab decomposition of the entire hypersurface $\mathrm{proj}_{V^{(a)}} \Sigma[Q;V^{(a)}]$.
\end{enumerate}
For a proof of these observations see, for instance, \cite[Lemma 3.4]{BHS}.

In light of the above, for each fixed $x'' \in (V^{(a)})^{\perp}$, the function $g_{x'', \theta}$ satisfies the hypotheses of the decoupling theorem for negatively-curved surfaces from \cite{BD2017}. Thus, for all $2 \leq p \leq 2 \cdot \frac{d - \nu + 1}{d - \nu-1}$ and $\varepsilon > 0$ the inequality 
\begin{equation*}
    \| \sum_{\theta \in \Theta(V;\delta)} g_{x'', \theta} \|_{L^p(V^{(a)})} \lesssim_{\varepsilon,Q} \delta^{-(d-1 - \nu)(1/4 - 1/2p) - \varepsilon} \big(\sum_{\theta \in \Theta(V;\delta)}\| g_{x'', \theta} \|^p_{L^p(V^{(a)})}\big)^{1/p}
\end{equation*}
holds uniformly in $x''$. Taking $p$ powers, integrating over all $x'' \in (V^{(a)})^{\perp}$ and then taking the $p$ roots, one concludes that 
\begin{equation}\label{quad dec 6}
\| \sum_{\theta \in \Theta(V;\delta)} (F_{\theta})_{\alpha} \|_{L^p(\R^n)} \lesssim_{\varepsilon,Q} \delta^{-(d-1 - \nu)(1/4 - 1/2p) - \varepsilon} \big(\sum_{\theta \in \Theta(V;\delta)}\|  (F_{\theta})_{\alpha} \|^p_{L^p(\R^n)}\big)^{1/p}.
\end{equation}
This efficiently decouples the $L^p$-norms on the right-hand side of \eqref{quad dec 5}.
\medskip

\noindent\textit{Combining the decouplings.} Finally, fix $\theta \in \Theta(V;\delta)$ and observe that 
\begin{equation}\label{quad dec 7}   
    \big(\sum_{\alpha \in \mathcal{A}(V;\delta)}\|  (F_{\theta})_{\alpha} \|^p_{L^p(\R^n)}\big)^{1/p} \leq \Big\| \big(\sum_{\alpha \in \mathcal{A}(V;\delta)} |(F_{\theta})_{\alpha}|^2 \big)^{1/2}  \Big\|_{L^p(\R^n)} \lesssim \|F_{\theta}\|_{L^p(\R^n)}
\end{equation}
by an elementary square function estimate (see, for instance, \cite[Lemma 2.4.6]{Sogge2017}). Combining \eqref{quad dec 5}, \eqref{quad dec 6} and \eqref{quad dec 7}, for all $2 \leq p \leq 2 \cdot \frac{d - \nu + 1}{d - \nu-1}$ and $\varepsilon > 0$ the inequality 
\begin{equation}\label{quad dec 8}
    \big\| \sum_{\theta \in \Theta(V,\delta)} F_{\theta} \big\|_{L^p(\R^n)} \lesssim_{\varepsilon, Q} \delta^{- (d-1 + \nu)(1/4 - 1/2p)  - \varepsilon}  \Big(\sum_{\theta \in \Theta(V,\delta)}\| F_{\theta} \|_{L^p(\R^n)}^p \Big)^{1/p}
\end{equation}
holds. By Lemma~\ref{rest form lem}, the $\delta$ dependence in \eqref{quad dec 8} is at least as good as that in Proposition~\ref{dec prop}. However, Lemma~\ref{rest form lem}, together with the definitions \eqref{flat dimensions}, \eqref{dec Leb exp} and \eqref{decoupling exponent}, also implies that 
\begin{equation}\label{quad dec 9}
  2 \cdot \frac{d - \nu + 1}{d - \nu-1} \leq p_{\mathrm{dec}}(n,\sigma,d),
\end{equation}
and so the range of $p$ in \eqref{quad dec 8} is potentially insufficient for the present purpose. To remedy this, one may interpolate against the trivial inequality
\begin{equation}\label{quad dec 10}
   \big\|\sum_{\theta \in \Theta(V;\delta)} F_{\theta} \big\|_{L^{\infty}(\R^n)} 
\lesssim \delta^{-(d-1)/2}  \max_{\theta \in \Theta(V;\delta)}\|F_{\theta} \big\|_{L^{\infty}(\R^n)}.
\end{equation}
Indeed, the desired decoupling inequality in Proposition~\ref{dec prop} follows by interpolating between \eqref{quad dec 8} and \eqref{quad dec 10}, in view of the exponent relation \eqref{quad dec 9}.




\subsection{Constant coefficient decoupling: general case} To complete the proof of Proposition~\ref{dec prop}, it remains to extend the result from quadratic surfaces to graphs of arbitrary smooth $h$ of signature $\sigma$. This is achieved via a now standard iteration argument originating in the work of Pramanik--Seeger \cite{PS2007}. The argument relies on the fact that, locally, each such $h$ is a small perturbation of a quadratic surface of the same signature, and also on special scaling properties of the decoupling inequalities which manifest in the proof of Lemma~\ref{dec rescaling lemma} below.

Consider the slight generalisation of the setup from the previous subsection where $h \colon \R^{n-1} \to \R$ is a quadratic of signature $\sigma$ defined by
\begin{equation}\label{scale dec 1}
    h(u) := \frac{1}{2} \inn{\mathcal{L}u}{u} + \inn{\vec{b}}{u} + a
\end{equation}
where $\mathcal{L} \colon \R^{n-1} \to \R^{n-1}$ is an invertible, self-adjoint linear mapping of signature $\sigma$, whilst $\vec{b} \in \R^{n-1}$ and $a \in \R$. Fix $V$ a $d$-dimensional subspace satisfying \eqref{V angle}, a pair of scales $0 < \delta < \rho < 1$ and a $\rho^{1/2}$-slab $\alpha$ on $\Sigma[h]$ with $\angle(G(u_{\alpha}), V) \leq \rho^{1/2}$.

\begin{lemma}\label{dec rescaling lemma} With the above setup, suppose $\Theta(\alpha) \subseteq \Theta(V;\delta)$ is a collection of $\delta^{1/2}$-slabs $\theta$ satisfying $\theta \subseteq \alpha$. For all $2 \leq p \leq p_{\mathrm{dec}}(n,\sigma,d)$ and $\varepsilon > 0$, the inequality 
\begin{equation}\label{scale dec 2}
    \big\| \sum_{\theta \in \Theta(\alpha)} F_{\theta} \big\|_{L^p(\R^n)} \lesssim_{\varepsilon, h} (\delta/\rho)^{-e(n,\sigma,d)(1/2 - 1/p)  - \varepsilon}  \Big(\sum_{\theta \in \Theta(\alpha)}\| F_{\theta} \|_{L^p(\R^n)}^p \Big)^{1/p}
\end{equation}
holds whenever $(F_{\theta})_{\theta \in \Theta(\alpha)}$ is a tuple of functions satisfying $\supp \hat{F}_{\theta} \subseteq \theta$ for all $\theta \in \Theta(\alpha)$.
\end{lemma} 

\begin{proof} Define functions $\tilde{F}_{\theta}$ via the Fourier transform by
\begin{equation*}
\big(\tilde{F}_{\theta}\big)\;\widehat{}\;(\xi) := \hat{F}_{\theta}\big([h]_{\alpha} \cdot \xi + \Gamma_h(u_{\alpha})\big).
\end{equation*}
and note that it suffices to prove the same inequality but with each $F_{\theta}$ replaced with $\tilde{F}_{\theta}$. Indeed, this follows by applying an affine rescaling and modulation to the functions appearing in both sides of the inequality in \eqref{scale dec 2}. 

By the Fourier support hypothesis on the $F_{\theta}$, it follows that each $\tilde{F}_{\theta}$ has Fourier support in the set
\begin{equation*}
\big\{ \xi \in \hat{\R}^n : \xi= [h]_{\alpha}^{-1}\circ [h]_{\theta} \cdot \eta + [h]_{\alpha}^{-1} \big(\Gamma_h(u_{\theta}) - \Gamma_h(u_{\alpha})\big) \textrm{ for some $\eta \in [-1,1]^n$}\big\}.   
\end{equation*}
Defining $u_{\tilde{\theta}} := \rho^{-1/2}\big(u_{\theta} - u_{\alpha}\big)$ and $\tilde{\theta} := \theta(u_{\tilde{\theta}}; \delta/\rho)$, a simple computation shows that
\begin{equation*}
    [h]_{\alpha}^{-1}\circ [h]_{\theta} = [Q]_{\tilde{\theta}} \qquad \textrm{and} \qquad  [h]_{\alpha}^{-1} \big(\Gamma_h(u_{\theta}) - \Gamma_h(u_{\alpha})\big) = \Gamma_{Q}(u_{\tilde{\theta}}),
\end{equation*}
where $Q$ is the leading homogeneous part of $h$, as defined in \eqref{quad dec 1}. In particular,
\begin{equation*}
    \supp \big(\tilde{F}_{\theta}\big)\;\widehat{}\; \subseteq \tilde{\theta} = \theta(u_{\tilde{\theta}}; \delta/\rho).
\end{equation*}

Let $(N_j)_{j=1}^{n-d}$ be an orthonormal basis for $V^{\perp}$ and write $N_j = (N_j', N_{j,n})$ where $N_j' \in \R^{n-1}$ is the vector formed by the first $n-1$ components of $N_j$. The angle condition \eqref{V angle} implies that the vectors $N_j'$ are quantitatively transverse in the sense that $|\bigwedge_{j=1}^{n-1} N_j'|\gtrsim_h 1$. Define
\begin{equation*}
    \tilde{N}_j := \rho^{1/2}[h]_{\alpha}^{-1} N_j
\end{equation*}
so that $\tilde{N}_j = (N_j', \tilde{N}_{j,n})$ where $\tilde{N}_{j,n} := \rho^{-1/2}\inn{G_{h,0}(u_{\alpha})}{N_j}$. Recall, by hypothesis, $\angle(G_h(u_{\alpha}), V) \leq \rho^{1/2}$ and therefore the vectors $\tilde{N}_j$ have magnitude $O_h(1)$. The vectors $\tilde{N}_j$ also inherit quantitative transversality from the $N_j'$.

Consider the $d$-dimensional subspace $\tilde{V} := \langle \tilde{N}_1, \dots, \tilde{N}_{n-d} \rangle^{\perp}$. A simple computation shows that
\begin{equation*}
    \inn{G_{Q,0}(u_{\tilde{\theta}})}{\tilde{N}_j} = \rho^{-1/2}\inn{G_{h,0}(u_{\theta})}{N_j} .
\end{equation*}
The condition $\angle(G_h(u_{\theta}), V) \leq \delta^{1/2}$ implies $|\inn{G_{h,0}(u_{\theta})}{N_j}| \lesssim_h \delta^{1/2}$ for $1 \leq j \leq n-d$ and, consequently, $\angle(G_{Q}(u_{\tilde{\theta}}), \tilde{V}) \lesssim_h (\delta/\rho)^{1/2}$. Thus, the claim follows by applying the decoupling inequality from the previous step to the function $Q$ at scale $\sim \delta/\rho$. 
\end{proof}

Following \cite{PS2007}, the general case of Proposition~\ref{dec prop} may be deduced from the quadratic case via an induction-on-scale procedure, using Lemma~\ref{dec rescaling lemma}. 

\begin{proof}[Proof (of Proposition~\ref{dec prop}: general case)] Fix $h \colon B^{n-1} \to \R$ of signature $\sigma$, a vector subspace  $V \subseteq \R^n$ of dimension $d$ satisfying \eqref{V angle} and a Lebesgue exponent $2 \leq p \leq \infty$. For $0 < \delta < 1$ define the \textit{decoupling constant} $\mathfrak{D}_{h,V,p}(\delta)$ to be the infimum over all constants $C \geq 1$ for which the inequality 
\begin{equation*}
    \big\| \sum_{\theta \in \Theta(V,\delta)} F_{\theta} \big\|_{L^p(\R^n)} \leq C\delta^{-e(n,\sigma,d)(1/2 - 1/p) }  \Big(\sum_{\theta \in \Theta(V,\delta)}\| F_{\theta} \|_{L^p(\R^n)}^p \Big)^{1/p}
\end{equation*}
holds for all $\delta^{1/2}-$slab decompositions $\Theta(V,\delta)$ on $\Sigma[h]$ along $V$ and all tuples of functions $(F_{\theta})_{\theta \in \Theta(V,\delta)}$ satisfying $\supp \hat{F}_{\theta} \subseteq \theta$ for all $\theta \in \Theta(V,\delta)$. With this notation, given $\varepsilon > 0$ the problem is to show that
\begin{equation}\label{cc dec 1}
    \fD_{h,V,p}(\delta) \lesssim_{h,\varepsilon} \delta^{- \varepsilon}.
\end{equation}

Fixing $\varepsilon > 0$, the argument proceeds by induction on the scale $\delta$, using the prototypical cases proved above to facilitate the induction step. In particular, let $\delta_{\circ} = \delta_{\circ}(h,\varepsilon) > 0$ be a fixed small parameter, depending only on $h$ and $\varepsilon$ and chosen sufficiently small for the purpose of the forthcoming argument. If $1 > \delta \geq \delta_{\circ}$, then the desired bound \eqref{cc dec 1} follows immediately from H\"older's inequality. This serves as the base case for the induction.\medskip 

\noindent \textbf{Induction hypothesis:} Fix $0 < \delta < \delta_{\circ}$ and suppose 
\begin{equation}\label{cc dec 2}
    \fD_{h,V,p}(\delta') \leq \mathbf{C}_{h,\varepsilon} (\delta')^{- \varepsilon}
\end{equation}
holds whenever $2\delta \leq \delta' < 1$.\medskip

Here $\mathbf{C}_{h,\varepsilon}$ is a fixed constant, which depends only on the admissible objects $h$ and $\varepsilon$, chosen sufficiently large for the purpose of the forthcoming argument. In particular, it suffices to take $\mathbf{C}_{h, \varepsilon}$ so that \eqref{cc dec 2} holds in the base case $1 > \delta' \geq \delta_{\circ}$ for the choice of $\delta_{\circ}$ determined below. 

Fix $\Theta(V,\delta)$ a $\delta^{1/2}$-slab decomposition on $\Sigma[h]$ along $V$. Let $\delta \ll \rho < 1$ be a second small parameter. Later in the argument $\rho$ is fixed by taking $\rho \sim_h \delta^{2/3}$, but for now it is helpful to keep it a free parameter. Fix a $\rho^{1/2}$-slab decomposition $\Theta(V,\rho)$ with the property that every $\theta\in \Theta(V,\delta)$ lies in at least one $\alpha \in \Theta(V,\rho)$.

Given a tuple of functions $(F_{\theta})_{\theta \in \Theta(V,\delta)}$ as in the statement of the proposition, form a tuple of functions $(F_{\alpha})_{\alpha \in \Theta(V,\rho)}$ by partitioning the collection $\Theta(V,\delta)$ into disjoint families $\Theta(\alpha)$ with $\theta \subseteq \alpha$ for all $\theta \in \Theta(\alpha)$ and taking
\begin{equation*}
 F_{\alpha} := \sum_{\theta \in \Theta(\alpha)} F_{\theta} \qquad \textrm{for all $\alpha \in \Theta(V,\rho)$.}
\end{equation*}
Clearly, $\supp \hat{F}_{\alpha} \subseteq \alpha$ and so, applying the induction hypothesis \eqref{cc dec 2} with $\delta' = \rho \geq 2 \delta$, one deduces that
\begin{align}\nonumber
      \big\| \sum_{\theta \in \Theta(V,\delta)} F_{\theta} \big\|_{L^p(\R^n)} &= \big\| \sum_{\alpha \in \Theta(V,\rho)} F_{\alpha} \big\|_{L^p(\R^n)} \\
      \label{cc dec 3}
      &\leq \mathbf{C}_{h,\varepsilon} \rho^{-e(n,\sigma,d)(1/2 - 1/p)  - \varepsilon}  \Big(\sum_{\alpha \in \Theta(V,\rho)}\| F_{\alpha} \|_{L^p(\R^n)}^p \Big)^{1/p}.
\end{align}
    
Fixing $\alpha \in \Theta(V,\rho)$, the problem is now to decouple the norm
\begin{equation*}
    \| F_{\alpha} \|_{L^p(\R^n)} = \big\| \sum_{\theta \in \Theta(\alpha)} F_{\theta} \big\|_{L^p(\R^n)}.
\end{equation*}
To achieve this, $h$ is locally approximated by a quadratic which facilitates application of the decoupling for quadratic surfaces derived in the previous steps. Let $u_{\alpha} \in B^{n-1}$ denote the centre of $\alpha$ and consider the second order approximation $h_{\alpha} \colon \R^{n-1} \to \R$ to $h$ around $u_{\alpha}$, given by 
\begin{equation*}
    h_{\alpha}(u) := \frac{1}{2} \inn{\partial_{uu}^2h(u_{\alpha})(u - u_{\alpha})}{u - u_{\alpha}} + \inn{\partial_{u}h(u_{\alpha})}{u - u_{\alpha}} + h(u_{\alpha}).
\end{equation*}
Note that each of the mappings $h_{\alpha}$ is of the form \eqref{scale dec 1}.

Let $\xi \in \theta = \theta(u_{\theta},\delta) \in \Theta(\alpha)$ so that there exists $\eta = (\eta',\eta_n) \in [-1,1]^{n-1} \times [-1,1]$ such that
\begin{equation*}
    \xi = [h]_{\theta} \cdot \eta + \Gamma_h(u_{\theta}).
\end{equation*}
A simple computation shows 
\begin{equation*}
    \xi = [h_{\alpha}]_{\theta} \cdot \tilde{\eta} + \Gamma_{h_{\alpha}}(u_{\theta}),
\end{equation*}
where $\tilde{\eta} = (\eta',\tilde{\eta}_n)$ for
\begin{equation*}
    \tilde{\eta}_n := \eta_n + \delta^{-1/2} \inn{\partial_{u}h(u_{\theta}) - \partial_{u}h_{\alpha}(u_{\theta})}{\eta'} + \delta^{-1}\big(h(u_{\theta}) - h_{\alpha}(u_{\theta})\big).
\end{equation*}
By Taylor's theorem and the hypothesis $\theta \in \Theta(\alpha)$, one deduces that
\begin{align*}
    |h(u_{\theta}) - h_{\alpha}(u_{\theta})| &\lesssim_h |u_{\theta} - u_{\alpha}|^3 \leq \rho^{3/2}, \\
    |\partial_{u}h(u_{\theta}) - \partial_{u}h_{\alpha}(u_{\theta})| &\lesssim_h |u_{\theta} - u_{\alpha}|^2 \leq \rho.
\end{align*}
Thus, taking $\rho := c_h \delta^{2/3}$ for a suitably small constant $c_h > 0$, depending only on the magnitude of the third order derivatives of $h$, one concludes that $\tilde{\eta}_n \in [-2,2]$. 

The previous observations show that each $F_{\theta}$ for $\theta \in \Theta(\alpha)$ has Fourier support in a $(2\delta)^{1/2}$-slab \textit{defined with respect to the quadratic surface $\Sigma[h_{\alpha}]$}. One may therefore apply the (rescaled version of the) decoupling inequality \eqref{scale dec 2} to conclude that
\begin{equation}\label{cc dec 4}
    \| F_{\alpha} \|_{L^p(\R^n)} \lesssim_{h,\varepsilon} (\delta/\rho)^{-e(n,\sigma,d)(1/2 - 1/p)  - \varepsilon/2} \Big(\sum_{\theta \in \Theta(\alpha)} \|  F_{\theta} \big\|_{L^p(\R^n)}\Big)^{1/p}.
\end{equation}
Combining \eqref{cc dec 3} and \eqref{cc dec 4} with the definition of the decoupling constant, 
\begin{equation*}
    \fD_{h,V,\varepsilon}(\delta) \leq C_{h,\varepsilon}  (\delta/\rho)^{\varepsilon/2} \mathbf{C}_{h,\varepsilon}\delta^{- \varepsilon},
\end{equation*}
where $C_{h,\varepsilon}$ is an amalgamation of the implicit constants arising in the above argument. Since, $\delta/\rho  \leq c_h^{-1} \delta_{\circ}^{1/3}$, by choosing $\delta_{\circ}$ from the outset to be sufficiently small, depending only on $h$ and $\varepsilon$, one may ensure that $C_{h,\varepsilon}(\delta/\rho)^{\varepsilon/2} \leq 1$ and so the induction closes. 
\end{proof}




\subsection{Variable coefficient decoupling}
Proposition~\ref{dec prop} can be used to study H\"ormander-type operators, provided that the operator is sufficiently localised. In particular, given a reduced phase function $\phi^{\lambda}$ of signature $\sigma$, recall from \S\ref{basic geometry section} that for a fixed vector $\bar{x} \in B(0,\lambda)$ in the spatial domain, 
\begin{equation*}
    h_{\bar{x}}^{\lambda}(u) := \partial_{x_n}\phi^{\lambda}\big(\bar{x}; \Psi^{\lambda}_{\bar{x}}(u)\big) 
\end{equation*}
is a smooth function of signature $\sigma$ on its domain. Moreover, if the corresponding operator $T^{\lambda}f$ is localised to a small ball around $\bar{x}$, then the Fourier transform of this localised function is supported in a neighbourhood of the surface $\Sigma[h_{\bar{x}}^{\lambda}]$. This facilitates application of the decoupling inequality from Proposition~\ref{dec prop} in this setting. 

To make the above discussion precise, fix a H\"ormander-type operator $T^{\lambda}$ and a function $f \in L^1(B^{n-1})$. Let $\mathcal{T}$ be a decomposition of the domain $B^{n-1}$ into finitely-overlapping balls $\tau \subseteq \R^{n-1}$ of radius $K^{-1}$, each with some centre $\omega_{\tau} \in B^{n-1}$, and fix a smooth partition of unity $\{\psi_{\tau}\}_{\tau \in \mathcal{T}}$ subordinate to $\mathcal{T}$. Correspondingly, decompose $f = \sum_{\tau \in \mathcal{T}} f_{\tau}$ where each $f_{\tau} = f \cdot \psi_{\tau}$; in particular, each $f_{\tau}$ satisfies $\supp f_{\tau} \subseteq \tau$.  

Thus, $T^{\lambda}f = \sum_{\tau \in \mathcal{T}} T^{\lambda}f_{\tau}$ and one is interested in studying this function localised to some ball $B_{K^2} = B(\bar{x}, K^2)$ of radius $K^2$. In view of this, let $\zeta \in C^{\infty}_c(\R^n)$ satisfy $\zeta(x) = 1$ for $x \in [-1,1]^n$ and $\zeta(x) = 0$ for $x \notin [-2,2]^n$ and define $\zeta_{B_{K^2}}(x) := \zeta(K^{-2}(x - \bar{x}))$. Let $T^{\lambda}_{B_{K^2}}$ denote the localised operator given by replacing the amplitude function $a^{\lambda}(x;\omega)$ in $T^{\lambda}$ with $a^{\lambda}(x;\omega)\cdot \zeta_{B_{K^2}}(x)$. The key observation is that, provided $K^2 \leq \lambda$, each function $T^{\lambda}_{B_{K^2}}f_{\tau}$ is essentially Fourier supported in a $K^{-1}$-slab, defined with respect to the function $h^{\lambda}_{\bar{x}}$. In particular, given $\varepsilon > 0$, for each $\tau$ associate a $K^{-(1-\varepsilon)}$-slab
\begin{equation}\label{vc slabs}
    \theta(\tau) := \theta\big(u_{\tau}; K^{-2(1-\varepsilon)}\big)
\end{equation}
defined as in Definition~\ref{slab def}, taking $h = h^{\lambda}_{\bar{x}}$ and $\bar{u} = u_{\tau} := \big(\Psi_{\bar{x}}^{\lambda}\big)^{-1}(\omega_{\tau})$. Let $\zeta_{\tau}$ denote the function obtained by precomposing $\zeta$ with the inverse of the affine transformation $\eta \mapsto \frac{1}{2}\cdot [h^{\lambda}_{\bar{x}}]_{\theta(\tau)} \cdot \eta + \Gamma(h^{\lambda}_{\bar{x}})(u_{\tau})$. Thus,  $\supp \zeta_{\tau} \subseteq \theta(\tau)$ for all $\xi \in \theta(\tau)$.

\begin{lemma}\label{freq loc lemma}  Given $\varepsilon > 0$ and $R^{\varepsilon} \lesssim_{\varepsilon} K^2 \leq \lambda$, with the above definitions,
\begin{equation}\label{freq loc eq}
    T^{\lambda}_{B_{K^2}}f_{\tau}(x) = \big(T^{\lambda}_{B_{K^2}}f_{\tau}\big) \ast \check{\zeta}_{\tau}(x) + \mathrm{RapDec}(R)(1+|x-\bar{x}|)^{-(n+1)}\|f_{\tau}\|_{L^2(B^{n-1})}.
\end{equation}
\end{lemma}

Once this lemma is established, one may immediately apply Proposition~\ref{dec prop} to deduce the following (pseudo) variable coefficient decoupling inequality. 

\begin{corollary}\label{dec cor} Let $2 \leq d \leq n$, $0 \leq \sigma \leq n-1$ with $n-1-\sigma$ even and $\lambda \geq 1$. Suppose $T^{\lambda}$ is a H\"ormander-type operator with reduced phase of signature $\sigma$ and $V \subseteq \R^n$ is a $d$-dimensional linear subspace. For $2 \leq p \leq p_{\mathrm{dec}}(n,\sigma,d)$ and $\varepsilon > 0$ one has
\begin{align*}
\big\| \sum_{\tau \in V} T^{\lambda} g_{\tau} \big\|_{L^p(B_{K^2})} &\lesssim_{\varepsilon} K^{2e(n,\sigma,d)(1/2 - 1/p) + \varepsilon} \big(\sum_{\tau \in V} \| T^{\lambda} g_{\tau}\|_{L^p(2 \cdot B_{K^2})}^p \big)^{1/p} \\
& \qquad + \mathrm{RapDec}(R)\|f\|_{L^2(B^{n-1})}
\end{align*}
whenever $R^{\varepsilon} \lesssim_{\varepsilon} K^2 \leq \lambda$. Here the sums are over all caps $\tau$ for which $\angle(G^{\lambda}(\bar{x},\tau), V) \leq K^{-1}$ where $\bar{x}$ is the centre of $B_{K^2}$.
\end{corollary}

\begin{proof} Defining the slabs $\theta(\tau)$ with $\varepsilon$ replaced with $\varepsilon':= \varepsilon/100n$ in \eqref{vc slabs}, the functions 
\begin{equation*}
   F_{\theta(\tau)} := \big(T^{\lambda}_{B_{K^2}}f_{\tau}\big) \ast \check{\zeta}_{\tau} 
\end{equation*}
 satisfy $\supp \hat{F}_{\theta(\tau)} \subseteq \theta(\tau)$. Recalling the discussion in \S\ref{basic geometry section}, the collection of slabs $\{\theta(\tau) : \tau \in V\}$ forms a  $K^{-(1-\varepsilon')}$-slab decomposition on $\Sigma_{\bar{x}}^{\lambda} = \Sigma[h_{\bar{x}}^{\lambda}]$ along $V$.\footnote{Strictly speaking, this is not quite true since the slabs have overlap depending on $K^{\varepsilon}$. However, since the collection can be partitioned into $O(K^{\varepsilon/10})$ finitely-overlapping subcollections, this only induces an acceptable $K^{\varepsilon/10}$ loss in the estimates.} Thus, for $p$ in the stated range, one may apply Proposition~\ref{dec prop} to deduce that
 \begin{equation*}
     \big\|\sum_{\tau \in V}F_{\theta(\tau)}\big\|_{L^p(\R^n)} \lesssim_{\varepsilon} K^{2e(n,\sigma,d)(1/2 - 1/p) + \varepsilon} \Big(\sum_{\tau \in V} \|F_{\theta(\tau)}\|_{L^p(\R^n)}^p \Big)^{1/p}.
 \end{equation*}
By transferring the $B_{K^2}$-localisation between the $L^p$-norm and the operator and applying the approximation from Lemma \ref{freq loc lemma}, the desired bound readily follows from the above display.
\end{proof}

It remains to prove the Fourier localisation lemma.

\begin{proof}[Proof (of Lemma~\ref{freq loc lemma})] Taking the Fourier transform, one may write 
\begin{equation*}
    \big(T^{\lambda}_{B_{K^2}}f_{\tau}\big)\;\widehat{}\;(\xi) = e^{-2\pi i \inn{\bar{x}}{\xi}}\int_{\R^{n-1}} B^{\lambda,K}_{\bar{x}}(\xi;\omega) f_{\tau}(\omega)\,\ud \omega
\end{equation*}
where the kernel $B^{\lambda,K}_{\bar{x}}$ satisfies
\begin{equation*}
   \partial_{\xi}^{\beta}B^{\lambda,K}_{\bar{x}}(\xi;\omega) := K^{2(n + |\beta|)} \int_{\R^n} e^{2\pi i \Phi^{\lambda,K}_{\bar{x}}(x;\xi;\omega)} a^{\lambda,K}_{\bar{x},\beta}(x;\omega)\,\ud x \qquad \textrm{for all $\beta \in \N_0^n$}
\end{equation*}
for phase function $\Phi^{\lambda,K}_{\bar{x}}$ and amplitude $a^{\lambda,K}_{\bar{x},\beta}$ given by
\begin{align*}
   \Phi^{\lambda,K}_{\bar{x}}(x;\xi;\omega) &:= \phi^{\lambda}(\bar{x}+K^2x;\omega) - K^2\inn{x}{\xi}, \\
   a^{\lambda,K}_{\bar{x},\beta}(x;\omega) &:= (2\pi i x)^{\beta}\cdot \zeta(x)\cdot a^{\lambda}(\bar{x} + K^2x;\omega).  
\end{align*}
Note that $a^{\lambda,K}_{\bar{x},\beta}$ is supported in $B^n \times B^{n-1}$ and, by the condition $K^2 \leq \lambda$, has derivatives uniformly bounded in $K$ and $\lambda$. On the other hand, 
\begin{equation}\label{freq loc 1}
    \partial_{x}  \Phi^{\lambda,K}_{\bar{x}}(x;\xi;\omega) = K^2\big( \partial_x\phi^{\lambda}(\bar{x};\omega) - \xi\big) + \big(\partial_x\phi^{\lambda}\big(\bar{x}+ K^2x;\omega\big) - \partial_x\phi^{\lambda}(\bar{x};\omega)\big),
\end{equation}
where the second term on the right-hand side is bounded above in magnitude by a constant depending only on the second derivatives of $\phi$. The key claim is that 
\begin{equation}\label{freq loc 2}
  \big\{\xi \in \hat{\R}^n: \big| \xi - \partial_x\phi^{\lambda}(\bar{x};\omega)\big| \leq K^{-2+\varepsilon} \textrm{ for some $\omega \in \tau$}\big\} \subseteq \theta(\tau). 
\end{equation}
Once this is established, one may bound the first term on the right-hand side of \eqref{freq loc 1} under appropriate hypotheses on $\omega$ and $\xi$ and, in particular, show that
\begin{equation*}
    |\partial_{x}  \Phi^{\lambda,K}_{\bar{x}}(x;\xi;\omega)| \gtrsim K^{\varepsilon} \qquad \textrm{for all $\omega \in \tau$ and $\xi \notin \theta(\tau)$.}
\end{equation*}
On the other hand, $|\partial_{x}^{\alpha}  \Phi^{\lambda,K}_{\bar{x}}(x;\xi;\omega)| \lesssim_{\alpha} 1$ for all $\alpha \in \N^n$ with $|\alpha| \geq 2$ and, thus, repeated integration-by-parts yields
\begin{equation*}
 \sup_{\omega \in \tau}\big| \partial_{\xi}^{\beta}\big[\big(1-\zeta_{\tau}(\xi)\big)B^{\lambda,K}_{\bar{x}}(\xi;\omega)\big]\big| \lesssim_{\beta,N} K^{-N} (1 + |\xi|)^{-(n+1)}  \qquad \textrm{for all $\beta \in \N_0^n$, $N \in \N$.}
\end{equation*}
From this it follows that
\begin{equation*}
 \big| \partial_{\xi}^{\beta}\big[e^{2\pi i \inn{\bar{x}}{\xi}}\big(T^{\lambda}_{B_{K^2}}f_{\tau} - T^{\lambda}_{B_{K^2}}f_{\tau} \ast \check{\zeta}_{\tau}\big)\;\widehat{}\;(\xi)\big]\big| \lesssim_{\beta,N} K^{-N} (1 + |\xi|)^{-(n+1)}\|f\|_{L^2(B^{n-1})}
\end{equation*}
 and the desired identity \eqref{freq loc eq} follows by taking inverse Fourier transforms and using repeated integration-by-parts to obtain the desired decay in the spatial variable. 

It remains to prove \eqref{freq loc 2}. Suppose $\xi \in \hat{\R}^n$ and $\omega \in \tau$ satisfy 
\begin{equation}\label{freq loc 3}
    | \xi - \partial_x\phi^{\lambda}(\bar{x};\omega)| \leq K^{-2+\varepsilon} 
\end{equation}
and let $u := \big(\Psi_{\bar{x}}^{\lambda} \big)^{-1}(\omega)$. Since $\big(\Psi_{\bar{x}}^{\lambda}\big)^{-1}$ is a diffeomorphism with bounded Jacobian, the condition $\omega \in \tau$ translates to $|u - u_{\tau}| \lesssim K^{-1}$, whilst \eqref{freq loc 3} implies that
\begin{equation}\label{freq loc 4}
    |\xi' - u| \leq K^{-2 + \varepsilon} \qquad \textrm{and} \qquad |\xi_n - h_{\bar{x}}^{\lambda}(u)| \leq K^{-2 + \varepsilon}
\end{equation}
Let $\eta = (\eta',\eta_n) \in \hat{\R}^n$ be given by 
\begin{equation*}
    \eta' := K^{1-\varepsilon}\big(\xi' - u_{\tau}\big), \qquad
    \eta_n := K^{2(1-\varepsilon)}\big(\xi_n - h^{\lambda}_{\bar{x}}(u_{\tau}) - K^{-1+\varepsilon}  \inn{\partial_uh_{\bar{x}}^{\lambda}(u_{\tau})}{\eta'}\big),
\end{equation*}
so that, in particular, $\eta$ satisfies 
\begin{equation*}
  \xi = [h_{\bar{x}}^{\lambda}]_{\theta(\tau)}\cdot \eta + \Gamma(h_{\bar{x}}^{\lambda})(u_{\tau}).
\end{equation*}
The last step is to show that $\eta \in [-1,1]^n$; indeed, once this is established it follows from the definitions that $\xi \in \theta(\tau)$, as required. It is clear from the earlier discussion that $\eta' \in [-1,1]^{n-1}$ and so matters are further reduced to showing $\eta_n \in [-1, 1]$. By Taylor's theorem,
\begin{align*}
    \xi_n - h_{\bar{x}}^{\lambda}(u) &= \xi_n - h_{\bar{x}}^{\lambda}(u_{\tau}) - \inn{\partial_uh_{\bar{x}}^{\lambda}(u_{\tau})}{u - u_{\tau}} + O(K^{-2}) \\
     &= \xi_n - h_{\bar{x}}^{\lambda}(u_{\tau}) - K^{-1+\varepsilon}\inn{\partial_uh_{\bar{x}}^{\lambda}(u_{\tau})}{\eta'} + O(K^{-2+\varepsilon}),
\end{align*}
where the second identity follows by writing $u - u_{\tau} = u - \xi' + K^{-1+\varepsilon}\eta'$ and using the first inequality in \eqref{freq loc 4}. Provided $K$ is sufficiently large, the result now follows by multiplying through by $K^{2(1-\varepsilon)}$ and applying the second inequality in \eqref{freq loc 4}. 

\end{proof}




\section{Proof of Theorem~\ref{main theorem}: from $k$-broad to linear estimates}\label{broad/narrow sec}




 Theorem~\ref{main theorem} may now be deduced as a consequence of the $k$-broad estimates from Theorem~\ref{k-broad theorem} and the decoupling inequality from Corollary~\ref{dec cor} via the method of \cite{Bourgain2011}. For the exponent $e(n,\sigma,d)$ as defined in \eqref{decoupling exponent}, the key proposition is as follows.

\begin{proposition}\label{proposition Bourgain Guth} Suppose that for all $K \geq 1$ and all $\varepsilon > 0$ any H\"ormander-type operator $T^{\lambda}$ with reduced phase of signature $\sigma$ obeys the $k$-broad inequality
\begin{equation*}
\|T^{\lambda}f\|_{\mathrm{BL}_{k,A}^p(B(0,R))} \lesssim_{\varepsilon} K^{C_{\varepsilon}} R^{\varepsilon} \|f\|_{L^p(B^{n-1})}
\end{equation*}
for some fixed $k, A, p, C_{\varepsilon}$ and all $R \geq 1$. If  
\begin{equation}\label{Bourgain--Guth exponents}
2 \cdot \frac{n - e(n,\sigma,k-1)}{n - 1 - e(n,\sigma,k-1)} \leq p \leq p_{\mathrm{dec}}(n,\sigma,k-1),
\end{equation}
then any H\"ormander-type operator $T^{\lambda}$ with reduced phase of signature $\sigma$ satisfies
\begin{equation*}
\|T^{\lambda}f\|_{L^p(B(0,R))} \lesssim_{\varepsilon} R^{\varepsilon} \|f\|_{L^p(B^{n-1})}. 
\end{equation*}
\end{proposition}

Here $p_{\mathrm{dec}}(n,\sigma,d)$ denotes the decoupling exponent defined in \eqref{dec Leb exp}. 

\begin{remark} In the positive-definite case, $\sigma = n - 1$ and 
\begin{equation*}
   e(n,n-1,k-1) = \frac{k-2}{2}, \qquad p_{\mathrm{dec}}(n,n-1,k-1) =  2 \cdot \frac{k}{k-2}
\end{equation*}

Thus, the condition \eqref{Bourgain--Guth exponents} becomes
\begin{equation*}
    2 \cdot \frac{2n- k + 2}{2n - k} \leq p \leq 2 \cdot \frac{k}{k-2} 
\end{equation*}
This is consistent with \cite[Proposition 9.1]{Guth2018} and \cite[Proposition 11.1]{GHI2019}.\footnote{In the references a more restrictive upper bound of $2 \cdot \frac{k-1}{k-2}$ appears rather than $2 \cdot \frac{k}{k-2}$. This is due to the use of non-endpoint decoupling inequalities in \cite{Guth2018, GHI2019}, which are in fact sufficient for the present purpose.} 
\end{remark}

Theorem~\ref{main theorem} is now a direct consequence of Proposition~\ref{proposition Bourgain Guth} and Theorem~\ref{k-broad theorem}. 

\begin{proof}[Proof (of Theorem~\ref{main theorem})] For each $k$ satisfying the constraint 
\begin{equation*}
2 \cdot \frac{n - e(n,\sigma,k-1)}{n - 1 - e(n,\sigma,k-1)} \leq \bar{p}(n,\sigma, k) 
\end{equation*}
one may apply Proposition~\ref{proposition Bourgain Guth} with $\bar{p}(n,\sigma, k) \leq  p \leq p_{\mathrm{dec}}(n,\sigma,k-1)$ to obtain a (potentially empty) range of estimates for the linear problem. It is not difficult to check that the optimal choice is given by 
\begin{equation*}
 k_* := \left\{ \begin{array}{ll}
    \frac{n+2}{2} & \textrm{for $n$ even} \\[5pt]
    \frac{n+1}{2} & \textrm{for $n$ odd}
 \end{array} \right.   
\end{equation*} 
and one may readily verify that $\bar{p}(n,\sigma, k_*) \leq p_{\mathrm{dec}}(n,\sigma,k_*-1)$. Thus, the linear estimate holds for all $p \geq \bar{p}(n,\sigma, k_*)$. This corresponds to the range of estimates stated in Theorem~\ref{main theorem}. 
 \end{proof}

\begin{proof}[Proof (of Proposition~\ref{proposition Bourgain Guth})] The proof of Proposition~\ref{proposition Bourgain Guth} relies on the induction-on-scales argument originating in \cite{Bourgain2011}. The details are identical to those of the proof of \cite[Proposition 11.2]{GHI2019} except that Corollary~\ref{dec cor} is now used in place of \cite[Theorem 11.5]{GHI2019}, and there are corresponding changes to the numerology. The reader is therefore referred to \cite{GHI2019} (see also \cite{Guth2018}) for the details.
\end{proof}




\bibliography{Reference}
\bibliographystyle{amsplain}

\end{document}